\documentclass[11pt,a4paper,english,reqno,a4paper]{amsart}
\usepackage{amsmath,amssymb,amsthm, comment,graphicx}
\usepackage{tikz}
\usepackage{fancyhdr}
\usepackage{epsfig}
\usepackage{mathrsfs}
\usepackage[colorlinks,linkcolor=blue,anchorcolor=blue,citecolor=blue]{hyperref}
\usepackage{pifont}
\usepackage{amsfonts}
\usepackage{geometry}
\usepackage{graphicx}
\usepackage[compress,sort]{cite}
\usepackage[title]{appendix}

\allowdisplaybreaks
\newcommand{\RNum}[1]{\uppercase\expandafter{\romannumeral #1\relax}}
\numberwithin{equation}{section}
\numberwithin{figure}{section}

\newcounter{En}[section]
\renewcommand{\theEn}{\thesection.\arabic{En}}
\newcommand{\Eqn}[1][]{\refstepcounter{En}(\theEn)}

\newcounter{Cond}[section]
\renewcommand{\theCond}{\textbf{C\arabic{Cond}}}
\newcommand{\Condition}[1][]{\noindent\refstepcounter{Cond}\indent(\theCond)\ }

\newcounter{CondE}[section]
\renewcommand{\theCondE}{\textbf{C$_{\rm E}$\arabic{CondE}}}
\newcommand{\ConditionE}[1][]{\noindent\refstepcounter{CondE}\indent(\theCondE)\ }

\newcounter{CondP}[section]
\renewcommand{\theCondP}{\textbf{C$_{\rm P}$\arabic{CondP}}}
\newcommand{\ConditionP}[1][]{\noindent\refstepcounter{CondP}\indent(\theCondP)\ }

\newcounter{TH}[section]
\renewcommand{\theTH}{\thesection.\arabic{TH}}
\newcommand{\Theorem}[1][]{\noindent\textbf{Theorem}\ \refstepcounter{TH}\textbf{\theTH.}\ }

\newcounter{CA}
\renewcommand{\theCA}{\arabic{CA}}
\newcommand{\Case}[1][]{\noindent\textbf{Case}\ \refstepcounter{CA}\textbf{\theCA.}\ }

\newcounter{LM}[section]
\renewcommand{\theLM}{\thesection.\arabic{LM}}
\newcommand{\Lemma}[1][]{\noindent\textbf{Lemma}\ \refstepcounter{LM}\textbf{\theLM.}\ }

\newcounter{PS}[section]
\renewcommand{\thePS}{\thesection.\arabic{PS}}
\newcommand{\Proposition}[1][]{\noindent\textbf{Proposition}\ \refstepcounter{PS}\textbf{\thePS.}\ }
\newcounter{DF}[section]
\renewcommand{\theDF}{\thesection.\arabic{DF}}
\newcommand{\Definition}[1][]{\noindent\textbf{Definition}\ \refstepcounter{DF}\textbf{\theDF.}\ }

\newcounter{Rm}[section]
\renewcommand{\theRm}{\thesection.\arabic{Rm}}
\newcommand{\Remark}[1][]{\noindent\textbf{Remark}\ \refstepcounter{Rm}\textbf{\theRm.}\ }

\newcounter{Co}[section]
\renewcommand{\theCo}{\thesection.\arabic{Co}}
\newcommand{\Corollary}[1][]{\noindent\textbf{Corollary}\ \refstepcounter{Co}\textbf{\theCo.}\ }

\usepackage{enumitem}
\allowdisplaybreaks[4]

\newcounter{math}
\addtocounter{math}{1}
\newcommand{\dd}{{\rm d}}
\newtheoremstyle{mystyle}{3pt}{3pt}{\it}{0cm}{\bf}{.}{0.5em}{}
\theoremstyle{mystyle}
\newtheorem{theorem}{Theorem}[section]
\newtheorem{lemma}{Lemma}[section]
\newtheorem{proposition}{Proposition}[section]
\newtheorem{definition}{Definition}[section]
\newtheorem{remark}{Remark}[section]
\newtheorem{corollary}{Corollary}[section]

\newcommand{\blue}{\color{blue}}
\begin{document}
	\title[Inverse Problems for 2-D Steady Supersonic Euler Flows past Curved Wedges]{On Inverse Problems for Two-Dimensional Steady Supersonic Euler Flows past Curved Wedges}
\date{\today}
\author{Gui-Qiang G. Chen}
\address{ Gui-Qiang G. Chen: \, Mathematical Institute, University of Oxford,
Radcliffe Observatory Quarter, Woodstock Road, Oxford, OX2 6GG, UK}

\email{\tt  chengq@maths.ox.ac.hk}

\author{Yun Pu}
\address{ Yun Pu: \, Academy of Mathematics and Systems Science, Chinese Academy of Sciences, Beijing 100190, China;
School of Mathematical Sciences, Fudan University, Shanghai 200433, China}
\email{\tt ypu@amss.ac.cn}

\author{Yongqian Zhang}
\address{Yongqian Zhang: \, School of Mathematical Sciences, Fudan University, Shanghai 200433, China}
\email{\tt  yongqianz@fudan.edu.cn}


\begin {abstract}
We are concerned with the well-posedness of an inverse problem for determining the wedge boundary and associated two-dimensional steady
supersonic Euler flow past the wedge, provided that the pressure distribution on the boundary surface of the wedge and the incoming state
of the flow in the $x$--direction are given.
We first establish the existence of wedge boundaries and associated entropy solutions of the inverse problem, when the pressure on the wedge boundary is larger than
that of the incoming flow but less than a critical value, and the total variation of
the incoming flow and the pressure distribution is sufficiently small.
This is achieved by a careful construction of
suitable approximate solutions
and corresponding approximate boundaries
via developing a wave-front tracking algorithm
and the rigorous proof of their strong convergence
subsequentially to a global entropy solution and a wedge boundary, respectively.
Then we establish the $ L^{\infty}$--stability of the wedge boundaries, by introducing a modified Lyapunov functional
for two different solutions with two distinct boundaries, each of which may contain a strong shock-front.
The modified Lyapunov functional is carefully designed to control the distance between the two boundaries
and is proved to be Lipschitz continuous with respect to the differences of the incoming flow and the pressure on the wedge,
which leads to
the existence of the Lipschitz semigroup $\mathfrak{S}_x$
as a converging limit of the approximate solutions and boundaries.
Finally, when the pressure distribution on the wedge boundary is sufficiently close to that of the incoming flow,
using this semigroup, we compare two solutions of the inverse problem in the respective supersonic full Euler flow and
potential flow and prove that, at $x>0$, the distance between the two boundaries
and the difference of the two solutions
are of the same order of $x$ multiplied by
the cube of the perturbations of the initial boundary data in $L^\infty\cap BV$.
\end{abstract}
\date{}
\maketitle

\section{Introduction}\label{sect-intro}
We are concerned with the well-posedness of an inverse problem for two-dimensional steady supersonic Euler flows
past wedges; see Fig.\,\ref{fig:1.1}.
The inviscid compressible flows are governed
by the following two-dimensional steady Euler system:
\begin{equation}
\left\{\begin{aligned}
	&(\rho u)_x+(\rho v)_y=0, \\
	&(\rho u^2+p)_x+(\rho uv)_y=0,\\
	&(\rho uv)_x+(\rho v^2+p)_y=0,\\
	&\big(\rho u(E+\frac{p}{\rho})\big)_x
	+\big(\rho v(E+\frac{p}{\rho})\big)_y=0,
\end{aligned}
\right.\label{eqn:TwoDFEuler}
\end{equation}
where $\textbf{u}=(u, v)^\top$ is the velocity, $p$ the pressure, $\rho$ the density, and $E$ the total energy:
\begin{align*}
E=\frac{1}{2}|\textbf{u}|^2+e(\rho,p),
\end{align*}
with the internal energy $e$ as a given function of $(\rho,p)$.\par

For ideal gases, the relation between pressure $p$ and internal energy $e$ can be expressed as
\begin{equation}
p=\rho RT,\quad\,\, e=c_{\nu}T
\label{eqn:thermoRela}
\end{equation}
with $T$ standing for the temperature, $S$ the entropy, and $\gamma=1+\frac{R}{c_{\nu}}>1$ for some constant $R>0$.
In particular, using the thermodynamic variables $(\rho, S)$, we have
\begin{equation}
p=p(\rho,S)=\kappa\rho^{\gamma}e^{S/c_{\nu}},\quad\,\, e=\frac{\kappa   }{\gamma-1}\rho^{\gamma-1}e^{S/c_{\nu}}=\frac{RT}{\gamma-1},
\label{eqn:thermoRela2}
\end{equation}
where $\kappa$ and $c_{\nu}>0$ are constants.
For the isentropic polytropic gas, $\gamma>1$, while, for the isothermal flow, $\gamma=1$.
The sonic speed of the flow is $c:=\sqrt{p_\rho(\rho,S)}$. For polytropic gases, $c=\sqrt{\gamma p/\rho}$.\par

\smallskip
For isentropic and irrotational flow,
the governing equations form the potential flow system:
\begin{equation}
\left\{\begin{aligned}
	&(\rho u)_x+(\rho v)_y=0, \\
	&v_x-u_y=0,
\end{aligned}
\right.\label{eqn:TwoDPotential}
\end{equation}
obeying the Bernoulli law:
\begin{equation}
\dfrac{1}{2}(u^2+v^2)+\frac{\gamma \rho^{\gamma-1}}{\gamma-1}=B_{\infty}.
\label{eqn:Bernouli}
\end{equation}
where we have used the pressure-density relation: $p=\rho^\gamma$ without loss of generality by scaling.

\begin{figure}
\begin{center}
	\includegraphics{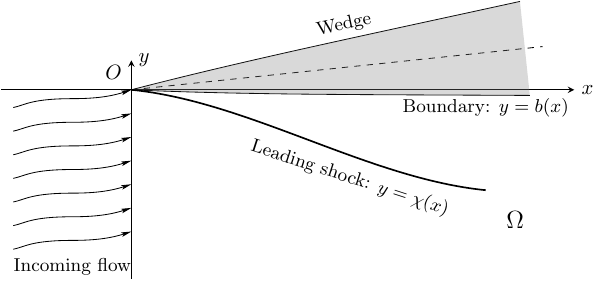}
\end{center}
\caption{The inverse problem for two-dimensional steady Euler equations}
\label{fig:1.1}
\end{figure}

The mathematical analysis for two-dimensional (2-D) steady supersonic flows past wedges
initiated in the 1940s ({\it cf.} Courant-Friedrichs \cite{Courant1948}).
As indicated in \cite{Courant1948}, when a supersonic flow passes a straight-sided wedge whose vertex angle
is less than the critical angle ({\it i.e.}, the sonic angle),
a supersonic shock issuing from the wedge vertex can be determined, and
both of the constant states connected by the shock are supersonic.
Moreover, the opening angle of the wedge completely determines whether there exists such a supersonic shock.
If the wedge is a perturbation of a straight-sided one, local solutions near the wedge vertex were
first studied in Gu \cite{Gu1962}, Li \cite{Li1980}, Schaeffer \cite{Schaeffer1976}, and the references cited therein.
The global existence of solutions for potential flows was obtained
in \cite{Chen1997,Chen1998,Chen1998J,Zhang1999,Zhang2003} in different types of setups.
For the
full Euler equations, Chen-Zhang-Zhu \cite{Chen2006} first established
the existence of global solutions for supersonic Euler flows past wedges and
the stability of the strong shock-front attached to the vertex via a modified Glimm scheme.
Later on, Chen-Li in \cite{Chen2008} used a wave-front tracking method to
establish the ${L}^1$--stability of entropy solutions with strong shock-fronts
and obtained the uniform estimates for the Lipschitz semigroup $S_x$ defined by the limit of the wave-front
tracking approximate solutions,
based on which the uniqueness of solutions within a broader class of viscosity solutions was proved.
Recently, Chen-Kuang-Xiang-Zhang studied the $L^1$--stability problem for hypersonic similarity laws
for steady compressible full Euler flows over 2-D Lipschitz wedges via a wavefront tracking approach
and obtained the optimal convergence rate in \cite{Chen2023,Chen2024} with {\it large data}.
For supersonic Euler flows over almost straight walls, the existence of the full Euler solutions and the stability
of large vortex sheets and entropy waves in BV were first established by Chen-Zhang-Zhu \cite{Chen2007},
while Chen-Kukreja \cite{Chen2022} established the well-posedness of that problem.
Later on, Zhang \cite{Zhang2007Z} made full use of the semigroup corresponding to the Cauchy problems
to prove that, at $x>0$, the solutions of the supersonic potential flow system approach that of the full Euler system
at an order of $x$ multiplied by  the cube of the perturbations of the initial boundary data;
(also {\it cf.} \cite{Bianchini2002,Saint-Raymond2000}).  See also \cite{Chen1998} and the references cited therein.
 \par

 Corresponding to the stability problems for shocks, vortex sheets, and entropy waves,
 two inverse problems have been investigated.
 One of them is to determine the shape of the wedge in a 2-D steady supersonic flow, provided that
 the location of the leading shock front is {\it a priori} given.
 This inverse problem
 was considered
 by Li-Wang in \cite{Wang2007,Li2009,Li2006,Li2007,Wang2014,Wang2019},
 in which a smooth leading shock is assumed and then the characteristic methods are applied for seeking
 a piecewise smooth solutions containing only one discontinuity, the leading shock; see also \cite{Li2022}.
 The other
 inverse problem is to determine the shape of the wedge or the cone with given pressure distribution
 on the wedge boundary in a 2-D steady supersonic flow or an axisymmetric conical steady supersonic flow;
 see \cite{Pu2023} for the inverse problem for the 2-D case and see\cite{Chen2023b} for the 3-D axisymmetric case.
 This inverse problem plays crucial roles in the aircraft design, especially in the inverse design;
 see \cite{Abbott1959,Abzalilov2005,Caramia2019,Goldsworthy1952,Golubkin1988,Golubkin1994,Maddalena2020,Mohammadi2001,
 Robinson1956,Vorobev1998}.
 Though some
 numerical methods and linearized algorithms to deal with this problem have been developed,
 it seems that there is no available rigorous mathematical analysis on the well-posedness of solutions to
 the inverse problem for a supersonic steady Euler flow past wedges.\par

In this paper, for completeness, we first establish the existence of entropy solutions and wedge boundaries of
the second inverse problem by employing the wave-front tracking method,
given the pressure distribution (whose total variation is suitably small)
on the wedge and the incoming flow (a BV perturbation of a uniform flow); see Fig.\,\ref{fig:1.1}.
Then we investigate the ${L}^\infty$--stability of the wedge boundary and the ${L}^1$--stability of entropy solutions
via a modified Lyapunov functional.
Based on these, we are able to deduce a uniformly Lipschitz semigroup $\mathfrak{S}_x$,
which is defined by the converging limit of both approximate boundaries and approximate solutions
generated by the wave-front tracking algorithm.
Finally, we use this semigroup to compare two solutions of the inverse problem in the respective
supersonic full Euler flow and potential flow, assuming the pressure distribution on the wedge boundary
is sufficiently close to the pressure of the incoming flow.
As a result, we prove that, at $x>0$, the distance between two boundaries and the difference of two solutions
are of the same order of $x$ multiplied by the cube of the perturbations of initial boundary data,
\textit{i.e.},
$O(1)x \|(\widetilde{\textbf{u}}_{\infty},
\tilde{p}_{b})\|_{ L^{\infty}\cap BV}^3$.
\par

\smallskip
To be precise, denote $U:=(\textbf{u}^\top,p,\rho)^\top=(u,v,p,\rho)^\top$, and consider vector functions:
\begin{align*}
W(U)&=
(\rho u, \rho u^2+p, \rho uv, \rho u(h+\frac{u^2+v^2}{2}))^\top,\\
H(U)&=
(\rho v, \rho uv, \rho v^2 + p, \rho v(h+\frac{u^2+v^2}{2}))^\top,
\end{align*}
with $h=\frac{\gamma p}{(\gamma-1)\rho}$.
Then system (\ref{eqn:TwoDFEuler}) is reformulated into the conservative form:
\begin{equation}
W(U)_x+H(U)_y=0.
\label{eqn:ConseForm}
\end{equation}

Our problem is to seek the wedge boundary $y=b(x)<0$ and solve (\ref{eqn:ConseForm}) in the corresponding domain:
\begin{align*}
\Omega=\{(x,y)\,:\,x\geq0,\,y<b(x)\}
\end{align*}
with the upper boundary
\begin{align*}	
\Gamma=\{(x,y)\,:\,x\geq0,\,y= b(x)\},
\end{align*}
such that
\begin{equation}
\textbf{u}\cdot\textbf{n}|_{\Gamma}=0,
\label{eqn:Boundary}
\end{equation}
with
\begin{align*}	
\textbf{n}=\textbf{n}(x,b(x))=\frac{(-b'(x),1)^\top}{\sqrt{1+(b'(x))^2}}
\end{align*}
as the corresponding outer normal vector to $\Gamma$ when $b(x)$ is differentiable.

To solve the inverse problem, we assume that the initial-boundary data satisfy the following condition:

\smallskip
\Condition\label{Cond:1} $\, U_\infty(y):=\overline{U}_{\infty}+\widetilde{U}_{\infty}(y)$ is the incoming flow at $x = 0$ such that
\begin{itemize}
	\item[(a)] $\overline{U}_{\infty}=(\overline{\textbf{u}}_{\infty}^\top,\overline{p}_{\infty},\overline{\rho}_{\infty})^\top$ is a constant vector with
	\begin{align}\label{flow-angle}
		\begin{aligned}
			&\overline{u}_{\infty}>0,\quad \overline{v}_{\infty}>0,\quad
			\overline{M}_{\infty}:=\dfrac{|\overline{\textbf{u}}_{\infty}|}{\overline{c}_{\infty}}=|\overline{\textbf{u}}_{\infty}|\sqrt{\frac{\overline{\rho}_{\infty}}{\gamma \overline{p}_{\infty}}}>1,\\
			&\dfrac{\overline{v}_{\infty}}{\overline{u}_{\infty}}=\tan\overline{\theta}_{\infty}
			:=\dfrac{\frac{\bar{p}_{b}}{\overline{p}_{\infty}}-1}{\gamma \overline{M}_{\infty}^2-\frac{\bar{p}_{b}}{\overline{p}_{\infty}}+1}
			\sqrt{\dfrac{(1+\frac{\gamma-1}{\gamma+1})(\overline{M}_{\infty}^2-1)-(\frac{\bar{p}_{b}}{\overline{p}_{\infty}}-1)}{\frac{\bar{p}_{b}}{\overline{p}_{\infty}}+\frac{\gamma-1}{\gamma+1}}}.
		\end{aligned}	
	\end{align}
\item[(b)] $\widetilde{U}_{\infty}(y)=(\widetilde{\textbf{u}}_\infty^\top,\widetilde{p}_\infty,\widetilde{\rho}_\infty)^\top(y)\in( L^1\cap\text{BV})(\mathbb{R};\mathbb{R}^4)$
is a BV perturbation at $x=0$ with sufficiently small
   $\|\widetilde{U}_{\infty}\|_{ L^{\infty}\cap BV}:=\|\widetilde{U}_{\infty}\|_{ L^{\infty}}+\text{T.V.}(\widetilde{U}_{\infty})$.
\end{itemize}

\smallskip
\Condition\label{Cond:2} $\,$ $p_b(x):=\bar{p}_b+\tilde{p}_{b}(x)$ is the pressure distributions
on the wedge boundary such that
\begin{itemize}
	\item[(a)] $\bar{p}_b$ a constant so that
	\begin{align*}
		\overline{p}_{\infty}<\bar{p}_b<p_{\text{\rm sonic}},
	\end{align*}
	where $p_{\text{\rm sonic}}$ is a critical pressure to be specified later in (\ref {eqn:sonic});
	\item[(b)] $\tilde{p}_{b}(x)\in( L^1\cap\text{BV})(\mathbb{R};\mathbb{R})$ is a BV perturbation with
     sufficiently small $\|\tilde{p}_{b}\|_{ L^{\infty}\cap BV}$.
\end{itemize}

\smallskip
With these setups, it suffices to consider the following inverse problem
for $(U, b)$:\\[5pt]
\textbf{Incoming Flow Condition:}
\begin{equation}
U|_{x=0}=U_\infty(y),
\label{eqn:Cauchyda}
\end{equation}
\textbf{Free Boundary Conditions:}
\begin{align}
&p(x,b(x))=p_b(x), \label{eqn:Boundarycon}\\
&(u,v)\cdot(-b'(x),1)=0. \label{1.9a}
\end{align}

\medskip
We seek entropy solutions of the inverse problem (\ref{eqn:ConseForm})--(\ref{1.9a}) in the following sense:

\begin{definition}[Entropy Solutions]\label{def:EntrSol}
A wedge boundary $\Gamma=\{(x,b(x))\,:\, b(x)\in\text{\rm Lip}([0,\infty))\}$ and
a vector function $U=(\textbf{u}^\top,p,\rho)^\top\in\text{BV}(\Omega)$
form an entropy solution of the inverse problem \eqref{eqn:ConseForm}--\eqref{1.9a}
if they satisfy the following{\rm :}
\begin{itemize}
\item[\rm(i)] $U$ is a global weak solution of \eqref{eqn:ConseForm} satisfying \eqref{eqn:Cauchyda}--\eqref{1.9a}
 in the trace sense{\rm ;}
\item[\rm (ii)] For any $a(S)\in C^1$ with $a'(s)\ge 0$, the entropy inequality{\rm :}
\begin{equation}	
(\rho u a(S))_x+(\rho v a(S))_y \geq 0
\label{eqn:Clausius}
\end{equation}
holds in the distributional sense on $\Omega\cup\Gamma$.
\end{itemize}
\end{definition}

\smallskip
The first main theorem of this paper is

\smallskip
\noindent
\textbf{Main Theorem I} (Well-posedness). There exists $\epsilon_\infty>0$ such that,
when $\|(\widetilde{U}_{\infty}, \tilde{p}_{b})\|_{ L^{\infty}\cap BV}
<\epsilon_\infty$, the following results hold{\rm :}
\begin{itemize}
\item [(\rm i)] \textit{Global existence}:
The wedge boundary $y=b(x)=\int_{0}^{x}b'_+(\xi)\,\dd \xi$, with $b'_+(x)\in {\rm BV}(\mathbb{R}_{+})$ as the right-derivative, can be determined
by the initial-boundary conditions (\ref{eqn:Boundary})--(\ref{1.9a}),
which is a small perturbation of
the straight wedge boundary $y=b_{0}x$, and a corresponding global entropy solution $U(x,y)$
that satisfies the requirements of Definition \ref{def:EntrSol} can be obtained, which has bounded total variation:
	\begin{equation}	
		\sup_{x>0}\text{T.V.}\{U(x,y)\, :\, -\infty<y<b(x)\}<\infty,
		\label{eqn:bdTV}
	\end{equation}	
and contains a strong shock $y=\chi(x)=\int_{0}^xs(\xi)\dd \xi$, where
$s(x)\in {\rm BV}(\mathbb{R}_{+})$ and $\chi(x)$ is a small perturbation of $y=s_{0}x$,
the straight strong shock corresponding to the straight wedge.

\item[\rm (ii)] \textit{Existence of semigroup}:
There exists $\varepsilon>0$ such that, for any $(b(0),U_\infty(\cdot+b(0)),p_b)\in\mathbb{D}^{\varepsilon}$ (see Definition \ref{def:Domain} below),
the solution of the inverse problem determines a uniform Lipschitz semigroup $\mathfrak{S}_{x}$:
	\begin{align*}
			(b(0),U_\infty(\cdot+b(0)),p_b)\mapsto (b(x),U(x,\cdot+b(x)),\iota_xp_b)
		\end{align*}		
satisfying that
\begin{align*}
&\mathfrak{S}_{0}(b(0),U_\infty(\cdot+b(0)),p_b)=(b(0),U_\infty(\cdot+b(0)),p_b),\\
&\mathfrak{S}_{x_1}\mathfrak{S}_{x_2}(b(0),U_\infty(\cdot+b(0)),p_b)=\mathfrak{S}_{x_1+x_2}(b(0),U_\infty(\cdot+b(0)),p_b),
\end{align*}
and there exist $L^\sharp$ and $L^\flat>0$ so that, for $(b_i(0),U_{\infty,i}(\cdot+b_i(0)),p_{b,i})\in\mathbb{D}^{\varepsilon}$ and any $x_i\ge0$, $i=1,\,2$,
\begin{align*}
		&\qquad\|\mathfrak{S}_{x_1}(b_1(0),U_{\infty,1}(\cdot+b_1(0)),p_{b,1})-\mathfrak{S}_{x_2}(b_2(0),U_{\infty,2}(\cdot+b_2(0)),p_{b,2})\|_{Y}\\
		&\qquad\,\leq L^\sharp\,\|(b_1(0),U_{\infty,1}(\cdot+b_1(0)),p_{b,1})-(b_2(0),U_{\infty,2}(\cdot+b_2(0)),p_{b,2})\|_{Y}+L^\flat\,|x_1-x_2|,
		\end{align*}
where $\iota_x$ and $\|\cdot\|_Y$ are given in \eqref{5.1a} and \eqref{5.2a} below, respectively.
\end{itemize}

\medskip
To compare the difference of the full Euler flows and the potential flows in solving the above inverse problem,
we also need to study the inverse problem for (\ref{eqn:TwoDPotential})--(\ref{eqn:Bernouli}).
Similarly, we have parallel results for the potential flows: There exist both a wedge boundary $y=b_{\rm P}(x)<0$
and an entropy solution $\textbf{u}_{\rm P}=(u_{\rm P},v_{\rm P})^\top$ in
\begin{align*}
	\Omega_{\rm P}=\{(x,y)\,:\, x\geq0,y<b_{\rm P}(x)\}
\end{align*}
with the upper boundary:
\begin{align*}	
	\Gamma_{\rm P}=\{(x,y)\,:\, x\geq0,y= b_{\rm P}(x)\},
\end{align*}
satisfying

\medskip
\noindent
{\bf Incoming Flow Condition:}
\begin{equation}
	\textbf{u}_{\rm P}|_{x=0}=\textbf{u}_{\infty},
\label{eqn:WSLCauchyData}
\end{equation}
{\bf Free Wedge Boundary Conditions:}
\begin{align}
&p_{\rm P}(x,b(x))=p_b(x), \label{eqn:WSLPFBoundaryData}\\[1mm]
&\textbf{u}_{\rm P}\cdot\textbf{n}_{\rm P}|_{\Gamma_{\rm P}}=0.
\label{eqn:WSLBoundaryConstruction}
\end{align}

\medskip
\noindent
The initial-boundary data are assumed to satisfy the following condition:\\[5pt]
\ConditionP \label{JS:1} $\,\textbf{u}_{\infty}(y):=\overline{\textbf{u}}_{\infty}+\widetilde{\textbf{u}}_{\infty}(y)$ is the incoming flow
at $x=0$ such that
\begin{itemize}
\item[(a)] $\overline{\textbf{u}}_{\infty}=(\overline{u}_\infty,0)^\top$ is a constant vector with
\begin{equation}
	\overline{u}_\infty>0,\quad \frac{2(\gamma-1)B_{\infty}}{\gamma+1}<|\overline{u}_{\infty}|^2<2B_{\infty}.
		\label{eqn:WSLIBVU}
	\end{equation}
	\item[(b)] $\widetilde{\textbf{u}}_{\infty}(y)=(\tilde{u}_\infty,\tilde{v}_\infty)^\top(y)\in( L^1\cap\text{BV})(\mathbb{R};\mathbb{R}^2)$ is a BV perturbation at $x=0$ with sufficiently small
   $\|\widetilde{\textbf{u}}_{\infty}\|_{ L^{\infty}\cap BV}$.
\end{itemize}

\smallskip
\ConditionP\label{JS:2}
$\, p_b(x):=\overline{p}_b+\tilde{p}_{b}(x)$  is the pressure distribution on the wedge boundary such that
\begin{itemize}
	\item[(a)] $\bar{p}_b=(\mathscr{R}(|\overline{\textbf{u}}_\infty|))^{\gamma}$ with
	\begin{equation}
	\mathscr{R}(r)=\frac{\gamma-1}{\gamma}\,\big(B_{\infty}-\frac{r^2}{2}\big)^{\frac{1}{\gamma-1}}.
		\label{eqn:WSLBonuli}
	\end{equation}	
	\item[(b)] $\tilde{p}_{b}(x)\in( L^1\cap\text{BV})(\mathbb{R};\mathbb{R})$ is a BV perturbation with
   sufficiently small $\|\tilde{p}_{b}\|_{ L^{\infty}\cap BV}$.
\end{itemize}

\medskip
Meanwhile, in solving inverse problems for the full Euler equations, we assume that\\[5pt]
\ConditionE\label{JS:3} At $x=0$, the incoming flow $U_\infty(y)=(\textbf{u}_\infty^\top,(\mathscr{R}(|\textbf{u}_\infty|))^{\gamma},\mathscr{R}(|\textbf{u}_\infty|))^\top$.\\
\ConditionE\label{JS:4} The pressure distribution  $p_b(x):=
\overline{p}_b+\tilde{p}_{b}(x)$ on the wedge boundary satisfies
\begin{itemize}
	\item[(a)] $\bar{p}_b=(\mathscr{R}(|\overline{\textbf{u}}_\infty|))^{\gamma}$, where $\mathscr{R}(r)$ is defined in (\ref{eqn:WSLBonuli}).
	\item[(b)] $\tilde{p}_{b}(x)\in( L^1\cap\text{BV})(\mathbb{R};\mathbb{R})$ is a BV perturbation with sufficiently small
  $\|\tilde{p}_{b}\|_{ L^{\infty}\cap BV}$.
\end{itemize}

\smallskip
\begin{remark}
In the above conditions, if $\bar{p}_b>(\mathscr{R}(|\overline{\textbf{u}}_\infty|))^{\gamma}$ but smaller than a critical value,
following the same argument for the proof of {\rm Main Theorem I},
we can also obtain the existence and stability of the solutions of the inverse problem
for the potential flow equations containing a strong shock.
\end{remark}
	
In particular, when only weak waves are involved, we have

\smallskip
\noindent
\textbf{Main Theorem II} (Comparison of the two models). Assume that (\ref{JS:1})--(\ref{JS:2}) and (\ref{JS:3})--(\ref{JS:4}) hold.
Let $U_{\rm E}=(\textbf{u}_{\rm E}^\top,p_{\rm E},\rho_{\rm E})^\top$ be the entropy solution
of (\ref{eqn:ConseForm})--(\ref{1.9a}) with the corresponding boundary function $y=b_{\rm E}(x)$,
and let $U_{\rm P}=(\textbf{u}_{\rm P}^\top,(\mathscr{R}(|\textbf{u}_{\rm P}|))^{\gamma},\mathscr{R}(|\textbf{u}_{\rm P}|))^\top$
be the entropy solution of (\ref{eqn:TwoDPotential})--(\ref{eqn:Bernouli})
and (\ref{eqn:WSLPFBoundaryData})--(\ref{eqn:WSLBoundaryConstruction}) with the corresponding boundary
function $y=b_{\rm P}(x)$.
Then there exist $\epsilon_{\rm c}>0$ and $C>0$ such that, when $\|(\widetilde{\textbf{u}}_{\infty}, \tilde{p}_{b})\|_{ L^{\infty}\cap BV}
<\epsilon_{\rm c}$, for any $x>0$,
\begin{align*}
\|(b_{\rm E}(x),U_{\rm E}(x,\cdot+b_{\rm E}(x)),p_b)-(b_{\rm P}(x),U_{\rm P}(x,\cdot+b_{\rm P}(x)),p_b)\|_{Y}
\leq Cx\, \|(\widetilde{\textbf{u}}_\infty, \tilde{p}_{b})\|_{ L^{\infty}\cap BV}^3.
\end{align*}

\smallskip
In comparison with the previous results on the initial-boundary value problem with fixed boundary
and the Cauchy problem, one of the new main difficulties in solving the inverse problem
is how the unknown wedge boundary is determined, especially in the solution space of low regularity, $L^\infty\cap BV$.
For the $ L^1$--stability of solutions of the initial-boundary value problem with fixed boundary,
the only thing that needs to be compared is the difference between two different solutions in the same fixed domain.
Thus, the Lyapunov functional that measures the $ L^1$--difference of two solutions
can be designed and then, after a careful analysis of the changes of this functional near the boundary,
a group of weights can be chosen to make the functional decrease, by combining the interaction estimates
only involving weak waves and the fact that the total strengths of weak waves in the other families are dominated
by the strength of that strong leading shock.
However, for the inverse problem, new phenomena occur, mainly owing to the unknown boundaries.
To establish the stability of solutions of the inverse problem, we have to determine a region
on which two solutions are compared and
the additional distance between the two boundaries is required to be controlled.
At a glance, it is ambiguous whether a corresponding Lyapunov functional could be constructed,
not to mention how it could be non-increasing. \par

To overcome this difficulty, our method is to consider the difference of the two solutions and the two boundaries simultaneously.
We first extend two solutions to a larger domain (exactly the union of the two domains on which the two inverse problems solve).
Then we construct a Lyapunov functional that controls the $ L^\infty$--norm of the difference of the two boundaries.
Using the boundary condition that the pressure near the boundary coincides with the given pressure distribution
on the wedge boundary (implying the additional quantitative relations),
we can obtain the precise estimates near the boundary ({\it cf.} (\ref{eqn:BoundaryLyapunov})),
which suggests a proper choice of the weights of the Lyapunov functional.
Then, employing the interaction estimates only involving weak waves and using the strength of the leading shock to control
the total strengths of weak waves in the other families, we obtain that the functional is non-increasing in the flow direction.\par

Furthermore, it seems to be difficult to directly compare the two solutions of the full Euler equations and the potential flow equations.
On the other hand, following \cite{Zhang2007Z}, we can make full use of the semigroup $\mathfrak{G}_x$
after taking the boundary influence into consideration.
We first compare the difference of the solutions of the Riemann-type problems including the Riemann-type inverse problems
in the two different models, and then establish an estimate of the difference between the two corresponding approximate solutions locally.
Combining the local estimates and properties of the semigroup (see Proposition \ref{prop:daoshudier}), we finally obtain our result as desired.

We organize the rest of this paper as follows:
In \S \ref{sect-pre}, we study some basic properties of the full Euler equations and related Riemann-type problems
including the Riemann-type inverse problems. We also obtain the corresponding nonlinear wave interaction estimates.
In \S \ref{sect-appro}, we first introduce a wave-front tracking algorithm to construct approximate boundaries and solutions.
Then an interaction potential $Q$ is given by considering both the influence of pressure changing on the wedge boundary
and the interaction estimates of wave-fronts together, based on which we design a Glimm-type functional
and prove that it is non-increasing.
Therefore, the global existence of entropy solutions of the inverse problem is obtained.
In \S \ref{sect-Lya}, a modified Lyapunov functional $\mathfrak{F}$ for the two solutions
is given, which is equivalent to the $ L^1$--distance between the two solutions and
the $ L^\infty$--distance between the two boundaries.
Then we prove that $\mathfrak{F}$ is non-increasing as $x$ increases,
which leads to
the $ L^1$--stability of the solutions that contain a strong shock
and the $ L^\infty$--stability of two boundaries.
Next, in \S \ref{sect-semi}, from the elementary estimates established in \S\ref{sect-appro}--\S \ref{sect-Lya},
we show that there exists a Lipschitz semigroup $\mathfrak{S}_x$ generating the entropy solutions and boundaries of
the inverse problem.
Finally, in \S\ref{sect-compare}, we use this semigroup to compare the two solutions of the inverse problem
in a supersonic Euler flow and a supersonic potential flow, and prove that, at $x>0$,
the distance between the two boundaries and the difference of the two solutions in $L^1$  are of the same order of $x$
multiplied by
the cube of the perturbation of initial boundary data, \textit{i.e.},
$O(1)x \,\|(\widetilde{\textbf{u}}_{\infty},\tilde{p}_{b})\|_{ L^{\infty}\cap BV}^3$.

\smallskip
\section{Steady Euler Equations and Riemann Problems}\label{sect-pre}
In this section, we first give some basic
properties of
system (\ref{eqn:ConseForm})
and then study nonlinear waves and related interaction estimates that are used in the subsequent development.

When $u>c$, system (\ref{eqn:ConseForm}) has four eigenvalues:
\begin{equation}\label{eqn:Eigenval}
\begin{aligned}
	\lambda_{j}&=\frac{uv+(-1)^{j}c\sqrt{u^2+v^2-c^2}}{u^2-c^2}\qquad \mbox{for $j=1,\,4$},\\
	\lambda_{k}&=\frac{v}{u}\qquad \mbox{for $k=2,\,3$},
\end{aligned}
\end{equation}
and corresponding eigenvectors:
\begin{equation}
\begin{aligned}
	\textbf{r}_{j}&=\mathfrak{k}_j(-\lambda_j,1,\rho(\lambda_ju-v),\frac{\rho(\lambda_ju-v)}{c^2})^\top\qquad \mbox{for $j=1,\,4$},\\
	\textbf{r}_{2}&=(u,v,0,0)^\top,\qquad \textbf{r}_{3}=(0,0,0,\rho)^\top,
\end{aligned}\label{eqn:Eigenvec}
\end{equation}
where
\begin{equation}
	\mathfrak{k}_j=\mathfrak{k}_j(U)=\frac{2}{\gamma+1}\,\frac{(u^2-c^2)\lambda_j-uv}{(1+\lambda_j^2)(\lambda_ju-v)}\
\qquad \mbox{for $j=1,\,4$}.\label{eqn:Coeffi}
\end{equation}
It is direct to see that $\textbf{r}_{j}\cdot\nabla\lambda_j=1$ for $j=1,\,4$, and $\textbf{r}_{k}\cdot\nabla\lambda_k=0$ for $k=2,\,3$.\par
The discontinuous wave curves of (\ref{eqn:ConseForm}) must satisfy the Rankine-Hugoniot conditions:
\begin{equation}
s[W(U)]=[H(U)],
\label{eqn:RHcondi}
\end{equation}
where $s$ is the discontinuity speed.\par

\smallskip
The contact Hugoniot curves $C_k(U_{0}), k=2,3$, through $U_{0}$ are

\smallskip
\indent\textit{Vortex sheets}:
\begin{equation}
C_2(U_{0}):\qquad s=\frac{v}{u}=\frac{v_{0}}{u_{0}},\quad p=p_{0},\quad \rho=\rho_{0};
\label{eqn:VorShe}
\end{equation}
\indent\textit{Entropy waves}:
\begin{equation}
C_3(U_{0}):\qquad s=\frac{v}{u}=\frac{v_{0}}{u_{0}},\quad \textbf{u}=\textbf{u}_{0},\quad  p=p_{0}.
\label{eqn:EntrWave}
\end{equation}
Corresponding to the repeated eigenvalues $\lambda_2=\lambda_3=\frac{v}{u}$,
there are two linearly independent eigenvectors.
Hence, in the physical $(x,y)$--plane, the vortex sheet and the entropy wave appear
as one characteristic discontinuity, while in the phase space, they need to be determined by two parameters independently. \par

The nonlinear $j$-waves, $j=1,\,4$, are shock waves or rarefaction waves.
In the state space, the rarefaction wave curves $R_j^-(U_0)$ through $U_0$ are given by
\begin{equation}
R_j^-(U_{0}):\quad \dd u=-\lambda_j\dd v,\quad \rho(\lambda_ju-v)\dd v=\dd p,\quad  \dd p=c^2\dd \rho
  \qquad \text{for $\rho<\rho_0$},\quad j=1,4.
\label{eqn:Rarefaction}
\end{equation}
The speeds of shock waves are
\begin{equation}
s_j:=
\frac{u_{0}v_{0}+(-1)^{j}\bar{c}_{0}\sqrt{u_{0}^2+v_{0}^2-\bar{c}_{0}^2}}{u_{0}^2-\bar{c}_{0}^2}
\qquad \mbox{for $j=1,\,4$},
\label{eqn:Shockspeed}
\end{equation}
where $\bar{c}_{0}^2=\frac{c_{0}^2}{b_{0}}\,\frac{\rho}{\rho_{0}}$
and $b_{0}=\frac{\gamma+1}{2}-\frac{\gamma-1}{2}\,\frac{\rho}{\rho_{0}}$.
Substituting $s_j$ into (\ref{eqn:RHcondi}) leads to
the $j$-Hugoniot curve $S_j(U_{0})$ across $U_{0}$:
\begin{equation}
S_j(U_{0}):\qquad [u]=-s_j[v],\quad [p]=\frac{c_{0}^2}{b_{0}}[\rho],\quad \rho_{0}(s_ju_{0}-v_{0})[v]=[p]\qquad\,\,\mbox{for $j=1,4$}. \label{eqn:Shockwave}
\end{equation}
For a piecewise smooth solution $U$ that contains a shock, each of the following conditions
is equivalent to (\ref{eqn:Clausius}) (see also \cite{Chen2006,Chen2007}):
\begin{itemize}
\item[(\rm i)] The density increases across the shock along the flow direction:
\begin{equation}	
	\rho_{\text{\rm back}} < \rho_{\text{\rm front}}.
\label{eqn:physiEntrCon}
\end{equation}
\item[(\rm ii)] The speed $s_j$ of the $j$th-shock satisfies
\begin{equation}	
	\qquad \lambda_j(\text{\rm back})<s_j<\lambda_j(\text{\rm front})\,\,\,\,\mbox{for $j=1,\,4$},
	\qquad\,\, s_1<\lambda_{2,3}(\text{\rm back}),\quad \lambda_{2,3}(\text{\rm back})<s_4.
\label{eqn:LaxEntrCon1}
\end{equation}
\end{itemize}

In the phase space, we denote the part of $S_j(U_{0})$ with $\rho>\rho_{0}$ by $S_j^+(U_{0})$, $j=1,\,4$.
In the $(x,y)$--plane, any state on $S_j^+(U_{0})$ leads to a shock connecting to the below state $U_{0}$
satisfying the entropy condition (\ref{eqn:Clausius}) so that $S_j^+(U_{0})$ are called shock curves.
Moreover, curves $S_j^+(U_{0})$ coincide with $R_j^-(U_{0})$ at state $U_{0}$ up to the second order for $j=1,\,4$.\par

As in \cite{Bressan2000,Smoller1983,DCM1}, we parameterize $R_j(U_{0})$ and $S_j(U_{0})$
by $\alpha_{j}\mapsto R_{j}(\alpha_{j})(U_{0})$ and $\alpha_{j}\mapsto S_{j}(\alpha_{j})(U_{0})$, respectively, such that
\begin{align*}
\left.\frac{\dd}{\dd\alpha}R_{j}(\alpha_{j})(U_{0})\right|_{\alpha_{j}=0}=	\left.\frac{\dd}{\dd\alpha}S_{j}(\alpha_{j})(U_{0})\right|_{\alpha_{j}=0}=\textbf{r}_{j}(U_{0}).
\end{align*}
Then we define the nonlinear wave curves $T_j(U_{0})=R_j^-(U_{0})\cup S_j^+(U_{0})$ and parameterize $T_j(U_{0})$ by $\alpha_{j}\mapsto\Phi_{j}(\alpha_{j};U_{0})$ so that
\begin{align*}
\Phi_{j}(\alpha_{j};U_{0})=
\left\{\begin{aligned}
	&R_{j}(\alpha_{j})(U_{0}),\quad \alpha_{j}\geq0, \\
	&S_{j}(\alpha_{j})(U_{0}),\quad \alpha_{j}<0,
\end{aligned}
\right.	\qquad \text{for } j=1,\,4.
\end{align*}
For the linearly degenerate case, $T_{k}(U_{0})=C_{k}(U_{0})$, and we choose parameter $\alpha_{k}\mapsto\Phi_{k}(\alpha_{k};U_{0})$ such that
\begin{equation}
\frac{\dd \Phi_{k}(\alpha_{k};U_{0})}{\dd \alpha}=\textbf{r}_{k}(U_{0})\qquad
\mbox{for $k=2,\,3$}.
\label{eqn:WeakWaveDiff}
\end{equation}
With these, we define
\begin{equation}
\Phi(\alpha_1,\alpha_2,\alpha_3,\alpha_4;U_{0})=\Phi_{4}(\alpha_{4};\Phi_{3}(\alpha_{3};\Phi_{2}(\alpha_{2};\Phi_{1}(\alpha_{1};U_{0})))),\label{eqn:WeakWavecurv}
\end{equation}
and denote the $i$-th component of $\Phi(\alpha_1,\alpha_2,\alpha_3,\alpha_4;U_{0})$
by $\Phi^{(i)}(\alpha_1,\alpha_2,\alpha_3,\alpha_4;U_{0})$, $i=1,\,2,\,3,\,4$.

\subsection{Riemann-type problems and Riemann solutions}
We now investigate several Riemann-type problems and their solutions,
which are essential in the construction of approximate solutions for
the inverse problem (\ref{eqn:ConseForm})--(\ref{1.9a}) when carrying out the front tracking algorithm.\\

\noindent\textit{Inverse Riemann problem.} Consider an inverse Riemann problem with a boundary $\Gamma$ to be determined:
\begin{equation}
\left\{\begin{aligned}
	&\text{(\ref{eqn:ConseForm})}, \\
	&U|_{\{x=\bar{x},\,y<\bar{y}\}}=U_{-},\\
	&p|_{\Gamma}=p_{+},\quad \textbf{u}\cdot\textbf{n}|_{\Gamma}=0,
\end{aligned}
\right.\label{eqn:InvRie}
\end{equation}
where $U_{-}=(\mathbf{u}_{-}^\top,p_{-},\rho_{-})^\top$ satisfies
$|\mathbf{u}_{-}|>\sqrt{\frac{\gamma p_-}{\rho_-}}$,
and $\Gamma$ starts at $(\bar{x},\bar{y})$.
\begin{figure}[ht]
	\begin{center}
		\includegraphics{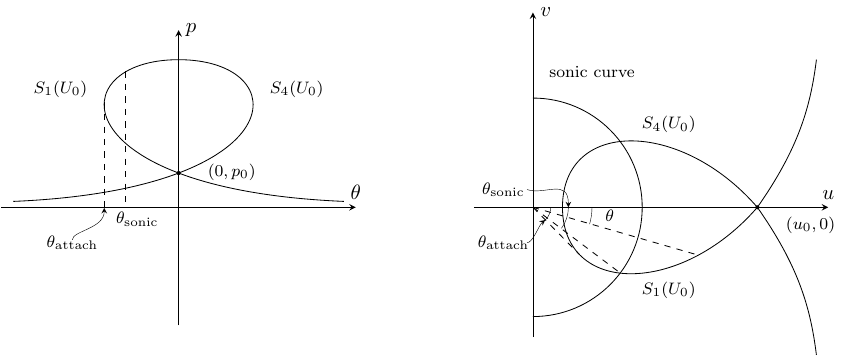}
	\end{center}
	\caption{Shock polar and critical angle}
	\label{fig:2.1}
\end{figure}
 Set $U_{0}=(|\textbf{u}_{-}|,0,p_{-},\rho_{-})^\top$ and denote the shock polar through $U_{0}$
 by $S(U_{0})=S_1(U_{0})\cup S_4(U_{0})$.
 For any state $U=(u,v,p,\rho)^\top$ on the shock polar $S(U_{0})$,
 we use $\theta=\arctan(\frac{v}{u})$ to denote the angle of
 the flow direction.
 Then, from \cite{Courant1948} (also see \cite{Chen2020}), we know that there is a critical angle
 $\theta_{\text{\rm sonic}}<0$
 such that $\theta_{\text{\rm sonic}}<\theta<0$.
 On the $(u,v)$--plane, ray $v=u\tan\theta$ with $u\geq0$ intersects with curve $S_1^+(U_{0})$
 at a supersonic state $U=(u,v,p,\rho)^\top$, $ u^2+v^2>c^2:=
\frac{\gamma p}{\rho}$; see Fig.\,\ref{fig:2.1}. Furthermore, from the relation (see \cite{Chen2020,Courant1948}):
\begin{align*}
	\tan\theta=-\dfrac{\frac{p}{p_{-}}-1}{\gamma M_{0}^2-\frac{p}{p_{-}}+1}
	\sqrt{\dfrac{(1+\frac{\gamma-1}{\gamma+1})(M_{0}^2-1)-(\frac{p}{p_{-}}-1)}{\frac{p}{p_{-}}+\frac{\gamma-1}{\gamma+1}}},\quad\, M_{0}^2=\dfrac{|\textbf{u}_{-}|^2\rho_{-}}{\gamma p_{-}},
	\end{align*}
for the critical angle $\theta_{\text{\rm sonic}}$, we can find a corresponding critical pressure:
\begin{equation}\label{eqn:sonic}
	p_{\text{\rm sonic}}=p_{\text{\rm sonic}}(|\textbf{u}_{-}|,p_{-},\rho_{-})
\end{equation}
such that, given $p_{-}<p<p_{\text{\rm sonic}}$,
there is a unique supersonic state $U=(u,v,p,\rho)^\top\in S_1^+(U_{0})$.\par
\begin{figure}[ht]
	\begin{center}
		\includegraphics{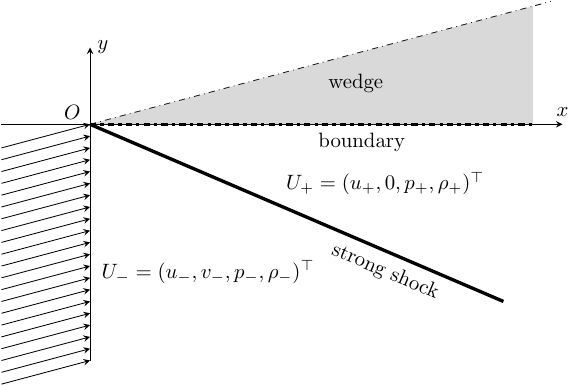}
	\end{center}
	\caption{Background solutions}
	\label{fig:Background}
\end{figure}
Moreover, as in \cite{Courant1948} (see also \cite{Chen2020}), recalling (\ref{flow-angle}),
it can be shown that, when
\begin{equation}
(\bar{x},\bar{y})=(0,0),\quad p_{+}=\bar{p}_{b},\quad U_{-}=\overline{U}_{\infty},	\label{eqn:Backstate}
\end{equation}
there is a unique entropy solution of the above inverse problem,
consisting of two constant states $U_{-}$ and
$U_{+}=(u_{+},0,p_{+},\rho_{+})^\top$ with $u_{\pm}> c_{\pm}>0$,
connected by a 1-shock wave of speed $s_{0}$, and the boundary is thus determined
as $\Gamma=\{(x,0)\,:\,x>0\}$; see Fig.\, \ref{fig:Background}. \\

\noindent\textit{Riemann problem $($only weak waves$)$}: Consider the Riemann problem:
\begin{equation}
\left\{\begin{aligned}
	&\text{(\ref{eqn:ConseForm})}, \\
	&U|_{x=\bar{x}}=	\left\{\begin{aligned}
		&U_{l}\quad\,\, \text{for $y<\bar{y}$}, \\
		&U_{r}\quad\,\, \text{for $y>\bar{y}$},
	\end{aligned}
	\right.
\end{aligned}
\right.\label{eqn:RiePro}
\end{equation}
where the constant states $U_{l}$ and $U_{r}$ are the below state
and the above state with respect to line $y=\bar{y}$,
respectively.
Then there exists $\epsilon_{\rm w}>0$ such that, for $U_{l}$, $U_{r}\in O_{\epsilon_{\rm w}}(U_{-})$, or $U_{l}$,
$U_{r}\in O_{\epsilon_{\rm w}}(U_{+})$, problem (\ref{eqn:RiePro}) has a unique admissible solution containing
at most four waves $(\alpha_1,\alpha_2,\alpha_3,\alpha_4)$ that connect $U_l$ and $U_r$ by $U_r=\Phi(\alpha_1,\alpha_2,\alpha_3,\alpha_4;U_l)$.\\

\noindent\textit{Riemann problem $($containing strong $1$-shock$)$}:  As in \cite{Chen2006,Chen2008}, we have
\begin{lemma}\label{Lem:strongshock}
If $U_{+}$ is connected with $U_{-}$ through a $1$-shock wave of speed $s_{0}$
with $\rho_{+}>\rho_{-}$, \textit{i.e.},
	\begin{equation}
		s_{0}\big(W(U_{-})-W(U_{+})\big)=H(U_{-})-H(U_{+}),
		\label{eqn:background}
	\end{equation}
then
\begin{align*}
&s_{0}<0,\quad u_{+}<u_{-}<(1+\dfrac{1}{\gamma})u_{+},\\
&\det(\nabla_UH(U_{+})-s_0\nabla_U W(U_{+}))>0.
\end{align*}
In addition, there exists $\epsilon_{\rm s}>0$ such that,
for any $U_{0}\in O_{\epsilon_{\rm s}}(U_{-})$, $S_{1}^{+}(U_{0})\cap O_{\epsilon_{\rm s}}(U_{+})$ can be parameterized
by the shock speed $s$ as{\rm :} $s\mapsto G(s;U_{0})$ near $(s_{0},U_{-})$
with $G=\left(G^{(1)},G^{(2)},G^{(3)},G^{(4)}\right)^\top$ $\in C^2$ and $G(s_{0};U_{-})=U_{+}$.
\end{lemma}
	
Now we give an explicit formula of the solution of the Riemann problem (\ref{eqn:RiePro})
in a forward neighbourhood $B_+((\bar x,\bar y),r)$ of $(\bar x,\bar y)$:
\begin{align*}
B_+((\bar x,\bar y),r)
=\{(x,y)\,:\,\mbox{$(x-\bar x)^2+(y-\bar y)^2<r^2$ and $x>\bar x$ for some $r>0$}\}.
\end{align*}
When only weak waves are involved,
assume that $U_{r}=\Phi(\alpha_1,\alpha_2,\alpha_3,\alpha_4;U_{l})$.
Then any solution of this Riemann problem generally has the following form:
\begin{equation}
U_{Rw}(x,y)=\left\{\begin{aligned}
	&U_{m_0} =
 U_{l},\ &&\xi<\sigma_1^-,\\
	&\Phi_1\left(\xi-\lambda_1(U_{m_0});U_{m_0}\right),\ &&\sigma_1^-<\xi<\sigma_1^+,\\
	&U_{m_1},\ &&\sigma_1^+<\xi<\sigma_2^-,\\
	&U_{m_3},\ &&\sigma_3^+<\xi<\sigma_4^-,\\
	&\Phi_4\left(\xi-\lambda_4(U_{m_3});U_{m_3}\right),\ &&\sigma_4^-<\xi<\sigma_4^+,\\
	&U_{m_4}=
 U_{r},\ &&\sigma_4^+<\xi,\\
\end{aligned}
\right.\label{eqn:SimRiemannSolv}
\end{equation}
where
\begin{equation}
\begin{aligned}
	&\xi=\frac{y-\bar y}{x-\bar x},\qquad U_{m_1}=\Phi(\alpha_1,0,0,0;U_{l}),
      \qquad U_{m_3}=\Phi(0,\alpha_2,\alpha_{3},0;U_{m_1}),\\
	&\sigma_j^+=\sigma_j^-=\sigma_j,\qquad  \sigma_j\text{ is the speed of the weak shock $\alpha_j$  when $\alpha_j<0$ for $j=1, 4$},\\
	&\sigma_j^-=\lambda_j(U_{m_{j-1}}),\quad \sigma_j^+=\lambda_j(U_{m_{j}})\qquad \text{ when $\alpha_j>0$ for $j=1, 4$},\\
	&\sigma_2^-=\sigma_3^+=\lambda_2(U_{m_1})=\lambda_3(U_{m_3}).
\end{aligned}
\label{eqn:SimRiemannSolvNota}
\end{equation}
When a strong 1-shock wave is involved,
saying $U_{r}=\Phi(0,\alpha_2,\alpha_3,\alpha_4;G(s;U_{l}))$,
it suffices to let $U_{m_1}=G(s;U_{l})$ and $\sigma_1^+=\sigma_1^-=s$ in (\ref{eqn:SimRiemannSolv}).

\subsection{Wave interaction and reflection estimates}
In the following estimates, $O(1)$ is denoted to be bounded so that the bound of $|O(1)|$ depends only on $U_-$, $U_+$,
and system (\ref{eqn:TwoDFEuler}). Firstly, we have estimates of the interactions among weak waves; see \cite{Chen2008,Chen2006}.

\begin{lemma}\label{lem:Iww}
There is a positive constant $\varepsilon_{w}$ such that,
for three constant states $U_{l}$, $U_{m}$, $U_{r}\in O_{\varepsilon_{w}}(U_{-})$, or $U_{l}$, $U_{m}$, $U_{r}\in O_{\varepsilon_{w}}(U_{+})$,
with $U_{m}=\Phi(\alpha_{1},\alpha_{2},\alpha_{3},\alpha_{4};U_{l})$ and
$U_{r}=\Phi(\beta_{1},\beta_{2},\beta_{3},\beta_{4};U_{m})$,
we can find $(\gamma_{1},\gamma_{2},\gamma_{3},\gamma_{4})$ such that
$U_{r}=\Phi(\gamma_{1},\gamma_{2},\gamma_{3},\gamma_{4};U_{l})$ and
	\begin{align*}
		\gamma_i=\alpha_i+\beta_i+O(1)\triangle(\alpha,\beta)\qquad \mbox{for $i=1,\,2,\,3,\,4$},
	\end{align*}
where $\triangle(\alpha,\beta)=\sum_{1\leq j<i\leq4}|\alpha_i||\beta_j|+\sum_{k=1}^4\triangle_k(\alpha,\beta)$ with
	\begin{equation*}
		\triangle_k(\alpha,\beta)=\left\{\begin{aligned}
			&0\ &&\mbox{when $\alpha_k\geq0$ and $\beta_k\geq0$}, \\
			&|\alpha_k||\beta_k|\ &&\text{otherwise}.
		\end{aligned}
		\right.
	\end{equation*}
\end{lemma}

To balance the alteration of the pressure distribution on the boundary,
we allow $1$-waves emanating from the boundary when solving the following initial-boundary value problem:
\begin{align*}
\left\{\begin{aligned}
	&\text{(\ref{eqn:ConseForm})}, \\
	&U|_{\{x=\bar{x},\,y<\bar{y}\}}=U_1,\\
	&p|_{\Gamma}=p_2,\ \textbf{u}\cdot\textbf{n}|_{\Gamma}=0.
\end{aligned}
\right.
\end{align*}
When only weak waves involve, we have
\begin{lemma}\label{lem:WCe}
There is a positive constant $\varepsilon_{p}$ such that, for any $U_1=(u_{1},v_{1},p_{1},\rho_{1})^\top\in O_{\varepsilon_{p}}(U_{+})$
and $p_{1}$, $p_{2}\in O_{\varepsilon_{p}}(p_{+})$, the equation{\rm :}
	\begin{equation}
		\Phi^{(3)}(\delta_{1},0,0,0;U_{1})=p_{2}  \label{eqn:Boundarycorner}
	\end{equation}
	determines a unique twice differentiable function $\delta_{1}=\delta_{1}(p_{2},U_{1})$.
Furthermore, there exists a bounded quantity $K_{b}$, whose bound is independent of $\delta_{1}$ and $p_{1}-p_{2}$, such that
\begin{align*}
\delta_{1}=K_{b}(p_{2}-p_{1}).
\end{align*}
\end{lemma}

\begin{proof}
Since $\Phi^{(3)}(0,0,0,0;U_{1})=p_{1}$, differentiating (\ref{eqn:Boundarycorner}) with respect to $\delta_{1}$, we have
\begin{align*}
	\dfrac{\partial\Phi^{(3)}(\delta_{1},0,0,0;U_{1})}{\partial\delta_{1}}\Big|_{\{\delta_{1}=0,\,U_{1}=U_+\}}
 =\mathfrak{k}_{1}(U_+)
 \rho_{+}\big(\lambda_{1}(U_+)u_{+}-v_{+}\big)\neq0.
\end{align*}
Then the implicit function theorem gives the result.
Furthermore, the bound of $K_{b}$ depends only on $U_{+}$ and system (\ref{eqn:TwoDFEuler}).
\end{proof}

Also, with the presence of a strong $1$-shock wave (see \cite{Chen2006,Chen2008}), we have
\begin{lemma}
\label{lem:SBe} There exists $\varepsilon_s>0$ such that,
when $p_2\in O_{\varepsilon_s}(p_{+})$ and $U_1\in O_{\varepsilon_s}(U_{-})$, the equation{\rm :}
\begin{equation}
	G^{(3)}(s;U_1)=p_2  \label{eqn:StrongOrigin}
\end{equation}
determines a unique twice differentiable function $s=s(p_2,U_1)$ with
\begin{align*}
	&s=s_0+ K_{bs}(|p_2-p_+|+|p_1-p_-|),
\end{align*}
where $|K_{bs}|$ has a bound depending only on $U_{-}$, $U_{+}$, and system \eqref{eqn:TwoDFEuler}.
\end{lemma}

\smallskip
Next, we give the estimates of reflections of weak waves on the boundary.
\begin{lemma}
\label{lem:WBe}
There exists a positive constant $\varepsilon_{r}$ such that, for any $U_{1}$, $U_{2}\in O_{\varepsilon_{r}}(U_{+})$
with $U_{2}=\Phi(0,\beta_{2},\beta_{3},\beta_{4};U_{1})$, the equation{\rm :}
\begin{equation}
	\Phi^{(3)}(\delta_{1},0,0,0;U_1)=\Phi^{(3)}(0,\beta_{2},\beta_{3},\beta_{4};U_{1})
 \label{eqn:BoundaryWeak}
\end{equation}
determines a unique twice differentiable function $\delta_{1}=\delta_{1}(\beta_{2},\beta_{3},\beta_{4},U_{1})$
satisfying
\begin{align*}
	\delta_{1}=K_{b2}\beta_{2}+K_{b3}\beta_{3}+K_{b4}\beta_{4},
\end{align*}
where $K_{bi}$, $i=2,\,3,\,4$, are $C^2$--functions of $(\beta_{2},\beta_{3},\beta_{4},U_{1})$ satisfying
\begin{align*}
	\left\{\begin{aligned}
		&K_{b2}|_{\{\beta_{2}=\beta_{3}=\beta_{4}=0,\,U_{1}=U_{+}\}}=K_{b3}|_{\{\beta_{2}=\beta_{3}=\beta_{4}=0,\,U_{1}=U_{+}\}}=0, \\
		&K_{b4}|_{\{\beta_{2}=\beta_{3}=\beta_{4}=0,\,U_{1}=U_{+}\}}=-1.
	\end{aligned}
	\right.
\end{align*}
\end{lemma}

\begin{proof}
The existence of  $\delta_{1}$ can be proved analogously as in Lemma \ref{lem:WCe}.
Differentiating (\ref{eqn:BoundaryWeak}) with respect to $\beta_{i}$ gives
\begin{align*}
	\frac{\partial \Phi^{(3)}(\delta_{1},0,0,0;U_1)}{\partial\delta_{1}}K_{bi}=\frac{\Phi^{(3)}(0,\beta_{2},\beta_{3},\beta_{4};U_{1})}{\partial\beta_{i}}.
\end{align*}
Using $U_{+}=(u_{+},0,p_{+},\rho_{+})^\top$, (\ref{eqn:Eigenvec})--(\ref{eqn:Coeffi}), and (\ref{eqn:WeakWaveDiff})--(\ref{eqn:WeakWavecurv}), we obtain our results.
\end{proof}

Furthermore, the following estimates of interactions with the presence of a strong shock are needed;
see \cite{Chen2008, Chen2006} for their proofs.

\begin{lemma}
	 \label{lem:Isr}
There exists a positive constant $\varepsilon_{1}$ such that,
for any $U_{l}\in O_{\varepsilon_{1}}(U_{-})$ and $U_{m}$, $U_{r}\in O_{\varepsilon_{1}}(U_{+})$
with $U_{m}=\Phi(0,\alpha_2,\alpha_3,\alpha_{4};G(s;U_{l}))$ and $U_{r}=\Phi(\beta_{1},0,0,0;U_{m})$, the equation{\rm :}
	\begin{equation}
		\Phi(0,\delta_{2},\delta_{3},\delta_{4};G(s';U_{l}))=\Phi(\beta_{1},0,0,0;\Phi(0,\alpha_2,\alpha_3,\alpha_{4};G(s;U_{l})))
  \label{eqn:WeakRefelctStrong}
\end{equation}
determines twice differentiable functions $(s',\delta_{2},\delta_{3},\delta_{4})$ with
	\begin{align*}
		s'=s+K_{s1}\beta_{1},\qquad \delta_{2}=\alpha_2+K_{s2}\beta_{1},
    \qquad \delta_{3}=\alpha_3+K_{s3}\beta_{1},
    \qquad \delta_{4}=\alpha_{4}+K_{s4}\beta_{1}.
	\end{align*}
	Moreover,
	\begin{align*}
		|K_{s4}|\big|_{\{\alpha_{4}=\beta_{1}=0,\,s=s_{0},\,U_{l}=U_{-}\}}
  =\left|\dfrac{\lambda_{1}(U_{+})-s_{0}}{\lambda_{4}(U_{+})-s_{0}}\right|\,
  \left|\dfrac{s_{0}u_{-}N-u_{+}\lambda_{4}(U_{+})M}{s_{0}u_{-}N+u_{+}\lambda_{4}(U_{+})M}\right|<1,
	\end{align*}
	and $|K_{si}|$ are bounded for $i=1,\,2,\,3$, where
	\begin{align*}
		M=\dfrac{c_{+}^2}{\gamma-1}(2u_{+}-u_{-})+u_{+}^2(u_{+}-u_{-}),\qquad
		N=-\dfrac{c_{+}^2}{\gamma-1}<0.
	\end{align*}
In particular, the following holds{\rm :}
	\begin{align*}
		|K_{s4}|\big|_{\{\alpha_{4}=\beta_{1}=0,\,s=s_{0},\,U_{l}=U_{-}\}}\left|\dfrac{\lambda_{4}(U_{+})-s_{0}}{\lambda_{1}(U_{+})-s_{0}}\right|
=\left|\dfrac{s_{0}u_{-}N-u_{+}\lambda_{4}(U_{+})M}{s_{0}u_{-}N+u_{+}\lambda_{4}(U_{+})M}\right|<1.
	\end{align*}
\end{lemma}

\begin{lemma}
	\label{lem:Isrbelow}
 There exists a positive constant $\varepsilon_{2}$ such that,
 for any $U_{l}$, $U_{m}\in O_{\varepsilon_{2}}(U_{-})$ and $U_{r}\in O_{\varepsilon_{2}}(U_{+})$
 with $U_{m}=\Phi(\alpha_{1},\alpha_{2},\alpha_{3},\alpha_{4};U_{l})$ and $U_{r}=\Phi(0,\beta_{2},\beta_{3},\beta_{4};G(s;U_{m}))$,
 the equation{\rm :}
	\begin{equation}
		\Phi(0,\delta_{2},\delta_{3},\delta_{4};G(s';U_{l}))=\Phi(0,\beta_{2},\beta_{3},\beta_{4};G(s;\Phi(\alpha_{1},\alpha_{2},\alpha_{3},\alpha_{4};U_{l})))\label{eqn:WeakRefelctStrongbelow}
	\end{equation}
	determines twice differentiable functions $(s',\delta_{2},\delta_{3},\delta_{4})$ with
\begin{align*}
s'=s+\sum_{i=1}^{4}\hat{K}_{1i}\alpha_{i},\quad \delta_{2}=\beta_{2}+\sum_{i=1}^{4}\hat{K}_{2i}\alpha_{i},
\quad	\delta_{3}=\beta_{3}+\sum_{i=1}^{4}\hat{K}_{3i}\alpha_{i}, \quad \delta_{4}=\beta_{4}+\sum_{i=1}^{4}\hat{K}_{4i}\alpha_{i},
	\end{align*}
	where $|\hat{K}_{ji}|$, $i,j=1, 2, 3, 4$, are bounded, depending only on $U_{-}$, $U_{+}$ and
 system \eqref{eqn:TwoDFEuler}.
\end{lemma}

\section{Construction of approximate solutions}\label{sect-appro}
In this section, a modified wavefront tracking algorithm is developed to construct approximate boundaries and solutions.
Moreover, some necessary estimates are given for the boundary value problem (\ref{eqn:TwoDFEuler}) and (\ref{eqn:ConseForm})--(\ref{1.9a}). \par

We choose
\begin{equation}
\hat\lambda>\sup\{\lambda_4(U)\,:\, U\in O_{\varepsilon_{0}}(U_{-})\cup O_{\varepsilon_{0}}(U_{+})\},\label{eqn:nonphyspeed}
\end{equation}
where $\varepsilon_0$ is a small constant satisfying
\begin{equation}
0<\varepsilon_0<\min\{\varepsilon_{p},\varepsilon_s,\varepsilon_{r},\varepsilon_{1},\varepsilon_{2}\}.\label{eqn:epsilon0}
\end{equation}
For any $\mu>0$, there exist $\Delta x>0$ and $\Delta y=2\hat\lambda\Delta x$ such that
$p_b$ and $U_{\infty}$ are approached by piecewise constant functions $p^{\mu,\Delta x}_{b}$
and $U^{\mu,\Delta x}_{\infty}$, respectively:
\begin{align*}
&p^{\mu,\Delta x}_{b}(x)= p^{\mu,\Delta x,h}_{b}
  &&\mbox{for $(h-1)\Delta x\leq x<h\Delta x,\ h=1,\,2,\,\cdots,\,N_0$},\\
&U^{\mu,\Delta x}_{\infty}(y)=U^{\mu,\Delta x,l}_{\infty}
   &&\mbox{for $\Delta y\leq x<(l+1)\Delta y,\ l=-1,\,-2,\,\cdots,\,-N_1$},
\end{align*}
with
\begin{align*}
&\text{T.V.}(p^{\mu,\Delta x}_{b})\leq \text{T.V.}(p_b)=\text{T.V.}(\tilde{p}_{b}),&&
\|p^{\mu,\Delta x}_{b}(\cdot)-p_b(\cdot)\|_{ L^1(\mathbb{R}_+)}<\mu,\\
&\text{T.V.}(U^{\mu,\Delta x}_{\infty})\leq \text{T.V.}(U_{\infty})=\text{T.V.}(\widetilde{U}_{\infty}),
&&
\|U^{\mu,\Delta x}_{\infty}(\cdot)-U_{\infty}(\cdot)\|_{ L^1(\mathbb{R}_-)}<\mu.
\end{align*}
We construct approximate solutions $U^{\mu,\Delta x}(x,y)$ of the Riemann problem
in a forward neighborhood of the origin.
In our construction, we also obtain
an approximate strong shock
$\mathtt{S}^{\mu,\Delta x}=\{(x,\chi^{\mu,\Delta x})\,:\, x\ge0\}$
and an approximate boundary
$\Gamma^{\mu,\Delta x}=\{(x,b^{\mu,\Delta x})\,:\, x\ge0\}$
in the forward neighborhood of the origin.\par

Then the approximate solutions are piecewise constant vector-valued functions that are separated by wavefronts.
We let every wavefront travel freely until it collides with other wavefronts or the boundary.
A new Riemann problem arises when different wavefronts collide and interact at some point. Also, on the approximate boundary,
due to the change of $p^{\mu,\Delta x}_{b}$, new Riemann problems come out at $x=h\Delta x$ for $h\in\mathbb{N}_+$. \par

To construct approximate solutions to those new Riemann problems, following \cite{Bressan2000,Chen2020,Kuang2020},
we introduce two kinds of Riemann solvers, each of which contains shocks, contact discontinuities, rarefaction fronts, and non-physical fronts. \par

Let $\delta=\delta(\mu)>0$ be a parameter that is larger than the maximum strength of rarefaction fronts, whose value is
determined later.
Moreover, we define non-physical fronts to be of family $5$ of speed $\hat\lambda$
and use $U_{r}=T_5(\epsilon)(U_{l})$ to indicate that the below state $U_{l}$ is connected with
the above state $U_{r}$ by a non-physical front of strength $\epsilon=|U_{r}-U_{l}|$.
Non-physical waves also belong to weak waves.\\

$\bullet$ \textit{Accurate Riemann solver}. The accurate Riemann solver provides us with an approximate solution of
the Riemann problem (\ref{eqn:ConseForm}), where the rarefaction region in the real solution of the Riemann problem
is replaced by piecewise constants that are separated by rarefaction wavefronts. To be specific,
suppose that a $j$-wave $\alpha_j$, originated at $(\bar x,\bar y)$, is a rarefaction wave connecting two constant states $U_{m_{j-1}}$ and $U_{m_j}$, $j=1,\,4$. Taking the minimal $n\in\mathbb{N}_+$ such that $\alpha_j<n\delta$,
we then define
\begin{align*}
U^\delta_{j,\alpha_j}(x,y):=
U_{m_{j-1}}+\sum_{i=1}^n(U_{m_{j-1},i}-U_{m_{j-1},i-1})H(y-\bar y-\lambda_j(U_{m_{j-1},i-1})(x-\bar x)),\quad\,
j=1, 4,
\end{align*}
where
\begin{align*}
U_{m_{j-1},i}=R_j(\frac{i\alpha_1}{n})(U_{m_{j-1}}),\qquad i=0,1,\cdots,n,\ j=1,\,4,
\end{align*}
and $H$ is the Heaviside function. To obtain an accurate Riemann solver $U^\delta_A(U_{l},U_{r};x,y)$,
we use $U^\delta_{j,\alpha_j}(x,y)$ to substitute
$R_j(\frac{x}{t}-\lambda_j(U_{m_{j-1}}))(U_{m_{j-1}})$ in the real solution (\ref{eqn:SimRiemannSolv}) when $(x,y)$ belongs to the rarefaction region
$\sigma_j^-(x-\bar x)<y-\bar y<\sigma_j^+(x-\bar x)$ for $j=1,\,4$.\par

\smallskip
$\bullet$ \textit{Simplified Riemann solver}. We have
several different cases.

\smallskip
\noindent
\textit{Case $1$}: \textit{Interactions of weak waves}. Assume that two weak waves collide at $(\bar x,\bar y)$ with below state $U_{l}$, middle state $U_{m}=\Phi_i(\beta;U_{l})$,
and above state $U_{r}=\Phi_j(\alpha;U_{m})$, $\,i, j=1,2,3,4$.
We introduce the auxiliary above state
\begin{align*}
U'_{r}=\left\{\begin{aligned}
	&\Phi_j(\beta;\Phi_i(\alpha;U_{l})),&& j>i,\\
	&\Phi_j(\alpha+\beta;U_{l}),&& j=i.\\
\end{aligned}
\right.
\end{align*}
Then the simplified Riemann solver in a forward neighbourhood of $(\bar x,\bar y)$ is defined as
\begin{align*}
U^\delta_{S}(x,y)=\left\{\begin{aligned}
	&U^\delta_{A}(U_{l},U'_{r};x,y),\ &&y-\bar y<\hat\lambda(x-\bar x), \\
	&U_{r},\ &&y-\bar y>\hat\lambda(x-\bar x).
\end{aligned}
\right.
\end{align*}

\noindent
\textit{Case $2${\rm :} Interactions of a non-physical wave with another weak wave}. Assume that a non-physical wave collides another weak wave at $(\bar x,\bar y)$ with below state $U_{l}$, middle state $U_{m}=\Phi_5(\epsilon;U_{l})$, and above state $U_{r}=\Phi_j(\alpha;U_{m})$
for $j=1, 2, 3, 4$.
We choose the auxiliary above state as
\begin{align*}
U'_{r}=\Phi_j(\alpha;U_{l}),
\end{align*}
and the simplified Riemann solver as
\begin{align*}
U^\delta_{S}(x,y)=\left\{\begin{aligned}
	&U^\delta_{A}(U_{l},U'_{r};x,y),\ &&y-\bar y<\hat\lambda(x-\bar x), \\
	&U_{r},\ &&y-\bar y>\hat\lambda(x-\bar x).
\end{aligned}
\right.
\end{align*}

\noindent
\textit{Case $3${\rm :} Interactions of a weak wave with the strong shock from above}.
Assume that a weak wave collides the strong shock at $(\bar x,\bar y)$ with below state $U_{l}$, middle state $U_{m}=G(s;U_{l})$, and above state $U_{r}=\Phi_j(\alpha;U_{m})$ for $j=1, 2, 3, 4$.
We define the simplified Riemann solver as
\begin{align*}
U^\delta_{S}(x,y)=\left\{\begin{aligned}
	&U_{l},\ &&y-\bar y<s(x-\bar x), \\
	&U_{m},\ &&\hat\lambda(x-\bar x) > y-\bar y> s(x-\bar x), \\
	&U_{r},\ &&y-\bar y>\hat\lambda(x-\bar x).
\end{aligned}
\right.
\end{align*}

\textit{Case $4${\rm :} Interactions of a weak wave with the strong shock from below}.
Assume that a weak wave collides the strong shock at $(\bar x,\bar y)$ with below state $U_{l}$, middle state $U_{m}=\Phi_j(\alpha;U_{l})$, and above state $U_{r}=G(s;U_{m})$ for $j=1, 2, 3, 4$.
We define the simplified Riemann solver as
\begin{align*}
U^\delta_{S}(x,y)=\left\{\begin{aligned}
	&U_{l},\ &&y-\bar y<s(x-\bar x), \\
	&G(s;U_{l}),\ &&\hat\lambda(x-\bar x)> y-\bar y> s(x-\bar x), \\
	&U_{r},\ &&y-\bar y>\hat\lambda(x-\bar x).
\end{aligned}
\right.
\end{align*}

Now we can define the approximate solutions inductively.
For simplicity of notation, in the section, we omit superscript $\delta$
and always take
$\|(\tilde{p}_{b}, \widetilde{U}_{\infty})\|_{ L^{\infty}\cap BV}$
small enough such that the conditions of Lemmas \ref{Lem:strongshock}--\ref{lem:Isrbelow} are satisfied (as we will prove later).
Given $p^{\mu,\Delta x,1}_{b}$, Lemma \ref{lem:SBe} provides us
with $s^{\mu,\Delta x}_{1}$ such that
\begin{align*}
G^{(3)}(s^{\mu,\Delta x}_{1};\,U^{\mu,\Delta x,-1}_{\infty})=p^{\mu,\Delta x,1}_{b}.
\end{align*}
For $x\in[0,\Delta x)$, define
\begin{align*}
&U^{\mu,\Delta x}_{b}(x)=G(s^{\mu,\Delta x}_{1};\,U^{\mu,\Delta x,-1}_{\infty}),
\qquad
b^{\mu,\Delta x}(x)=\int_0^x\frac{v^{\mu,\Delta x}_{b}(t)}{u^{\mu,\Delta x}_{b}(t)}\,{\rm d}t,\\
&s^{\mu,\Delta x}(x)=s^{\mu,\Delta x}_{1},\qquad \chi^{\mu,\Delta x}(x)=\int_0^xs^{\mu,\Delta x}(t)\,{\rm d}t.
\end{align*}
Then we solve the standard Riemann problems at points $(0,l\Delta y)$,
$l\in\mathbb{N}_{-}$, with $U^{\mu,\Delta x,l}_{\infty}$
and $U^{\mu,\Delta x,l-1}_{\infty}$ being the corresponding above
and below states. Thus,
we can construct the approximate solutions and the approximate strong shocks on the region:
\begin{align*}
\Omega^{\mu,\Delta x,1}
=
\big\{(x,y)\,:\,y<b^{\mu,\Delta x}(x),\ x\in[0,\Delta x)\big\}.
\end{align*}
\begin{figure}
\begin{center}
	\includegraphics[width=\textwidth]{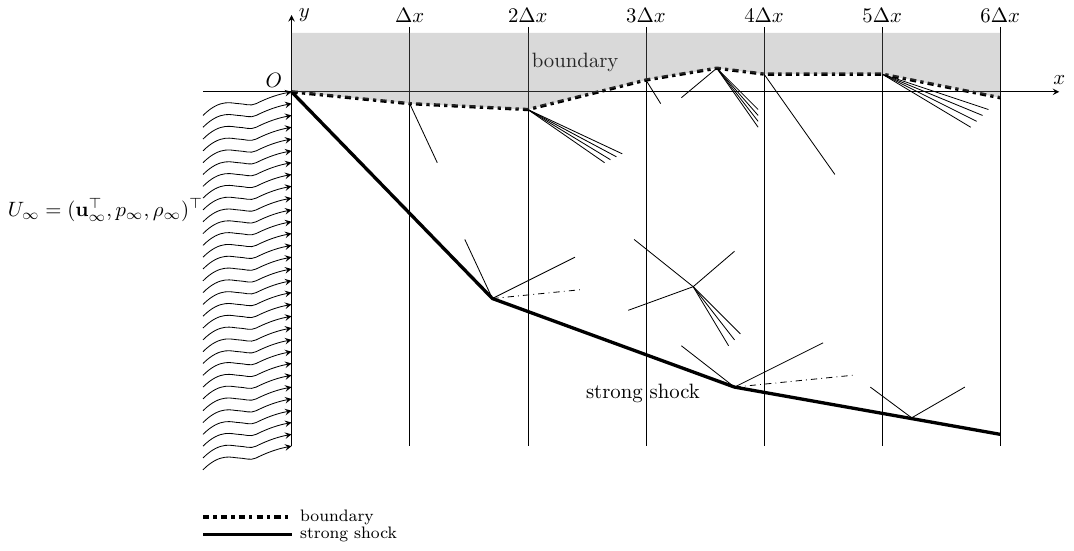}
\end{center}
\caption{Approximate solutions}
\label{fig:4.1}
\end{figure}

Suppose now that our approximate solutions $U^{\mu,\Delta x}(x,y)$
are constructed on
\begin{align*}
\bigcup_{k=1}^{h}\Omega^{\mu,\Delta x,k}
:=
\bigcup_{k=1}^{h}\big\{(x,y)\,:\, y<b^{\mu,\Delta x}(x),\ x\in[(k-1)\Delta x,k\Delta x)\big\},
\end{align*}
and contain the jumps of rarefaction fronts $\mathcal{R}$, weak shock fronts $\mathcal{S}$,
contact discontinuities $\mathcal{C}$, non-physical fronts $\mathcal{NP}$, and strong shock fronts $\mathcal{S}_s$.
We also suppose that the approximate solutions are defined on the approximate boundaries $b^{\mu,\Delta x}(x)$
as $U^{\mu,\Delta x}_{b}=U^{\mu,\Delta x}(x,b^{\mu,\Delta x}(x))$ for $x\in[0,h\Delta x)$.
Moreover, in $\Omega^{\mu,\Delta x,k}$, the number of these wavefronts is finite for $k=1,\cdots,h$.
\par

To extend our solutions, we generally let wavefronts travel ahead until they collide with another wavefront.
In order to avoid the cases that more than two wavefronts collide at one point and that more than one wavefront
interact with the boundary or the strong shock, we can modify the speeds of these wavefronts slightly such that
the difference between the modified speeds and the original speeds is no more than $\mu$. \par

When two wavefronts collide at some point $(\bar x,\bar y)$ for $\bar x\in[h\Delta x,(h+1)\Delta x)$,
three states separated by the wavefronts, from below to above, are labeled as $U_{l}$, $U_{m}$, and $U_{r}$.
To make the number of wavefronts remain finite on region
$\{x<h\Delta x, \,\, h\in\mathbb{N}_+\}$,
we need to use the simplified Riemann solver
to continue our construction; see \cite{Bressan2000}.\par

Let $\nu>0$ be a threshold parameter to determine when an accurate or a simplified Riemann solver is applied. With little abuse of notation,
a front itself is denoted by $\alpha$, while its strength is
denoted by $|\alpha|$.

\smallskip
\textit{Case $1${\rm :} $\alpha$ and $\beta$ are weak waves colliding at $(\bar x,\bar y)$}. We solve the Riemann problem $(\ref{eqn:RiePro})$ as follows:\\[2mm]
$\bullet$ When $\alpha$ and $\beta$ are physical with $|\alpha\beta|>\nu$, we apply the accurate Riemann solver.\\
$\bullet$ When $\alpha$ and $\beta$ are physical with $|\alpha\beta|\leq\nu$, or one of them is non-physical, we apply the simplified Riemann solver.

\smallskip
\textit{Case $2${\rm :} A weak wave $\alpha$ interacts with a strong shock $s$
at $(\bar x,\bar y)$}. We solve the corresponding Riemann problem as follows:\\[2mm]
$\bullet$ When $\alpha$ is physical with $|\alpha|>\nu$, we apply the accurate Riemann solver.\\
$\bullet$ When $\alpha$ is non-physical or $|\alpha|\leq\nu$, we apply the simplified Riemann solver.

\smallskip
\textit{Case $3${\rm :} When the approximate pressure changes at $(h\Delta x,b^{\mu,\Delta x}(h\Delta x))$, we solve the Riemann problem $(\ref{eqn:RiePro})$ as follows}: Let
\begin{align*}
U_{1}=U^{\mu,\Delta x}(h\Delta x-,b^{\mu,\Delta x}(h\Delta x-)-),\qquad p_2=p^{\mu,\Delta x,h}_{b}.
\end{align*}
Then Lemma \ref{lem:WCe} provides us with $\delta_1$ such that
\begin{align*}
	\Phi^{(3)}(\delta_1,0,0,0;U_{l})=p^{\mu,\Delta x,h}_{b}.
\end{align*}
We define $U^{\mu,\Delta x}_{b}=\Phi(\delta_1,0,0,0;U_{l})$ and always apply the accurate Riemann solver.\\

\textit{Case $4${\rm :}. When a weak physical wave $\alpha$ separating states $U_{r}$ {\rm (}the above one{\rm )} and $U_{l}$ hits the boundary}:

\smallskip
\noindent
$\bullet$ If $|\alpha|>\nu$, let
\begin{align*}
	U_{1}=U_l,\quad p_2=p^{\mu,\Delta x,h}_{b}.
\end{align*}
$\,\,\,\,$ Then Lemma \ref{lem:WBe} provides us with $\delta_1$ such that
\begin{align*}
	\Phi^{(3)}(0,0,0,\alpha;U_{l})=\Phi^{(3)}(\delta_1,0,0,0;U_{l}).
\end{align*}
$\,\,\,\,$ Define $U^{\mu,\Delta x}_{b}=\Phi(\delta_1,0,0,0;U_{l})$, where an accurate Riemann solver is used.\\
$\bullet$ If $|\alpha|\leq\nu$, we let $\alpha$ cross the boundary and change $U^{\mu,\Delta x}_{b}$ from $U_{r}$ to $U_{l}$.\\

\textit{Case 5}. When $U_{r}$ (the above) and $U_{l}$ are separated by a non-physical weak wave that
hits the boundary at $(x,b^{\mu,\Delta x}(x))$ for $x\in((h-1)\Delta x,h\Delta x)$,
we allow it cross the boundary and define $U^{\mu,\Delta x}_{b}=U_{l}$. \par

\smallskip
Finally, let
\begin{align*}
&\chi^{\mu,\Delta x}(x)=\int_0^xs^{\mu,\Delta x}(t){\rm d}t,\,\, b^{\mu,\Delta x}(x)=\int_{0}^x\frac{v^{\mu,\Delta x}_{b}}{u^{\mu,\Delta x}_{b}}(t)\,{\rm d}t
\qquad\,\mbox{for $x\in[h\Delta x,(h+1)\Delta x)$},
\end{align*}
and let our approximate solutions $U^{\mu,\Delta x}$ have been constructed on the region
\begin{align*}
\Omega^{\mu,\Delta x,h}:=
\big\{(x,y)\,:\,y<b^{\mu,\Delta x}(x),\ x\in[h\Delta x,(h+1)\Delta x)\big\}
\end{align*}
with approximate strong shocks $\mathtt{S}^{\mu,\Delta x}$,
where $y=b^{\mu,\Delta x}(x)$ is our approximate boundary. In addition, on approximate boundaries, we define
\begin{align*}
U^{\mu,\Delta x}(x,b^{\mu,\Delta x}(x))=U^{\mu,\Delta x}_{b}(x)
\qquad\mbox{for $x\in[0,(h+1)\Delta x)$}.
\end{align*}

To show these approximate solutions can be well defined in
\begin{align*}
\Omega^{\mu,\Delta x}\cup \Gamma^{\mu,\Delta x}
:=
\bigcup_{h=1}^\infty\Omega^{\mu,\Delta x,h}\cup\bigcup_{h=1}^\infty \Gamma^{\mu,\Delta x,h},
\end{align*}
via the steps exhibited above, we need to prove that there exists a uniform bound, where $\Gamma^{\mu,\Delta x,h}=\{(x,b^{\mu,\Delta x}(x))\, :\, (h-1)\Delta x\leq x< h\Delta x\}$ for $h\in\mathbb{N}_{+}$.
Assume that $U^{\mu,\Delta x}$ has been defined on
$\big(\bigcup_{k=1}^h\Omega^{\mu,\Delta x,k}\big)\bigcup\big(\bigcup_{k=1}^{h} \Gamma^{\mu,\Delta x,h}\big)$ and, furthermore, the following conditions hold:

\begin{itemize}[leftmargin=5em,,rightmargin=5em]
\item[$\rm H_1(h)$:]  There is a strong 1-shock
$$
\mathtt{S}^{\mu,\Delta x,k}=\big\{(x,\chi^{\mu,\Delta x}(x))\,:\, (k-1)\Delta x\leq x< k\Delta x\big\}
$$
in $\Omega^{\mu,\Delta x,k}$ for $1\leq k\leq h$, dividing $\Omega^{\mu,\Delta x,k}$ into $\Omega^{\mu,\Delta x,k}_{-}$ and $\Omega^{\mu,\Delta x,k}_{+}$, where $\chi^{\mu,\Delta x}(x)=\int_0^xs^{\mu,\Delta x}(t)\,{\rm d}t$, and $\Omega^{\mu,\Delta x,k}_{+}$ is the part bounded by $\mathtt{S}^{\mu,\Delta x}$ and $\Gamma^{\mu,\Delta x}$;
\item[$\rm H_2(h)$:] $U^{\mu,\Delta x}|_{\Omega^{\mu,\Delta x,k}_{-}}\in O_{\varepsilon_{0}}(U_{-})$ and $U^{\mu,\Delta x}|_{\Omega^{\mu,\Delta x,k}_{+}}\in O_{\varepsilon_0}(U_{+})$ in each $\Omega^{\mu,\Delta x,k}$ for $1\leq k\leq h$, and $U^{\mu,\Delta x}|_{\Gamma^{\mu,\Delta x}}=U^{\mu,\Delta x}_{b}(x)\in O_{\varepsilon_0}(U_+)$ for $x\in((h-1)\Delta x,h\Delta x)$, where $\varepsilon_0>0$ is introduced in (\ref{eqn:epsilon0});
\item[$\rm H_3(h)$:] $U^{\mu,\Delta x}|_{\bigcup_{k=1}^h\Omega^{\mu,\Delta x,k}}$ is piecewise constant and contains the jumps of rarefaction fronts $\mathcal{R}$, weak shock fronts $\mathcal{S}$, contact discontinuities $\mathcal{C}$, non-physical fronts $\mathcal{NP}$, and the number of these wavefronts is finite, saying $N_h$.
\end{itemize}

It suffices to prove that $U^{\mu,\Delta x}$ can be extended to $\Omega^{\mu,\Delta x,h+1}$ and satisfy $\rm H_1(h+1)$, $\rm H_2(h+1)$, and $\rm H_3(h+1)$. To this end, we introduce a Glimm-type functional and prove that it is non-increasing, which ensures the smallness of total variations of approximate solutions. The following lemmas are needed in our proofs.

\begin{lemma}
\label{Lem:BVe} The following statements hold{\rm :}
\begin{itemize}
	\item[\rm(i)] If $U_2=\Phi(\alpha_1,\alpha_2,\alpha_3,\alpha_4;U_1)$ with $U_1$, $U_2\in O_{\varepsilon_0}(U_{-})$ or $U_1$, $U_2\in O_{\varepsilon_0}(U_{+})$, then
	\begin{align*}
		|U_1-U_2|\leq C_1(|\alpha_1|+|\alpha_2|+|\alpha_3|+|\alpha_4|);
	\end{align*}
	\item[\rm(ii)] If $U_2=\Phi(\alpha_1,\alpha_2,\alpha_3,\alpha_4;U_1)$, $U_3=G(s;U_1)$, and $U_4=G(s;U_2)$ with $U_1$, $U_2\in O_{\varepsilon_0}(U_{-})$ and $U_3$, $U_4\in O_{\varepsilon_0}(U_{+})$, then
	\begin{align*}
		|U_3-U_4|\leq C_2(|\alpha_1|+|\alpha_2|+|\alpha_3|+|\alpha_4|),
	\end{align*}
\end{itemize}
where $C_1>0$ and $C_2>0$ are constants depending only on $U_{-}$, $U_{+}$, and  system \eqref{eqn:TwoDFEuler}.
\end{lemma}

\begin{lemma}
\label{Lem:NPw} For any $U_1$,  $U_2$, $U_3\in O_{\varepsilon_{0}}(U_{-})$ or $U_1$, $U_2$, $U_3\in O_{\varepsilon_{0}}(U_{+})$, suppose that
\begin{itemize}
	\item [\rm(i)] $U_2=\Phi_{i}(\beta_j;\Phi_{i}(\alpha_i;U_1))$ and $U_3=\Phi_{i}(\alpha_i+\beta_j;U_1)$ for $i\in\{1,\,2,\,3,\,4\},$
	\item[\rm(ii)] $U_2=\Phi_{j}(\beta_j;\Phi_{i}(\alpha_i;U_1))$ and $U_3=\Phi_{i}(\alpha_i;\Phi_{j}(\beta_j;U_1))$ for $i\neq j$ and $i,\,j\in\{1,\cdots,5\}$.
\end{itemize}
Then $|U_3-U_2|=O(1)|\alpha_i||\beta_j|$.
\end{lemma}

\smallskip
\begin{definition}[Approaching]\
\begin{itemize}
\item[\rm(i)] Suppose that two fronts $\alpha$ and $\beta$ are located at points $y_{\alpha}<y_{\beta}$ and belong to the characteristic families $i_{\alpha}$, $i_{\beta}\in\{1,\cdots,5\}$, respectively. Then we say that they are approaching if $i_{\alpha}>i_{\beta}$ or if $i_{\alpha}=i_{\beta}$ and one of them is a shock. In this case, we denote the approaching relation by $(\alpha,\beta)\in\mathcal{A}$.
\item[\rm(ii)] For any weak front $\alpha\in\bigcup_{h=1}^{\infty}\Omega^{\mu,\Delta x}_{h,+}$ of family $1$, or $\alpha\in\bigcup_{h=1}^{\infty}\Omega^{\mu,\Delta x}_{h,-}$ of family $i$, $i\in\{1,\cdots,5\}$, we say that it approaches the strong shock and write $\alpha\in\mathcal{A}_s$ in this case.
\item[\rm(iii)] For any weak front $\alpha$ of family $4$ or $5$, we say that it approaches the boundary and write $\alpha\in\mathcal{A}_b$.
\end{itemize}
\end{definition}

Similar to \cite{Chen2008}, we define the following Glimm-type functionals.
Let the weighted strength for an $i$-weak wave $\alpha$ be
\begin{equation}
b_{\alpha}=\left\{
\begin{aligned}
&K_{-}\alpha\quad &\mbox{for $\alpha\in\bigcup_{h=1}^{\infty}\Omega^{\mu,\Delta x}_{-}$},\\
&\alpha &\mbox{for $\alpha\in\bigcup_{h=1}^{\infty}\Omega^{\mu,\Delta x}_{+}$},
\end{aligned}\right.\label{eqn:weistren}
\end{equation}
where $K_{-}=2\max_{1\leq i\leq4,2\leq j\leq4}\{|\hat{K}_{ij}|\}+4C_2$ with coefficients $\hat{K}_{ij}$ in Lemma \ref{lem:Isrbelow}.
Then, for each $x\in[h\Delta x,(h+1)\Delta x)$ with $h\in\mathbb{N}_+$, the weighted total strength of weak waves in $U^{\mu,\Delta x}(x,\cdot)$ is defined to be
\begin{align*}
L(x)=\sum_{\alpha}|b_{\alpha}|.
\end{align*}
The interaction potential is defined as
\begin{equation}
Q(x)=K_{0}\sum_{i>h}\omega_i+K_s\sum_{\alpha\in\mathcal{A}_{s}}|b_{\alpha}|+\sum_{\beta\in\mathcal{A}_{b}}|b_{\beta}|+K\sum_{(\alpha,\beta)\in\mathcal{A}}|b_{\alpha}||b_{\beta}|,\label{eqn:GlimPotential}
\end{equation}
where $\omega_i=|p^{\mu,\Delta x,i+1}_{b}-p^{\mu,\Delta x,i}_{b}|$, $K_{0}$, $K_s$, and $K$ are constants that need to be specified later.

\smallskip
\begin{definition}
\label{Def:GTF}
For each $x\in[h\Delta x,(h+1)\Delta x)$ with $h\in\mathbb{N}_+$,
define
	\begin{align*}
		F(x)=L(x)+\mathcal{K}Q(x)+|U_{\diamond}(x)-U_{\infty}^{-}|+|U^{\diamond}(x)-U_{\infty}^{+}|
	\end{align*}
with $\mathcal{K}>0$ to be determined later, where vector $U_{\diamond}$
{\rm (}$U^{\diamond}${\rm )} is the below $($above$)$ state of the large shock at \textit{time} $x$, and $U_{\infty}^{-}$ {\rm (}$U_{\infty}^{+}${\rm )} is the below $($above$)$ state of the large shock at $x = 0$.
\end{definition}

When two wavefronts collide (a wavefront hits the boundary or the boundary pressure changes) at $(\tau,\xi)$,
we have the following proposition.

\begin{proposition}\label{prop:befIm}
There are constants $K_{0}$, $K_s$, $K$,
and $\mathcal{K}$ so that there exists $\varepsilon>0$ such that,
when $F(\tau-)\leq\varepsilon$,
	\begin{align*}
		F(\tau+)\leq F(\tau-).
	\end{align*}
	\end{proposition}

\begin{proof}
Assume that, on $x=\tau$, only one interaction happens.
Our proof is divided into the following five cases, according to the location where an interaction takes place.
Let $C>0$ be a universal constant that may vary at different occurrences,
depending only on $U_-$, $U_+$, and system (\ref{eqn:TwoDFEuler}).

 \smallskip\noindent
 \Case\label{Case1} {\it Interior interactions between weak waves}.
 A weak wavefront $\alpha_i$ of family $i$ interacts the other weak wavefront $\beta_j$
 of $j$-family at $(\tau,\xi)$ for $i, j\in\{1,\cdots,5\}$.
 Then Lemma \ref{lem:Iww} leads to
\begin{align*}
&L(\tau+)-L(\tau-)\leq C|\alpha_i||\beta_j|,\\
&\mathcal{K}Q(\tau+)-\mathcal{K}Q(\tau-)\leq \mathcal{K}\big(-K(1-CL(\tau-))+(K_s+1)C\big)|\alpha_i||\beta_j|,
\end{align*}
which implies
\begin{align*}
F(\tau+)-F(\tau-)\leq\mathcal{K}\big(-K(1-CL(\tau-))+C(K_s+1)\big)|\alpha_i||\beta_j|+ C|\alpha_i||\beta_j|.
\end{align*}

\begin{figure}
		\begin{center}
			\includegraphics{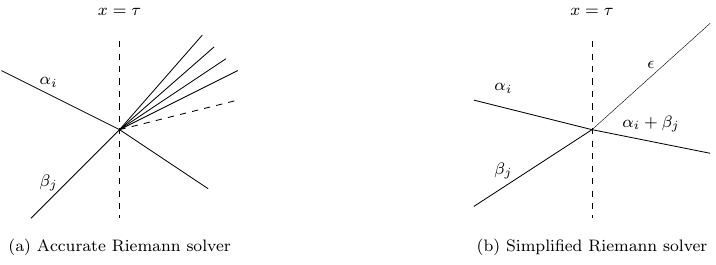}
		\end{center}
		\caption{Case 1. {\it Interior interactions between weak waves}}
		\label{fig:5.3}
\end{figure}

\Case\label{Case2} {\it A weak physical wavefront collides the strong shock}. A weak physical wavefront $\beta_1$ of $1$-family
interacts the shock wavefront $s$ at $(\tau,\xi)$ from above.

\smallskip
$\bullet$ When an accurate Riemann solver is used, then Lemmas \ref{lem:Isr} and \ref{Lem:BVe} imply that
\begin{align*}
		&L(\tau+)-L(\tau-)\leq\sum_{i=2}^{4}|K_{si}||\beta_1|-|\beta_1|,\\
		&\mathcal{K}Q(\tau+)-\mathcal{K}Q(\tau-)
        \leq \mathcal{K}\Big(KL(\tau-)\sum_{i=2}^{4}|K_{si}|-K_s+|K_{s4}|\Big)|\beta_1|,\\[2mm]
		&|U_{\diamond}(\tau-)-U_{\infty}^{-}|+|U^{\diamond}(\tau-)-U_{\infty}^{+}|
		-|U_{\diamond}(\tau+)-U_{\infty}^{-}|+|U^{\diamond}(\tau+)-U_{\infty}^{+}|\\
		&\,\,\leq|U_{\diamond}(\tau-)-U_{\diamond}(\tau+)|+|U^{\diamond}(\tau-)-U^{\diamond}(\tau+)|
         \leq C_{1}\Big(\sum_{i=2}^{4}|K_{si}|+1\Big)|\beta_1|.
	\end{align*}
Thus, we have
\begin{align*}
F(\tau+)-F(\tau-)\leq C|\beta_1|-\mathcal{K}\Big(K_s-|K_{s4}|-KL(\tau-)\sum_{i=2}^{4}|K_{si}|\Big)|\beta_1|.
\end{align*}
	
$\bullet$ When a simplified Riemann solver is used, then Lemma \ref{Lem:BVe} indicates
	\begin{align*}
		&L(\tau+)-L(\tau-)=-|\beta_1|+C_{1}|\beta_1|,\\
		&\mathcal{K}Q(\tau+)-\mathcal{K}Q(\tau-)
             \leq -\mathcal{K}K_s|\beta_{1}|,
	\end{align*}
	which gives
	\begin{align*}
		F(\tau+)-F(\tau-)\leq-\big(\mathcal{K}K_s+1-C_{1}\big)|\beta_{1}|.
	\end{align*}\begin{figure}
		\begin{center}
			\includegraphics{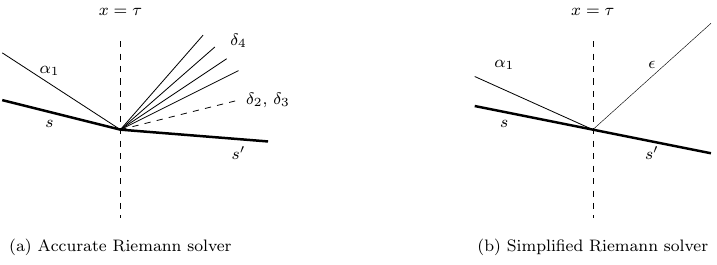}
		\end{center}
		\caption{Case 2. {\it A weak physical wave collides the strong shock from above}}
		\label{fig:5.4}
	\end{figure}

 \Case\label{Case3} {\it A weak wavefront collides the strong shock}. A weak wavefront $\alpha_i$ of $i$-family
 interacts the strong shock wavefront $s$ at $(\tau,\xi)$ from below.

 \smallskip
$\bullet$ When an accurate Riemann solver is used, then Lemmas \ref{lem:Isrbelow} and \ref{Lem:BVe}
 lead to
	\begin{align*}
		&L(\tau+)-L(\tau-)\leq\sum_{i=2}^{4}|\hat{K}_{ij}||\alpha_{i}|-|b_{\alpha_{i}}|,\\
		&\mathcal{K}Q(\tau+)-\mathcal{K}Q(\tau-)
          \leq \mathcal{K}\Big(KL(\tau-)\sum_{j=2}^{4}|\hat{K}_{ji}||\alpha_i|-|b_{\alpha_i}|+|\hat{K}_{4i}||\alpha_{i}|\Big),\\[2mm]
		&|U_{\diamond}(\tau-)-U_{\infty}^{-}|+|U^{\diamond}(\tau-)-U_{\infty}^{+}|
		-|U_{\diamond}(\tau+)-U_{\infty}^{-}|+|U^{\diamond}(\tau+)-U_{\infty}^{+}|\\
		&\,\,\leq |U_{\diamond}(\tau-)-U_{\diamond}(\tau+)|+|U^{\diamond}(\tau-)-U^{\diamond}(\tau+)|\leq C_{1}\Big(\sum_{j=2}^{4}|\hat{K}_{ji}|+1\Big)|\alpha_i|.
	\end{align*}
	Then we obtain
	\begin{align*}
		F(\tau+)-F(\tau-)&\leq -\Big(K_{-}-\sum_{i=2}^{4}|\hat{K}_{ij}|\Big)|\alpha_{i}|
           -\mathcal{K}\Big(K_{-}-KL(\tau-)\sum_{j=2}^{4}|\hat{K}_{ji}|-|\hat{K}_{4i}|\Big)|\alpha_{i}|\\
		&\quad\,\,+ C_{1}\Big(\sum_{j=2}^{4}|\hat{K}_{ji}|+1\Big)|\alpha_i|.
	\end{align*}

	$\bullet$ When a simplified Riemann solver is used, then Lemma \ref{Lem:BVe} indicates
	\begin{align*}
		&L(\tau+)-L(\tau-)=-K_{-}|\alpha_i|,\\[1mm]
		&\mathcal{K}Q(\tau+)-\mathcal{K}Q(\tau-)
            \leq \mathcal{K}\big(KL(\tau-)C_{2}-K_sK_{-}+C_{2}\big)|\alpha_i|,\\[1mm]
		&|U_{\diamond}(\tau-)-U_{\infty}^{-}|+|U^{\diamond}(\tau-)-U_{\infty}^{+}|
		-|U_{\diamond}(\tau+)-U_{\infty}^{-}|+|U^{\diamond}(\tau+)-U_{\infty}^{+}|\\
		&\,\,\leq|U_{\diamond}(\tau-)-U_{\diamond}(\tau+)|+|U^{\diamond}(\tau-)-U^{\diamond}(\tau+)|
		\leq (C_{2}+C_{1})|\alpha_i|,
	\end{align*}
	which gives
	\begin{align*}
		F(\tau+)-F(\tau-)\leq-K_{-}|\alpha_i|-\mathcal{K}\big(K_{-}K_s-KL(\tau-)C_{2}-C_{2}\big)|\alpha_i|+(C_{2}+C_{1})|\alpha_i|.
	\end{align*}\begin{figure}
		\begin{center}
			\includegraphics{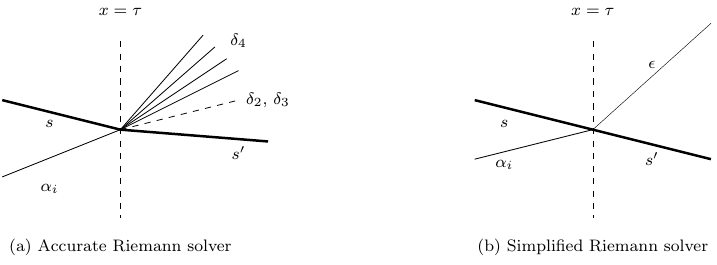}
		\end{center}
		\caption{Case 3. {\it A weak wavefront collides the strong shock from below}}
		\label{fig:5.42}
	\end{figure}
	
	\Case\label{Case4} {\it A weak physical wavefront is generated from the boundary}. A weak physical wavefront $\alpha_1$ of $1$-family
 is generated from the boundary at $(\tau,\xi)=(h\Delta x,b^{\mu,\Delta x}(h\Delta x))$.

 \smallskip
 	$\bullet$ When the accurate Riemann solver is applied, then Lemma \ref{lem:WCe} leads to
	\begin{align*}
		L(\tau+)-L(\tau-)=\,|&\alpha_1|\leq|K_{b}|\omega_h,\\
		\mathcal{K}Q(\tau+)-\mathcal{K}Q(\tau-)
        &\leq \mathcal{K}(-K_{0}\omega_h+KL(\tau-)|\alpha_1|+K_s|\alpha_1|)\\
		&\leq \mathcal{K}\big(-K_{0}+KL(\tau-)|K_{b}|+K_s|K_{b}|\big)\omega_h.
	\end{align*}
	Therefore, we have
	\begin{align*}
		F(\tau+)-F(\tau-)\leq|K_{b}|\omega_h-\mathcal{K}\big(K_{0}-KL(\tau-)|K_{b}|-K_s|K_{b}|\big)\omega_h.
	\end{align*}\begin{figure}[ht]
		\begin{center}
			\includegraphics{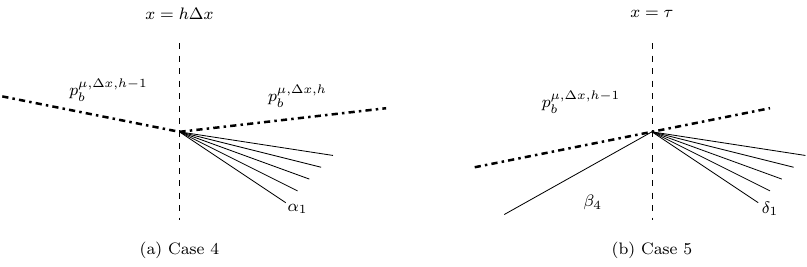}
		\end{center}
		\caption{Case 4 \& 5. {\it Interaction involving a weak physical wave and the boundary}}
		\label{fig:5.5}
	\end{figure}
	
 \Case\label{Case5} {\it A weak physical wavefront hits the boundary}. A weak physical wavefront $\beta_4$ of $4$-family
 hits the boundary at $(\tau,\xi)$.\par

	$\bullet$ When an accurate Riemann solver is used, then Lemma \ref{lem:WBe} implies
	\begin{align*}
		L(\tau+)-L(\tau-)=\,|&\delta_1|-|\beta_4|\leq \big(|K_{b4}|-1\big)|\beta_4|,\\
		\mathcal{K}Q(\tau+)-\mathcal{K}Q(\tau-)
        &\leq \mathcal{K}\big(KL(\tau-)|\delta_1|+K_{s}|\delta_1|-|\beta_{4}|\big)\\
		&\leq \mathcal{K}\big(KL(\tau-)|K_{b4}|+K_{s}|K_{b4}|-1\big)|\beta_{4}|,
	\end{align*}
	which leads to
	\begin{align*}
		F(\tau+)-F(\tau-)\leq \big(|K_{b4}|-1\big)|\beta_4|-\mathcal{K}\big(1-K_{s}|K_{b4}|-KL(\tau-)|K_{b4}|\big)|\beta_4|.
	\end{align*}
	To conclude, from Lemma \ref{lem:WBe}, when $\varepsilon$ is small enough, we may take
	\begin{align*}
		\max\{\dfrac{1}{2},|K_{s4}|\}<K_s<\min\{1,\dfrac{1}{|K_{b4}|}\}.
	\end{align*}
Moreover, we take $K_{0}$, $K$, and $\mathcal{K}$ large enough, and then take $\varepsilon$ smaller if necessary to obtain from the above estimates:
	\begin{align*}
		F(\tau+)\leq F(\tau-).
	\end{align*}
This completes the proof.
\end{proof}

Using a similar argument as in \cite{Chen2008}, we can obtain
\begin{corollary}\label{coro:EstofGlim}
If $F(0)<\varepsilon$, then $F(x)<\varepsilon$ for any $x >0$.
\end{corollary}

\smallskip
Next, we give an estimate of the non-physical waves. Let $\epsilon(t)$ be a non-physical wave of $U^{\mu,\Delta}$ crossing $x=t$.
Then we can obtain the following estimates of the total strength of non-physical waves,
whose proof will be given in Appendix \ref{appen:A1}.
\begin{proposition}\label{prop:EoNP}
For every $x>0$, there exists $\nu_0>0$ such that, when the threshold parameter $\nu\in (0,\nu_0)$,
\begin{align*}
		\sum_{\epsilon\in\mathcal{NP}}\epsilon(x)<\mu.
\end{align*}
\end{proposition}

By Corollary \ref{coro:EstofGlim}, applying a similar argument as in \cite{Chen2008}, we obtain
\begin{proposition}\label{prop:Important2}
There exists $\tilde\varepsilon>0$
such that, given $\|(\tilde{p}_{b}, \widetilde{U}_{\infty})\|_{ L^{\infty}\cap BV}<\tilde\varepsilon$,
then, for every $\mu$, $\Delta x$, and $\delta >0$,
the modified wave-front tracking algorithm provides global approximate solutions $U^{\mu,\Delta x}$
and the corresponding approximate boundaries $\Gamma^{\mu,\Delta x}$ and $1$-strong shocks $\mathtt{S}^{\mu,\Delta x}$
in $\Omega^{\mu,\Delta x}$ satisfying all $H_1(h)$--$H_3(h)$ for any $h\geq1$,
	\begin{align*}
		&\text{T.V.}\{U^{\mu,\Delta x}(x,\cdot):\ (-\infty,\chi^{\mu,\Delta x}(x))\}
        <C\,\text{T.V.}\{(\tilde{p}_{b},\widetilde{U}_{\infty})\},\\
		&\text{T.V.}\{U^{\mu,\Delta x}(x,\cdot):\ (\chi^{\mu,\Delta x}(x),b^{\mu,\Delta x}(x))\}
		<C\,\text{T.V.}\{(\tilde{p}_{b}, \widetilde{U}_{\infty})\},
	\end{align*}
	and
	\begin{align*}
		&|b^{\mu,\Delta x}(x+t)-b^{\mu,\Delta x}(x)|\leq\big(|\dfrac{v_{+}}{u_{+}}|+\mathcal{E}\big)|t|,\\
		&|\chi^{\mu,\Delta x}(x+t)-\chi^{\mu,\Delta x}(x)|\leq\big(|s_{0}|+\mathcal{E}\big)|t|,
	\end{align*}
	for any $x\geq0$ and $t>0$, where $\mathcal{E}>0$ and $C>0$ are constants
 depending only on $U_-$, $U_+$, and system \eqref{eqn:TwoDFEuler}.\\
	\end{proposition}

Furthermore, we now give estimates about the strong $1$-shock and the boundary.
\begin{proposition}\label{Prop:Important2}
There is a constant $\bar{M}>0$ such that
\begin{align*}
&T.V.\{\sigma^{\mu,\Delta x}(\cdot)\,:\,[0,\infty)\}
  =\sum_{\tau\in\Lambda}|\sigma^{\mu,\Delta x}(\tau+)-\sigma^{\mu,\Delta x}(\tau-)|\leq \bar M,\\
&T.V.\{U^{\mu, \Delta x}_{b}(\cdot)\,:\,[0,\infty)\}
 =\sum_{\tau\in\Lambda}|U^{\mu, \Delta x}_{b}(\tau+)-U^{\mu, \Delta x}_{b}(\tau-)|\leq \bar M(1+\mu),
\end{align*}
where $\Lambda$ is a set consisting of all the $x$-coordinates at which a colliding happens,
and $\bar{M}$ is independent of $\mu$ and $\Delta x$.
\end{proposition}

\begin{proof}
For any $\tau\in\Lambda$, we define
\begin{equation}
E_{U^{\mu,\Delta x}}(\tau):=
\left\{\begin{aligned}
			&\eta_{1}|\alpha_i||\beta_j|\quad &&\text{ for Case \ref{Case1},} \\
			&\eta_{1}|\beta_1| &&\text{ for Case \ref{Case2},} \\
			&\eta_{1}|\alpha_i| &&\text{ for Case \ref{Case3},} \\
			&\eta_{1}|\omega_h|&&\text{ for Case \ref{Case4},}\\
			&\eta_{1}|\beta_4| &&\text{ for Case \ref{Case5},}
\end{aligned}
\right.\label{eqn:DecreaseQuantities}
\end{equation}
for some $\eta_{1}>0$.
From the proof of Proposition \ref{prop:befIm}, when $\eta_{1}$ is sufficiently small,
it can be verified that
\begin{align*}
\sum_{\tau\in\Lambda}E_{U^{\mu,\Delta x}}(\tau)\leq\sum_{\tau\in\Lambda}\eta_{1}^{-1}(F(\tau-)-F(\tau+))\leq\eta_{1}^{-1}F(0).
\end{align*}
Meanwhile, Lemmas \ref{lem:SBe} and \ref{lem:Isr} imply that
	\begin{align*}
		\sum_{\tau\in\Lambda}|s^{\mu,\Delta x}(\tau+)-s^{\mu,\Delta x}(\tau-)|
		\leq \sum_{\tau\in\Lambda}(|K_{bs}|+|K_{s1}|)E_{U^{\mu,\Delta x}}(\tau)
		\leq(\sup|K_{bs}|+\sup|K_{s1}|)\eta_{1}^{-1}F(0).
	\end{align*}
	
As for the estimates of the boundary, we need to note that there are errors produced
due to non-physical waves.
However, those non-physical waves do not change, once they hit the boundary.
Hence, before $x=t$, the total strength of non-physical waves that hit the boundary
is less than the total strength of all the non-physical waves at $x=t$, which is less than $\mu$,
by Proposition \ref{prop:EoNP}. Then similar argument shows
\begin{align*}
\sum_{\tau\in\Lambda}|U^{\mu, \Delta x}_{b}(\tau+)-U^{\mu, \Delta x}_{b}(\tau-)|
\leq C\mu+\sum_{\tau\in\Lambda}CE_{U^{\mu,\Delta x}}(\tau)
=C\mu+C\eta_{1}^{-1}F(0).
\end{align*}
Combining these estimates yields the result.
\end{proof}

Combining Proposition \ref{prop:EoNP}--\ref{Prop:Important2}, we now conclude one of our main theorems (see \cite{Zhang2003,Chen2006,Chen2007}). For completeness, a proof is provided in Appendix \ref{appen:A2}.
\begin{theorem}\label{thm:EaS}
There exists $\tilde\varepsilon>0$ such that,
when $\|(\widetilde{U}_{\infty},\tilde{p}_{b})\|_{ L^{\infty}\cap BV}
<\tilde\varepsilon$,
there is a subsequence $\{\mu_l\}_{l=1}^\infty$ and corresponding $\{\Delta x_l\}_{l=1}^\infty$ so that
\begin{itemize}
\item[\rm(i)] In any bounded $x$-interval, $(b^{\mu_l,\Delta x_l}, \chi^{\mu_l,\Delta x_l})$ converges to $(b, \chi)$ uniformly{\rm ;}
\item[\rm(ii)] $U^{\mu_l, \Delta x_l}_{b}$ converges to $U_{b}\in \text{BV}([0,\infty))$ a.e. with
$\dot{b}(x)=\frac{v_b(x)}{u_b(x)}$ a.e.,
while $s^{\mu_l,\Delta x_l}$ converges to $s\in \text{BV}([0,\infty))$ a.e. with $\dot{\chi}(x)=s(x)$ a.e.{\rm ;}
\item[\rm(iii)] For every $x>0$, $U^{\mu_l,\Delta x_l}$ converges to $U$
in $ L^1_{\text{\rm loc}}\big((-\infty,b(x));\mathbb{R}^4\big)$,
which is an entropy solution to problem \eqref{eqn:TwoDFEuler}
in $\Omega=\{(x,y)\,:\,x\geq0,y<b(x)\}$ satisfying \eqref{eqn:ConseForm}--\eqref{1.9a}.
\end{itemize}
\end{theorem}

\section{A modified Lyapunov functional}\label{sect-Lya}
Relying on the analysis in \S \ref{sect-appro},
we now analyze the well-posedness of this system.
Compared to the problem with the given wedge boundary, our solutions may be constructed
on the different regions.
As a result, it seems hard to measure the distance of two weak solutions $U_1$ and $U_2$
in $ L^1$ directly.
However, in this paper, we extend two different solutions to the same domain and compare
the corresponding  $ L^1$--norm of two extended solutions in this domain.
With this setup, we can introduce a Lyapunov-type functional
to measure the new $ L^1$--distance, as well as
the $ L^{\infty}$--distance between the two boundaries.
We first extend $U^{\mu,\Delta x}(x,y)$ to
\begin{align*}
U^{\mu,\Delta x}_E(x,y)=\left\{\begin{aligned}
	&U^{\mu,\Delta x}(x,y) \quad &&\mbox{for all $y\leq b^{\mu,\Delta x}(x)$}, \\
	&U^{\mu,\Delta x}(x,b^{\mu,\Delta x}(x))\quad &&\mbox{for all $y>b^{\mu,\Delta x}(x)$},
\end{aligned}
\right.
\end{align*}
for all $x>0$, where $b^{\mu,\Delta x}(x)$ is the corresponding boundary.
As a result, our weak solutions are extended as
\begin{align*}
U_E(x,y)=\left\{\begin{aligned}
	&U(x,y) &&\mbox{for all $y\leq b(x)$}, \\
	&U(x,b(x)) \quad&&\mbox{for all $y>b(x)$}.
\end{aligned}
\right.
\end{align*}

To simplify the notation, we omit subscript $_E$ and superscript $^{\mu,\Delta x}$, and use $U$
to denote an extended approximate solution.
Given two suitable initial data functions $U_{\infty,1}$ and $U_{\infty,2}$
with corresponding pressure distributions
$p_{b,1}$ and $p_{b,2}$ on the boundaries, according to our previous construction,
for every $\mu>0$, there are two $\mu$-approximate solutions $U_1$ and $U_2$
with approximate boundaries $y=b_1(x)$ and $y=b_2(x)$, respectively.
Set
\begin{align*}
b_{\max}(x):=\max\{b_1(x),b_2(x)\},\qquad\, b_{\min}(x):=\min\{b_1(x),b_2(x)\},
\end{align*}
and call
\begin{align*}
\Gamma_{\max}:=\big\{(x,b_{\max}(x))\,:\,x\ge 0\big\},\qquad\, \Gamma_{\min}(x)=\big\{(x,b_{\min}(x))\,:\, x\ge 0\big\}
\end{align*}
the outer and inner boundary of $U_1$ and $U_2$, respectively.
Similar to \cite{Bressan2000,Bressan1999,Chen2008,Holden2015,Lewicka2000,Lewicka2001,Lewicka2002,Liu1999},
fixing $x$ and given $U_1$, $U_2$, the scalar functions $h_i(y)$ are implicitly defined by
\begin{itemize}
\item $U_2\big(x,y\big)=H\big(h_1(y),h_2(y),h_3(y),h_{4}(y);U_1(x,y)\big)$, $U_2\in O_{\varepsilon_{0}}(U_+)$,
and $U_1\in O_{\varepsilon_{0}}(U_-)\cup O_{\varepsilon_{0}}(U_+)$;
\item $U_1\big(x,y\big)=H\big(h_1(y),h_2(y),h_3(y),h_{4}(y);U_2(x,y)\big)$, $U_1\in O_{\varepsilon_{0}}(U_+)$,
and $U_2\in O_{\varepsilon_{0}}(U_-)$,
\end{itemize}
where
\begin{align*}
H(h_1,h_2,h_3,h_4;U)
:=
S_4(h_4)\big(\Phi_{3}(h_3;\Phi_{2}(h_2;S_1(h_1)(U)))\big),
\end{align*}
and $S_i$ are $i$-Hugoniot curves, $i=1, 2, 3,4$.
Moreover, $h_{b,i}(x), i=1,2,3,4$, are implicitly defined by
\begin{equation}
U_2(x,b_{\max}(x))=H(h_{b,1}(x),h_{b,2}(x),h_{b,3}(x),h_{b,4}(x);U_1(x,b_{\min}(x))).
\label{eqn:StabiBoundary}
\end{equation}
Furthermore, for an $i$-wave that connects two states on the $i$-Hugoniot curve,
let $\lambda_{i}(x)$ be its speed.
Then we define the weighted $ L^1$--strengths:
\begin{equation}
\begin{aligned}
	&q_i(y)=\left\{\begin{aligned}
		&c_i^uh_i(y)\quad &&\text{if }U_1,\ U_2\in O_{\varepsilon_{0}}(U_+),\\
		&c_i^mh_i(y)&&\text{if }U_1,\ U_2\text{ are in different domains},\\
		&c_i^lh_i(y)&&\text{if }U_1,\ U_2\in O_{\varepsilon_{0}}(U_-),
	\end{aligned}
	\right.\\[2mm]
	&q_{b,i}(x)=c_i^uh_{b,i}(x),
\end{aligned}
\label{eqn:L1weight}
\end{equation}
where $c_i^u$, $c_i^m$, $c_i^l<1$ are constants that remain to be determined based on
the interaction and reflection estimates obtained in Lemma \ref{lem:Iww}--\ref{lem:Isrbelow}.
In addition, when $h_1$ is a large shock that connects a state in $O_{\varepsilon_{0}}(U_-)$
and the other state in $O_{\varepsilon_{0}}(U_+)$,
we set $q_1=B$, where $B<1$ is a constant larger than the total strength of small waves
of the other families, which is regarded as the ``strength" of the strong shock. \par

For each $x>0$, the set of all weak waves in $U_1$ and $U_2$ is denoted
by $\mathcal{J}=\mathcal{J}(U_1)\cup\mathcal{J}(U_2)$,
and the strength of a $k_\alpha$-wave $\alpha\in\mathcal{J}$,
located at point $y_\alpha$, is denoted by $|\alpha|$.
Then we define the following quantities:
\begin{align*}
B_i(y)&=\Bigg(
\sum_{\substack{\alpha\in\mathcal{J}(U_1)\\y_\alpha<y,\, i<k_\alpha\leq4}}
+\sum_{\substack{\alpha\in\mathcal{J}(U_2)\\y_\alpha<y,\, i<k_\alpha\leq4}}
+\sum_{\substack{\alpha\in\mathcal{J}(U_1)\\y_\alpha>y,\, 1\leq k_\alpha<i}}
+\sum_{\substack{\alpha\in\mathcal{J}(U_2)\\y_\alpha>y,\, 1\leq k_\alpha<i}}\Bigg)
|\alpha|,\\
C_i(y)&=\left\{
\begin{aligned}
	&\left(\textstyle\sum\nolimits_{\alpha\in\mathcal{J}(U_1),\, y_\alpha<y,\, k_\alpha=i}
	+\sum\nolimits_{\alpha\in\mathcal{J}(U_2),\, y_\alpha>y,\, k_\alpha=i}\right)|\alpha|\qquad\text{if $q_i(y)<0$,}\\
	&\left(\textstyle\sum\nolimits_{\alpha\in\mathcal{J}(U_2),\, y_\alpha<y,\, k_\alpha=i}
	+\sum\nolimits_{\alpha\in\mathcal{J}(U_1),\, y_\alpha>y,\, k_\alpha=i}\right)|\alpha|\qquad\text{if $q_i(y)>0$,}
\end{aligned}
\right.\\[1mm]
D_i(y)&=B-D^{c}_{i}(y),\\[1mm]
D^{c}_i(y)&=\left\{
\begin{aligned}
	&B\quad&&\text{if $i=1$ and $U_1$, $U_2$ are in different domains}\\
	&  \quad &&\text{\,\,\,\, or $i=2,3$ and $U_1$, $U_2\in O_{\varepsilon_{0}}(U_+)$}, \\
	&0 \quad &&\text{other cases,}
\end{aligned}
\right.\\
F_i(y)&=\Bigg(\sum_{\substack{\alpha\in\mathcal{J}\\y_\alpha>y,\, k_\alpha=1\\
\text{both states joined by $\alpha$}\\ \text{are located in $O_{\varepsilon_{0}}(U_+)$}}}
+\sum_{\substack{\alpha\in\mathcal{J}\\y_\alpha<y,\, k_\alpha=1\\
\text{both states joined by $\alpha$}\\ \text{are located in $O_{\varepsilon_{0}}(U_-)$}}}
\Bigg)|\alpha|.
\end{align*}
Let
\begin{equation}
A_i(y)=B_i(y)+D_i(y)
+\left\{
\begin{aligned}
	&C_i(y)\quad &&\text{if }q_i(y) \text{ is small},\\
	&F_i(y)&&\text{if }i=1\text{ and }q_i(y)=B \text{ is large},
\end{aligned}
\right.
\label{eqn:DefofA}
\end{equation}
where the ``small" and the ``large" mean a wave connecting both states in either
$O_{\varepsilon_{0}}(U_+)$ or $O_{\varepsilon_{0}}(U_-)$ and a strong shock wave connecting a state in  $O_{\varepsilon_{0}}(U_-)$ and the other in $O_{\varepsilon_{0}}(U_+)$, respectively.
Thus, $A_i(y)$ equals to the total strength of the waves in $U_1$ and $U_2$ approaching
the $i$-wave $q_i(y)$. We now introduce the following modified Lyapunov functional:
\begin{align*}
\mathfrak{F}(U_1(x),U_2(x))=c_b\int_0^x\big(|h_{b,1}(\tau)|+|h_{b,4}(\tau)|\big)\,\dd\tau
+\sum_{i=1}^{4}\int_{-\infty}^{b_{\max}} W_i(y)|q_i(y)|\,\dd y,
\end{align*}
where $c_0>0$ to be chosen and
\begin{align*}
W_i(y)=1+\kappa_1A_i(y)+\kappa_2\big(Q(U_1)+Q(U_2)\big),
\end{align*}
with the two constants $\kappa_1$ and $\kappa_2$ to be determined later and the interaction potential $Q$
introduced in (\ref{eqn:GlimPotential}). Taking the initial value of the Glimm functional $F(0)$ small enough,
we can prove
\begin{align*}
&1\leq W_i(y)\leq C_{0} \qquad \mbox{for $i=1, 2, 3, 4$},
\end{align*}
where $C_0$ is independent of $x$ and $\mu$.

\smallskip
\begin{proposition}\label{prop:GeneralInteractionpoint}
At each
$x>0$ where two fronts of $U_1$ or $U_2$ interact, or one of the approximate pressure corresponding to the outer boundary changes,
or a physical wavefront of $U_1$ or $U_2$ hits the outer boundary, then
	\begin{equation*}
		\mathfrak{F}(U_1(x+),U_2(x+))\leq\mathfrak{F}(U_1(x-),U_2(x-)).
	\end{equation*}
	\end{proposition}

\begin{proof}
A direct computation leads to
\begin{align*}
	\mathfrak{F}(U_1(x+),U_2(x+))-\mathfrak{F}(U_1(x-),U_2(x-))
  =\sum_{i=1}^{4}\int_{-\infty}^{b_{\max}}\big(W_i(x+,y)-W_i(x-,y)\big)|q_i(y)|\,\dd y.
\end{align*}
Then Propositions \ref{prop:befIm} and \ref{Prop:Important2} yield
\begin{align*}
	&Q(U_j(x+))-Q(U_j(x-))\leq-\frac{1}{4}E_{U_j}(x),\\
	&A_i(x+,y)-A_i(x-,y)=O(1)(E_{U_1}(x)+E_{U_2}(x)),
\end{align*}
where $E_{U_j}$, $j=1,\,2$, are given in (\ref{eqn:DecreaseQuantities}).
Therefore, if $\kappa_2$ is large enough, all the weight functions $W_i(y), i=1,2,3,4$, decrease.
\end{proof}

\begin{proposition}\label{prop:NonInteractionpoint}
Let $F_j(t)$ denote the Glimm functional of $U_j$, $j=1, 2$.
If there are no interactions at $x>0$, and $c_b$ and $F_j(0)$ are sufficiently small for $j=1,\,2$,
then
\begin{equation}
	\frac{\dd}{\dd x}\mathfrak{F}(U_1(x),U_2(x))
 \leq O(1)\mu+O(1)\big|p_{b,2}^{\mu,\Delta x}(x)-p_{b,1}^{\mu,\Delta x}(x)\big|.\label{eqn:StabilityDiff}
\end{equation}
\end{proposition}

\begin{proof} We divide the proof into three steps.

\smallskip
1. Denote the speed of the $i$-wave $q_i(x)$ and $q_{b,i}(x)$ by $\lambda_i$ and $\lambda_{b,i}$, respectively.
Taking the derivative of $\mathfrak{F}(U_1(x),U_2(x))$ with respect to $x$ leads to
\begin{align*}
&\frac{\dd}{\dd x}\mathfrak{F}(U_1(x),U_2(x))\\
&=c_b\big(|h_{b,1}(x)|+|h_{b,4}(x)|\big)
  +\sum_{\alpha\in\mathcal{J}}\sum_{i=1}^4\big(|q_i^{\alpha-}|W_i^{\alpha-}-|q_i^{\alpha+}|W_i^{\alpha+}\big)\dot{y}_{\alpha}\\
&\quad +\sum_{i=1}^{4}|q_{b,i}(x)|W_i(b_{\max}(x))\dot{b}_{\max}(x)\\
&=c_b\big(|h_{b,1}(x)|+|h_{b,4}(x)|\big)
 +\sum_{\alpha\in\mathcal{J}}\sum_{i=1}^{4}\big(|q_i^{\alpha-}|W_i^{\alpha-}(\dot{y}_{\alpha}-\lambda_i^{\alpha-})
-|q_i^{\alpha+}|W_i^{\alpha+}(\dot{y}_{\alpha}-\lambda_i^{\alpha+})\big)\\
&\quad+\sum_{i=1}^{4}|q_{b,i}(x)|W_i(b_{\max}(x))\big(\dot{b}_{\max}(x)-\lambda_{b,i}(x)\big),
\end{align*}
where $\dot{y}_{\alpha}$ stands for the speed of $\alpha\in\mathcal{J}$, and $q_i^{\alpha\pm}=q_i\big(y_{\alpha}\pm\big)$,
	$W_i^{\alpha\pm}=W_i\big(y_{\alpha}\pm\big)$, $\lambda_i^{\alpha\pm}=\lambda_i\big(y_{\alpha}\pm\big)$,
    and $\lambda_{b,i}(x)=\lambda_{i}(b_{\max}(x)-)$ for $\alpha\in\mathcal{J}(U_j)$. Define
	\begin{align*}
		E_{\alpha,i}
		&=|q_i^+|W_i^+(\lambda_i^+-\dot{y}_{\alpha})
		-|q_i^-|W_i^-(\lambda_i^--\dot{y}_{\alpha}),\\
		E_{b,i}
		&=|q_{b,i}(x)|W_i(b_{\max}(x))\big(\dot{b}_{\max}(x)-\lambda_{b,i}(x)\big),
	\end{align*}
	with $q_i^{\pm}=q_i^{\alpha\pm}$, $W_i^{\pm}=W_i^{\alpha\pm}$, and $\lambda_i^{\pm}=\lambda_i^{\alpha\pm}$.
    Then
	\begin{equation}
		\frac{\dd}{\dd x}\mathfrak{F}(U_1(x),U_2(x))
		=c_b(|h_{b,1}(x)|+|h_{b,4}(x)|)+\sum_{\alpha\in\mathcal{J}}\sum_{i=1}^{4}E_{\alpha,i}+
		\sum_{i=1}^{4}E_{b,i}.
	\label{eqn:StabilityForm}
	\end{equation}

2. We need the following lemma to complete the proof.

 \begin{lemma}\label{lem:Stability1}
 The following estimates hold{\rm :}
		\begin{eqnarray}
			\text{ $\alpha$ is a strong shock}:&	&\sum_{i=1}^4E_{\alpha,i}\leq0,
			\label{eqn:StabilityGoal1}\\
			\text{ $\alpha$ is a non-physical wave}:&	&\sum_{i=1}^4E_{\alpha,i}\leq O(1)|\alpha|,
			\label{eqn:StabilityGoal2}\\
			\text{ $\alpha$ is a weak wave}:& &\sum_{i=1}^4E_{\alpha,i}\leq O(1)\mu|\alpha|,
			\label{eqn:StabilityGoal3}\\
			\text{on the boundary}:&&\sum_{i=1}^4E_{b,i}+c_0(|h_{b,1}(x)|+|h_{b,4}(x)|)\nonumber\\
			&&\leq O(1)|p_{b,2}^{\mu,\Delta x}(x)-p_{b,1}^{\mu,\Delta x}(x)|+O(1)\mu,
			\label{eqn:StabilityGoal4}
	\end{eqnarray}
 where $c_0>0$ is a sufficiently small constant, independent of the waves and the boundary.
 \end{lemma}
	
 \begin{proof} There are
 five cases.

 \smallskip
 \noindent
\textit{Case $1${\rm :} The states connected by $\alpha$ is a weak wave $($either physical or non-physical$)$}.
In this case, we choose $B$ suitably small and take $\kappa_1$ large enough so that
the estimates can be obtained by following Bressan-Liu-Yang \cite{Bressan1999}.

 \smallskip
 \noindent
\textit{Case $2${\rm :} $\alpha$ is the below strong shock in $U_1$ or $U_2$}; see also \cite{Lewicka2002,Chen2008}. Then we have
		\begin{align*}
			E_{\alpha,1}&=BW_1^+(\lambda_1^+-\dot{y}_{\alpha})-|q_1^-|W_1^-(\lambda_1^--\dot{y}_{\alpha})\\
			&\leq O(1)B\sum_{i=1}^{4}|q_{1}^{-}|-\kappa_{1}B|\lambda_1^--\dot{y}_{\alpha}|,
		\end{align*}
		and
\begin{align*}
\sum_{i=2}^4E_{\alpha,i}&=\sum_{i=2}^4\big(|q_i^-|(\lambda_i^--\dot{y}_{\alpha})(W_i^+-W_i^-)
+W_i^-(|q_i^+|(\lambda_i^+-\dot{y}_{\alpha})-|q_i^-|(\lambda_i^--\dot{y}_{\alpha}))\big),\\
&\leq \sum_{i=2}^4O(1)\kappa_1B|q_i^+||\lambda_i^+-\dot{y}_{\alpha}|-\dfrac{3}{4}\sum_{i=2}^4\kappa_1B|q_i^-||\lambda_i^--\dot{y}_{\alpha}|.
\end{align*}
When $\kappa_1$ is large enough, we have
\begin{align*}
\sum_{i=1}^4E_{\alpha,i}
&=\sum_{i=2}^4O(1)\kappa_1B|q_i^+||\lambda_i^+-\dot{y}_{\alpha}|-\dfrac{1}{2}\sum_{i=1}^4\kappa_1B|q_i^-||\lambda_i^--\dot{y}_{\alpha}|.
\end{align*}
Recall that the wave curve and the Hugoniot curve passing through $U_0$ have the same curvature at $U_0$.
Carrying out similar arguments as in Lemma \ref{lem:Isrbelow} (by letting $\alpha_i=h_i^-$, $\beta_i=0$,
and $\delta_i=h_i^+$), we obtain
		\begin{align*}
			|h_k^+|=O(1)\sum_{i=1}^4|h_i^-| \qquad\mbox{for $k=2,\,3,\,4$}.
		\end{align*}
In (\ref{eqn:L1weight}), we take $c_i^l$, $i=1,\,2,\,3,\,4$, larger enough than $c_i^m$, $i=2,\,3,\,4$,
to complete the proof of (\ref{eqn:StabilityGoal1}).

  \smallskip
\noindent
\textit{Case $3${\rm :} $\alpha$ is a weak wave lying between strong shocks}; see also \cite{Lewicka2002,Chen2008}. When $i=1$, we have
		\begin{align*}
		E_{\alpha,1}
		&=B\big((W_1^+-W_1^-)(\lambda_1^\pm-\dot{y}_{\alpha})+W_1^\mp(\lambda_1^+-\lambda_1^-)\big)\\
		&\leq B\big(-\kappa_1|\alpha||\lambda_1^\pm-\dot{y}_{\alpha}|+O(1)|\alpha|\big).
		\end{align*}
		As for $i=2,\,3,\,4$, we can obtain
		\begin{align*}
			E_{\alpha,i}&=|q_i^\pm|(W_i^+-W_i^-)(\lambda_i^\pm-\dot{y}_{\alpha})+W_i^\mp(|q_i^+|(\lambda_i^+-\dot{y}_{\alpha})-|q_i^-|(\lambda_i^--\dot{y}_{\alpha}))\\
			&\leq O(1)\big(O(1)+\kappa_1B\big)(|q_i^+-q_i^-|+|q_i^-||\alpha|)+O(1)|\alpha|.
		\end{align*}
		As a result, summing all the estimates obtained above, we conclude
\begin{align*}
\sum_{i=1}^4E_{\alpha,i}
&\leq B\,\big(-\kappa_1|\alpha||\lambda_1^\pm-\dot{y}_{\alpha}|+O(1)|\alpha|\big)
+\big(O(1)+\kappa_1B\big)O(1) \big(|q_i^+-q_i^-|+|q_i^-||\alpha|\big)
 +O(1)|\alpha|\\
&\leq \big(-\kappa_1cB+O(1)\big)|\alpha|+B\kappa_1\,O(1)\Big(\sum_{i=2}^4|q_i^+-q_i^-|+\sum_{i=2}^4|q_i^-||\alpha|\Big)\\
&= \big(-\frac{\kappa_1cB}{2}+O(1)\big)|\alpha|
 +B\,\Big(-\frac{\kappa_1c}{2}|\alpha|+\kappa_1O(1)\,\Big(\sum_{i=2}^4|q_i^+-q_i^-|+\sum_{i=2}^4|q_i^-||\alpha|\Big)\Big),
\end{align*}
where $c>0$ is a lower bound of the difference between the speed of the strong $1$-shock and a weak shock.
Choose $\kappa_1$ large enough and all the weights $c_i^m$ sufficiently small to obtain
  		\begin{align*}
			\sum_{i=1}^4E_{\alpha,i}\leq0.
		\end{align*}
	
\textit{Case 4{\rm :} $\alpha$ is the above strong shock in $U_1$ or $U_2$}; see also \cite{Lewicka2002,Chen2008}.
Similar arguments as in Lemma \ref{lem:Isr} yield
		\begin{equation}
			h_4^-=K_{s4}h_1^++h_4^+.
		\label{eqn:StabilitySecond1}
		\end{equation}
Due to Lemma \ref{lem:Isr},
when $F_j(0)$, $j=1,\,2$, are sufficiently small, we can choose $c_1^u$  and $c_4^u$ such that
		\begin{equation}
			\begin{aligned}
				&\frac{c_1^u}{c_4^u}<1,\\
				&|K_{s4}|\,\frac{c_4^u}{c_1^u}\,\frac{\lambda_4(U_+)-s_0}{\lambda_1(U_+)-s_0}<\gamma_0<1.
			\end{aligned}
		\label{eqn:StabilitySecond2}
		\end{equation}
		Then, when $i=1$, we obtain
		\begin{align*}
			E_{\alpha,1}&=-BW_1^-(\lambda_1^--\dot{y}_{\alpha})+|q_1^+|W_1^+(\lambda_1^+-\dot{y}_{\alpha})\\
			&\leq O(1)B|q_1^+|-\kappa_1B|q_1^+||\lambda_1^+-\dot{y}_{\alpha}|\\
			&\leq O(1)B|q_1^+|-\kappa_1Bc_1^u|h_1^+||\lambda_1^+-\dot{y}_{\alpha}|.
		\end{align*}
		When $i=2,\,3$, we have
		\begin{align*}
			E_{\alpha,i}&=|q_i^-|(W_i^+-W_i^-)(\lambda_i^--\dot{y}_{\alpha})+W_i^+(|q_i^+|(\lambda_i^+-\dot{y}_{\alpha})-|q_i^-|(\lambda_i^--\dot{y}_{\alpha}))\\
			&\leq-\kappa_1B|q_i^-|(\lambda_i^--\dot{y}_{\alpha}) +O(1)|q_i^+|\\
			&\leq-\kappa_1B|q_i^-|(\lambda_i^--\dot{y}_{\alpha}) +O(1)(|q_i^-|+|q_1^+|).
		\end{align*}
By (\ref{eqn:StabilitySecond1}) and (\ref{eqn:StabilitySecond2}),
\begin{align*}
	E_{\alpha,4}&= W_4^+|q_4^+|(\lambda_4^+-\dot{y}_{\alpha})-|q_4^-|W_4^-(\lambda_4^--\dot{y}_{\alpha})\\
	&\leq(\kappa_1B+O(1))|q_4^+|(\lambda_4^+-\dot{y}_{\alpha})
	-\kappa_1B|q_4^-|(\lambda_4^--\dot{y}_{\alpha})\\
	&\leq (\kappa_1B+O(1))c_4^u(|h_4^-|+K_{s4}|h_1^+|)(\lambda_4^+-\dot{y}_{\alpha})-\kappa_1Bc^m_4|h_4^-|(\lambda_4^--\dot{y}_{\alpha})\\
	&\leq  (\kappa_1B+O(1))\big(c_4^u|h_4^-|(\lambda_4^+-\dot{y}_{\alpha})+\gamma_0c_1^u|h_1^+||\lambda_1^+-\dot{y}_{\alpha}|-c^m_4|h_4^-|(\lambda_4^--\dot{y}_{\alpha})\big).
	\end{align*}
Choosing $c_4^u$ relatively smaller than $c_4^m$ and $\kappa_1$ suitably large, we obtain
\begin{align*}
\sum_{i=1}^4E_{\alpha,i}
\leq&-(1-\gamma_0)\kappa_1B|q_1^+||\lambda_1^+-\dot{y}_{\alpha}|+O(1)|q_1^+||\lambda_1^+-\dot{y}_{\alpha}|+O(1)|q_1^+|\\
&+(\kappa_1B+O(1))\big(c_4^u|h_4^-|(\lambda_4^+-\dot{y}_{\alpha})-c^m_4|h_4^-|(\lambda_4^--\dot{y}_{\alpha})\big)\\
&+\sum_{i=2}^3\big(-\kappa_1B|q_i^-|(\lambda_i^--\dot{y}_{\alpha}) +O(1)|q_i^-|\big)\\
&\leq 0.
\end{align*}
	
 \textit{Case 5{\rm :} Near the boundary}. First of all, let $U=(u,v,p,\rho)^\top$ be in a small neighborhood of $U_+$.
 Consider the following equation:
		\begin{equation}
			H^{(3)}(h_{1},h_{2},h_{3},h_{4};U)=p+r,
		\label{eqn:StabiBoundaryest}
		\end{equation}
where $H^{(3)}(h_{1},h_{2},h_{3},h_{4};U)$ is the third component of $H(h_{1},h_{2},h_{3},h_{4};U)$. Similar to Lemma \ref{lem:WBe}, we deduce from (\ref{eqn:StabiBoundaryest}) that
		\begin{align}
				h_{4}&= h_{4}(h_{1},h_{2},h_{3},r)-h_{4}(0,h_{2},h_{3},r)
				+h_{4}(0,h_{2},h_{3},r)-h_{4}(0,h_{2},h_{3},0)\nonumber\\
				&=h_{1}\int_0^1\frac{\partial h_{4}}{\partial h_{1}}(\lambda h_{1},h_{2},h_{3},r)\,
                  \dd\lambda+O(1)r\label{eqn:BoundaryLyapunov}
		\end{align}
		for $|r|$ small enough. In addition, we have
		\begin{align*}
			\left|\int_0^1\frac{\partial h_{4}}{\partial h_{1}}(\lambda h_{1},h_{2},h_{3},r)\,\dd\lambda\right|\rightarrow 1
            \qquad\,\,\mbox{as $(h_{1},h_{2},h_{3},r)\rightarrow0\,$ and $\,U\to U_+$}.
		\end{align*}
From (\ref{eqn:StabiBoundary}), considering the effect of non-physical waves on the boundary, we have
		\begin{align*}
			H^{(3)}(h_{b,1}(x),h_{b,2}(x),h_{b,3}(x),h_{b,4}(x);U_1(x,b_1(x)))=p_{b,1}^{\mu,\Delta x_1}(x)+p_{b,2}^{\mu,\Delta x_2}(x)-p_{b,1}^{\mu,\Delta x_1}(x)+O(1)\mu.
		\end{align*}
The former arguments tell us that
		\begin{align*}
|h_{b,1}(x)|\leq|h_{b,4}(x)|\eta+O(1)\big(p_{b,2}^{\mu,\Delta x_2}(x)-p_{b,1}^{\mu,\Delta x_1}(x)+\mu\big)
\end{align*}
for some $\eta$ close to $1$.
Note that, crossing a vortex sheet and an entropy wave, the flow direction does not change. Then we have
		\begin{equation}
			|\dot{b}_1(x)-\dot{b}_2(x)|=\left|\dfrac{v_2(x,b_2(x))}{u_2(x,b_2(x))}-\dfrac{v_1(x,b_1(x))}{u_1(x,b_1(x))}\right|=O(1)\big(|h_{b,1}(x)|+|h_{b,4}(x)|\big).
		\label{eqn:StabilityBdry}
		\end{equation}
Therefore, when the initial values of the Glimm functionals $F_j(0)$, $j=1,\,2$, are small enough,
 we obtain
		\begin{align*}
			E_{b,1}&=|q_{b,1}(x)|W_1(b_{\max}(x))\big(\dot{b}_{\max}(x)-\lambda_{b,1}(x)\big)
			=c_1^u|h_{b,1}|\kappa_1B|\lambda_{b,1}|+O(1)|h_{b,1}|\\
			&\leq \kappa_1Bc_1^u\,\eta|h_{b,4}||\lambda_{b,1}|+O(1)|h_{b,4}|
               +O(1)\big(p_{b,2}^{\mu,\Delta x_2}(x)-p_{b,1}^{\mu,\Delta x_1}(x)+\mu\big),\\
			E_{b,k}&=|q_{b,k}(x)|W_k(b_{\max}(x))(\dot{b}_{\max}(x)-\lambda_{b,k}(x))
			=c_k^u|h_{b,k}|O(1)|\dot{b}_{1}-\lambda_{b,k}|=O(1)(|h_{b,1}|+|h_{b,4}|)\\
			&\leq O(1)|h_{b,4}|+O(1)\big(p_{b,2}^{\mu,\Delta x_2}(x)-p_{b,1}^{\mu,\Delta x_1}(x)+\mu\big)
            \qquad \mbox{for $k=2,\,3$,}\\
			E_{b,4}&=|q_{b,4}(x)|W_4(b_{\max}(x))\big(\dot{b}_{\max}(x)-\lambda_{b,4}(x)\big)\\
			&=-c_4^u|h_{b,4}|\kappa_1B|\lambda_{b,1}|
			+c_4^u|h_{b,4}|\kappa_1B(|\lambda_{b,1}|-\lambda_{b,4})+O(1)|h_{b,4}|\\
			&\leq -c_4^u|h_{b,4}|\kappa_1B|\lambda_{b,1}|+O(1)|h_{b,4}|.
		\end{align*}
In (\ref{eqn:StabilitySecond2}), we have chosen $c_1^u<c_4^u$ with $\kappa_1$ suitably large,
we conclude
		\begin{align*}
			\sum_{i=1}^4E_{b,i}
			&\leq \kappa_1Bc_1^u\,\eta|h_{b,4}||\lambda_{b,1}|+O(1)|h_{b,4}|
             +O(1)\big(p_{b,2}^{\mu,\Delta x_2}(x)-p_{b,1}^{\mu,\Delta x_1}(x)+\mu\big)
			-c_4^u|h_{b,4}|\kappa_1B|\lambda_{b,1}|\\
			&\leq (c_1^u\,\eta-c_4^u)\kappa_1B|h_{b,4}||\lambda_{b,1}|+O(1)|h_{b,4}|
                +O(1)\big(p_{b,2}^{\mu,\Delta x_2}(x)-p_{b,1}^{\mu,\Delta x_1}(x)+\mu\big)\\[1mm]
			&\leq -\zeta|h_{b,4}|+O(1)\big(p_{b,2}^{\mu,\Delta x_2}(x)-p_{b,1}^{\mu,\Delta x_1}(x)+\mu\big)
		\end{align*}
		for some $\zeta>0$. When $c_0$ is sufficiently small, we obtain
		\begin{align*}
			\sum_{i=1}^4E_{b,i}+c_0(|h_{b,1}(x)|+|h_{b,4}(x)|)
			&\leq-\zeta|h_{b,4}|+O(1)\big(p_{b,2}^{\mu,\Delta x_2}(x)-p_{b,1}^{\mu,\Delta x_1}(x)+\mu\big)\\
			&\quad\,+c_0\big((\eta+1)|h_{b,4}(x)|+O(1)\big(p_{b,2}^{\mu,\Delta x_2}(x)-p_{b,1}^{\mu,\Delta x_1}(x)+\mu\big)\big)\\[1mm]
			&\leq O(1)\big(p_{b,2}^{\mu,\Delta x_2}(x)-p_{b,1}^{\mu,\Delta x_1}(x)+\mu\big).
		\end{align*}
		This concludes  Lemma \ref{lem:Stability1}.
	\end{proof}
		
3. By Lemma \ref{lem:Stability1}, (\ref{eqn:StabilityBdry}), and (\ref{eqn:StabilityForm}), as long as $F_1(0), F_2(0)$, and
 $c_b$ are chosen small enough, we obtain (\ref{eqn:StabilityDiff}).
This completes the proof of   Proposition \ref{prop:NonInteractionpoint}.
\end{proof}

\begin{proposition}\label{prop:InnerInteractionpoint}
If the approximate pressures $p_{b,i}$, $i=1,\,2$, corresponding to the inner boundary, change at some $x>0$, or there is a reflection on the inner boundary at $x$, then
\begin{equation*}
\mathfrak{F}(U_1(x+),U_2(x+))
\leq \big(1+O(1)|\alpha|\big)\mathfrak{F}(U_1(x-),U_2(x-)),
\end{equation*}
where $|\alpha|=|p_{b,i}(x+)-p_{b,i}(x-)|$ or $|\alpha|$ denotes
the strength of the incoming wavefront that hits the inner boundary.\end{proposition}

\begin{proof}
	With out loss of generality, denote $b_{\max}(x)=b_1(x)$ and $b_{\min}(x)=b_2(x)$. Then, in both cases,
\begin{align*}
	|U_2(x+)-U_2(x-)|\leq O(1)|\alpha|.
\end{align*}
Then
\begin{align*}
	&\mathfrak{F}(U_1(x+),U_2(x+))-\mathfrak{F}(U_1(x-),U_2(x-))\\
	&=\sum_{i=1}^4\int_{-\infty}^{b_{\max}(x)}W_i(x+,y)|q_i(x+,y)|\,\dd y
 -\sum_{i=1}^4\int_{-\infty}^{b_{\max}(x)} W_i(x-,y)|q_i(x-,y)|\,\dd y\\
	&=\sum_{i=1}^4\int_{b_{\min}(x)}^{b_{\max}(x)}\big(W_i(x+,y)|q_i(x+,y)|-W_i(x-,y)|q_i(x-,y)|\big)\,\dd y\\
	&\quad\,+\sum_{i=1}^4\int_{-\infty}^{b_{\min}(x)} \big(W_i(x+,y)|q_i(x+,y)|-W_i(x-,y)|q_i(x-,y)|\big)\,\dd y\\
	&\leq \sum_{i=1}^4\int_{b_{\min}(x)}^{b_{\max}(x)} \big(W_i(x+,y)|q_i(x+,y)|-W_i(x-,y)|q_i(x-,y)|\big)\,\dd y
	\\
	&\leq O(1)|U_2(x+)-U_2(x-)|\,|b_{\max}(x)-b_{\min}(x)|\\
	&\leq O(1)|\alpha|\,|b_{\max}(x)-b_{\min}(x)|.
\end{align*}
Using the boundary condition (\ref{eqn:Boundary}), we have
\begin{align*}
	|b_{\max}(x)-b_{\min}(x)|\leq&\int_0^{x}\big|\frac{v_1}{u_1}(s)-\frac{v_2}{u_2}(s)\big|\,\dd s\\
	&\leq O(1)\int_0^{x}\big(|h_{b,1}(s)|+|h_{b,4}(s)|\big)\,\dd s\\
	&\leq O(1)\,\mathfrak{F}(U_1(x-),U_2(x-)).
\end{align*}
Therefore, we conclude
\begin{align*}
\mathfrak{F}(U_1(x+),U_2(x+))
\leq \big(1+O(1)|\alpha|\big)\mathfrak{F}(U_1(x-),U_2(x-)).
\end{align*}
\end{proof}

To conclude, we have
\begin{proposition} \label{prop:Stability}
For any $x>0$,
\begin{align*}
&\int_0^x|\dot{b}_1(t)-\dot{b}_2(t)|\,\mathrm{d}t+\int_{-\infty}^{b_{\max}(x)}|U_1(x,y)-U_2(x,y)|\,\dd y\\
&\,\leq
	O(1)x\mu+O(1) \|p_{b,2}^{\mu,\Delta x_2}-p_{b,1}^{\mu,\Delta x_1}\|_{ L^{1}(0,x)}+O(1)\|U_{\infty,2}^{\mu,\Delta x_2}-U_{\infty,1}^{\mu,\Delta x_1}\|_{ L^{1}((-\infty,0))}.
\end{align*}
\end{proposition}

\begin{proof}
For any fixed $x>0$, let
\begin{align*}
P_{d}=\{\tau\in[0,x)\,:\,\text{an interaction occurs at } x=\tau\},\qquad
P_{c}=\{\tau\in[0,x)\,:\,\tau\notin P_{d}\}.
\end{align*}
Then $P_d$ is finite, which is written as
\begin{align*}
	P_{d}=\{x_i, i\in\mathbb{N}\,:\,0=x_0<x_1<\cdots<x_n\leq x<x_{n+1}\}.
\end{align*}
We may assume that, for each $x=x_i$,
there is an interaction or a nonphysical wave $\alpha(i)$ cross the boundary,
or there is a physical wave $\alpha(i)$ hitting the boundary,
or the variance of the approximate pressure on the inner boundary is equal to $|\alpha(i)|$.
If there is no nonphysical wave or only a nonphysical wave crossing the boundary,
we set $\alpha(i)=0$.
Then Propositions \ref{prop:GeneralInteractionpoint}--\ref{prop:InnerInteractionpoint} implies
	\begin{align*}
		&\mathfrak{F}(U_1(x),U_2(x))-\mathfrak{F}(U_1(0),U_2(0))\\
		&=\int_{x_n}^x\frac{\dd}{\dd s}\mathfrak{F}(U_1(s),U_2(s))\,\dd s
		+\mathfrak{F}(U_1(x_n),U_2(x_n))-\mathfrak{F}(U_1(0),U_2(0))\\
		&\leq O(1)(x-x_{n})\mu+O(1)\|p_{b,2}^{\mu,\Delta x_2}-p_{b,1}^{\mu,\Delta x_1}\|_{ L^{1}((x_{n},x))}
		+\exp(O(1)|\alpha_{n}|)\mathfrak{F}(U_1(x_n-),U_2(x_n-))\\
		&\quad\,\,-\mathfrak{F}(U_1(0),U_2(0))\\
		&\leq O(1)\exp\big(O(1)\sum_{i=1}^{n}|\alpha(i)|\big)x\mu
		+ O(1)\exp\big(O(1)\sum_{i=1}^{n}|\alpha(i)|\big)\|p_{b,2}^{\mu,\Delta x_2}-p_{b,1}^{\mu,\Delta x_1}\|_{ L^{1}(0,x)}\\
		&\quad\,\,+\Big(\exp\big(O(1)\sum_{i=1}^{n}|\alpha(i)|\big)-1\Big)\mathfrak{F}(U_1(0),U_2(0)).
	\end{align*}
Since Propositions \ref{prop:EoNP}--\ref{Prop:Important2} give an upper bound of $\sum_{i=1}^{n}|\alpha(i)|$, which is independent of our approximate solutions, we obtain
\begin{align*}
\mathfrak{F}(U_1(x),U_2(x))
\leq O(1)x\mu+O(1)\|p_{b,2}^{\mu,\Delta x_2}-p_{b,1}^{\mu,\Delta x_2}\|_{ L^{1}(0,x)}
+O(1)\,\mathfrak{F}(U_1(0),U_2(0)).
\end{align*}
Finally, the construction of our Lyapunov functional leads to
\begin{align*}
\int_0^x|\dot{b}_1(t)-\dot{b}_2(t)|\,\mathrm{d}t
+\int_{-\infty}^{b_{\max}(x)}|U_1(x,y)-U_2(x,y)|\,\dd y
\leq O(1)\,\mathfrak{F}(U_1(x,\cdot),U_2(x,\cdot)).
\end{align*}
This completes the proof.
\end{proof}

\begin{corollary}
For any $x>0$,
\begin{align*}
&|b_1(x)-b_2(x)|+\int_{-\infty}^{b_{\max}(x)}|U_1(x,y)-U_2(x,y)|\,\dd y\\
&\,\,\leq
 O(1)x\mu+O(1) \|p_{b,2}^{\mu,\Delta x_2}-p_{b,1}^{\mu,\Delta x_1}\|_{ L^{1}(0,x)}
 +O(1)\,\mathfrak{F}(U_1(x,\cdot),U_2(x,\cdot)).
\end{align*}
\end{corollary}

\section{Existence of the Semigroup}\label{sect-semi}
Combining all the analysis in \S \ref{sect-appro}--\S\ref{sect-Lya},
we now establish the existence of the semigroup that generates the solution of the inverse problem.

First we introduce the following definitions:

	\begin{itemize}
	\item [\rm(i)] For $f\in L^{1}_{\rm loc}(x,\infty)$, define
	\begin{equation}\label{5.1a}
		\iota_x:\quad L^{1}_{\rm loc}(x,\infty)\rightarrow L^{1}_{\rm loc}(\mathbb{R}_+),\\
		(\iota_x f)(\theta):= f(\theta+x);
	  \end{equation}
\item [\rm(ii)] for $(b_j,\check{U}_j^\top,p_j)^\top \in \mathbb{R}\times L^{1}_{\rm{loc}}(\mathbb{R}_-)\times L^{1}_{\rm{loc}}(\mathbb{R}_+)$, $j=1,\,2$, define
\begin{equation}\label{5.2a}
	\big\|(b_1,\check{U}_1^\top,p_1)^\top-(b_2,\check{U}_2^\top,p_2)^\top\big\|_{Y}
	:=|b_1-b_2|
	+\|\check{U}_1-\check{U}_2\|_{L^1(\mathbb{R}_-)}
	+\|p_1-p_2\|_{L^{1}(\mathbb{R}_+)}.
	\end{equation}
\end{itemize}

\smallskip
\begin{definition}
	\label{def:Domain}Given $\varepsilon>0$, define
	\begin{equation}
		\mathbb{D}^{\varepsilon}=\text{\rm cl}\left\{(b,\check{U}^\top,p)^\top\,:\,b\in\mathbb{R},\
		\begin{aligned}
			&\check{U}\in \textbf{PWC},& &\check{U}-\bar{U}_\flat\in L^1(\mathbb{R}_-;\mathbb{R}^4),\\
			&p\in \textbf{PWC},&&p-\bar{p}_{b}\in L^1(\mathbb{R}_+;\mathbb{R}),
		\end{aligned}
		\,\,F(0;b,\check{U}^\top,p)\le \varepsilon\right\},
		\nonumber
	\end{equation}
where $\textbf{PWC}$ stands for the piecewise constant functions $($vectors$)$, $\bar{U}_\flat(y)$ satisfies
	\begin{equation}
		\bar{U}_\flat(y)=\left\{
		\begin{aligned}
			&U_{-}&&\mbox{for $y<\chi_\flat$},\\
			&U_{+}&&\mbox{for $\chi_\flat<y<0$},
		\end{aligned}\right.
		\nonumber
	\end{equation}
	for some $\chi_\flat\le0$, $F(0;b,\check{U}^\top,p)$ is the Glimm-type functional corresponding to the initial data $\check{U}$
and the pressure distribution $p_b$ (see {\rm Definition \ref{Def:GTF}}), and {\rm cl} represents the closure in $\|\cdot\|_{Y}$.
\end{definition}

\medskip
We also need the following lemma (see Lemma 2.3 in \cite{Bressan2000}).
\begin{lemma}
\label{lem:LipschitzContinuous}
If $U:\mathbb{R}\rightarrow\mathbb{R}^n$ has bounded total variation, then
\begin{align*}
	\int_{-\infty}^{\infty}|U(x+t)-U(x)|\,\mathrm{d}x\leq t\, \text{T.V.}(U)\qquad\,\,\mbox{for any $t>0$}.
\end{align*}	
\end{lemma}
We now establish the existence theorem of the semigroup that generates the solution of this inverse problem.

\begin{theorem}\label{Thm:Important3}
Suppose that $\varepsilon>0$ is sufficiently small. Then, for any $(b(0),\check{U}^\top_{\infty},p_b)^\top\in\mathbb{D}^{\varepsilon}$,
corresponding to the initial data $U_{\infty}^\top(y)=\check{U}^\top_{\infty}(y-b(0))$ for $y<b(0)$ and the pressure distribution $p_b$,
there is a subsequence of $\mu$-approximate solutions $(b^{\mu,\Delta x},U^{\mu,\Delta x})$ converging to a unique solution $(b,U)$ as $\mu\rightarrow0$.
The map{\rm :}
	\begin{align*}
		(b(0),\check{U}^\top_{\infty},p_b,x)^\top\mapsto (b(x),U^\top(x,\cdot+b(x)),\iota_xp_b)^\top:=\mathfrak{S}_{x}(b(0),\check{U}^\top_{\infty},p_b)^\top
	\end{align*}
	is a semigroup that generates the solution of the inverse problem so that, for any
	\begin{align*}
		(b(0),\check{U}_{\infty}^\top,p_{b})^\top,\,(b_i(0),\check{U}_{\infty,i}^\top,p_{b,i})^\top\in \mathbb{D}^{\varepsilon},
	\end{align*}
	and $x_i\ge 0$, $i=1,\,2$,
	\begin{align*}
		&\mathfrak{S}_{0}(b(0),\check{U}_{\infty}^\top,p_{b})^\top=(b(0),\check{U}_{\infty}^\top,p_{b})^\top,\\
  &\mathfrak{S}_{x_1}\mathfrak{S}_{x_2}(b(0),\check{U}_\infty^\top,p_b)^\top=\mathfrak{S}_{x_1+x_2}(b(0),\check{U}_\infty^\top,p_b)^\top.
		\end{align*}
		Moreover, there are constants $L^\sharp>0$ and $L^\flat>0$ such that
		\begin{align*}
			&\big\|\mathfrak{S}_{x_1}(b_1(0),\check{U}_{\infty,1}^\top(\cdot),p_{b,1})^\top-\mathfrak{S}_{x_2}(b_2(0),\check{U}_{\infty,2}^\top(\cdot),p_{b,2})^\top\big\|_{Y}\\
				&\,\leq L^\sharp\,\big\|(b_1(0),\check{U}_{\infty,1}^\top(\cdot),p_{b,1})^\top-(b_2(0),\check{U}_{\infty,2}^\top(\cdot),p_{b,2})^\top\big\|_{Y}+L^\flat\,|x_1-x_2|.
				\end{align*}
	\end{theorem}

\begin{proof}
For any $\mu_1$, $\mu_2>0$, let $U^{\mu_1}$ and $U^{\mu_2}$ be the $\mu_j$-approximate solutions
of (\ref{eqn:TwoDFEuler}) and (\ref{eqn:ConseForm})--(\ref{1.9a}), whose initial data are $U_{\infty}^{\mu_1}$ and $U_{\infty}^{\mu_2}$, and
pressure distributions are $p_{b}^{\mu_{1}}$ and $p_{b}^{\mu_2}$, respectively.
Fixing $x>0$, by Propositions \ref{prop:GeneralInteractionpoint}--\ref{prop:InnerInteractionpoint} and Lemma \ref{lem:LipschitzContinuous}, we have
\begin{align*}
	&\big\|(b^{\mu_1}(x),(U^{\mu_1}(x,\cdot+b^{\mu_1}(x)))^\top,\iota_xp_{b}^{\mu_1})^\top
    -(b^{\mu_2}(x),(U^{\mu_2}(x,\cdot+b^{\mu_2}(x)))^\top,\iota_xp_{b}^{\mu_2})^\top\big\|_{Y}\\
	&\leq  O(1)\mathfrak{F}(U^{\mu_1}(x),U^{\mu_2}(x))
	+\|p_{b}^{\mu_1}-p_{b}^{\mu_2}\|_{ L^{1}(x,\infty)}\\
	&\leq  O(1)\mathfrak{F}(U^{\mu_1}(0),U^{\mu_2}(0))+O(1)\|p_{b}^{\mu_1}-p_{b}^{\mu_2}\|_{ L^{1}(0,x)}
	+O(1)\max\{\mu_1,\mu_2\}x\\
	&\quad \, +\|p_{b}^{\mu_1}-p_{b}^{\mu_2}\|_{ L^{1}(x,\infty)}\\
	&\leq C\big\|(b^{\mu_1}(0),(\check{U}_{\infty}^{\mu_1})^\top(\cdot),p_{b}^{\mu_1})^\top-(b^{\mu_2}(0),(\check{U}_{\infty}^{\mu_2})^\top(\cdot),p_{b}^{\mu_2})^\top\big\|_{Y}+C\max\{\mu_1,\mu_2\}x.
\end{align*}
Thus, as $\mu_1$, $\mu_2\rightarrow0$, $\|(b^{\mu_1}(x),(U^{\mu_1}(x,\cdot+b^{\mu_1}(x)))^\top,p_{b}^{\mu_1})^\top-(b^{\mu_2}(x),(U^{\mu_2}(x,\cdot+b^{\mu_2}(x)))^\top,p_{b}^{\mu_2})^\top\|_{Y}$ tends to zero, which implies that the sequence is a Cauchy sequence converging to a unique limit, saying $(b(x),U^\top(x,\cdot+b(x)),\iota_xp_b)^\top$. Then the semigroup properties follow from the uniqueness.\par
Finally, for $\mu>0$, let $U_1^{\mu}$ and $U_2^{\mu}$ be $\mu$-approximate solutions to (\ref{eqn:TwoDFEuler}) and (\ref{eqn:ConseForm})--(\ref{1.9a}) satisfying
\begin{align*}
	&\|U_1^{\mu}(0,\cdot+b_1(0))-\check{U}_{\infty,1}(\cdot)\|_{ L^1(\mathbb{R}_-)}<\mu,&&\|U_2^{\mu}(0,\cdot+b_2(0))-\check{U}_{\infty,2}(\cdot)\|_{ L^1(\mathbb{R}_-)}<\mu,\\
	&\|p_{b,1}^{\mu}(\cdot)-p_{b,1}(\cdot)\|_{ L^1(\mathbb{R}_+)}<\mu,&&\|p_{b,2}^{\mu}(\cdot)-p_{b,2}(\cdot)\|_{ L^1(\mathbb{R}_+)}<\mu.
\end{align*}
Then
\begin{align*}
	&\big\|(b_1^{\mu}(x),(U_1^{\mu}(x,\cdot+b_1^\mu(x)))^\top,\iota_xp_{b,1}^{\mu})
      -(b_2^{\mu}(x),(U_2^{\mu}(x,\cdot+b_2^\mu(x)))^\top,\iota_xp_{b,2}^{\mu})\big\|_{Y}\\
	&\leq O(1)\,\mathfrak{F}(U_1^\mu(x),U_2^\mu(x))
	+\|p_{b,1}^{\mu}-p_{b,2}^{\mu}\|_{ L^{1}(x,\infty)}\\
	&\leq O(1)\,\mathfrak{F}(U_1^\mu(0),U_2^\mu(0))+O(1)\|p_{b,1}^{\mu}-p_{b,2}^{\mu}\|_{ L^{1}(0,x)}
	+O(1)\mu x
	+\|p_{b,1}^{\mu}-p_{b,2}^{\mu}\|_{ L^{1}(x,\infty)}\\
	&\leq  C\|(b_1^\mu(0),(\check{U}_{\infty,1}^{\mu})^\top(\cdot),p_{b,1}^{\mu})^\top-(b_2^\mu(0),(\check{U}_{\infty,2}^{\mu})^\top(\cdot),p_{b,2}^{\mu})^\top\|_{Y}+C\mu x.
\end{align*}
Taking $\mu\rightarrow 0$, we obtain
\begin{align*}
	&\|(b_1(x),U_1^\top(x,\cdot+b_1(x)),\iota_xp_{b,1})^\top-(b_2(x),U_2^\top(x,\cdot+b_2(x)),\iota_xp_{b,2})^\top\|_{Y}\\
	&\leq L^\sharp\,\|(b_1(0),U_{\infty,1}^\top(\cdot),p_{b,1})^\top-(b_2(0),U_{\infty,2}^\top(\cdot),p_{b,2})^\top\|_{Y},
\end{align*}
for some $L^\sharp>0$, which gives the Lipschitz continuity. Moreover, from Proposition \ref{prop:Important2}, (\ref{eqn:Lipcon}), and the fact that
\begin{align*}
	\|\iota_xp_{b,1}-\iota_xp_{b,2}\|_{ L^1(\mathbb{R}_+)}\leq\|p_{b,1}-p_{b,2}\|_{ L^1(\mathbb{R}_+)},
\end{align*}
we conclude the Lipschitz continuity on $x$.
\end{proof}
Main Theorem I is a direct corollary of Theorem \ref{thm:EaS} and Theorem \ref{Thm:Important3}.

\section{Approximate the Full Euler Equations by the Potential Flow Equations}\label{sect-compare}
This section focuses on the comparison between the two solutions of the inverse problem,
which are obtained by solving the full Euler equations and the potential flow equations, respectively.
If there is no strong shock, at {\it time} $x$, we show that the difference of two solutions
in the norm $\|\cdot\|_{Y}$ is up to the third order
of the total variation of the initial boundary data, multiplying $x$.

\subsection{Existence and stability of the potential flow equations}
Regarding $x$ as {\it time},
when $u^2+v^2<2B_{\infty}$ and $u>c_{*}=\sqrt{2B_{\infty}(\gamma-1)/(\gamma+1)}$,
system (\ref{eqn:TwoDPotential})--(\ref{eqn:Bernouli}) is strictly hyperbolic,
whose eigenvalues are
\begin{equation}
	\lambda_j(\textbf{u})=\frac{uv+(-1)^{j}c\sqrt{u^2+v^2-c^2}}{u^2-c^2},
	\qquad j=1,\,2,\label{eqn:PFEigenvalue}
\end{equation}
with corresponding eigenvectors (see \cite{Zhang1999})
\begin{equation}
	r_j(\textbf{u})=\frac{(-\lambda_j,1)^\top}{(-\lambda_j(\textbf{u}),1)\cdot\nabla\lambda_j(\textbf{u})},
	\qquad j=1,\,2.\label{eqn:PFEigenvector}
\end{equation}
Using the same method as we have developed in the previous
sections, we have

\begin{theorem}\label{thm:WeishiliuCunzai}
Let \eqref{JS:1}--\eqref{JS:2} hold. Then there is $\check\varepsilon>0$ such that,
 when $\|(\widetilde{\textbf{u}}_{\infty},\tilde{p}_{b})\|_{ L^{\infty}\cap BV}<\check\varepsilon$, there exists a subsequence $\{\mu_l\}_{l=1}^\infty$ and $\{\Delta x_l\}_{l=1}^\infty$ so that
	\begin{itemize}
		\item[\rm(i)] $b^{\mu_l,\Delta x_l}$ converges uniformly to $b$ on any compact subset contained in the $x$-axis{\rm ;}
		\item[\rm(ii)] $\textbf{u}^{\mu_l, \Delta x_l}_{b}$ converges to $\textbf{u}_{b}$ in $\text{BV}([0,\infty))$,  and $\dot{b}(x)=\frac{v_b(x)}{u_b(x)}$\, a.e.{\rm ;}
		\item[\rm(iii)] for any $x>0$, $\textbf{u}^{\mu_l,\Delta x_l}$converges to $\textbf{u}$ in $ L^1_{\text{\rm loc}}\big((-\infty,b(x));\mathbb{R}^2\big)$, and $\textbf{u}$
is an entropy solution of equations \eqref{eqn:TwoDPotential}--\eqref{eqn:Bernouli} satisfying \eqref{eqn:WSLBoundaryConstruction}--\eqref{eqn:WSLPFBoundaryData}.
	\end{itemize}
\end{theorem}

\begin{definition}
	Let $(b_j,\check{\textbf{u}}_j^\top,p_j)^\top\in \mathbb{R}\times L^{1}_{\rm loc}(\mathbb{R}_-)\times L^{1}_{\rm loc}(\mathbb{R}_+)$, $j=1,\,2$. Define
	\begin{align*}
		\big\|(b_1,\check{\textbf{u}}_1^\top,p_1)^\top-(b_2,\check{\textbf{u}}_2^\top,p_2)^\top\big\|_{Y_{\text{P}}}
        =|b_1-b_2|+\|\check{\textbf{u}}_1-\check{\textbf{u}}_2\|_{L^1(\mathbb{R}_-)}+\|p_1-p_2\|_{ L^{1}(\mathbb{R}_+)}.
	\end{align*}
	\end{definition}

	\begin{definition}
		\label{def:WSDomain}
		Given $\hat{\varepsilon}>0$, define
		\begin{equation*}
			\mathbb{D}^{\hat{\varepsilon}}_{\text{P}}=\text{\rm cl}\left\{(b,\check{\textbf{u}}^\top,p)^\top\,:\, \, b\in\mathbb{R},\,\,\,
			\begin{aligned}
				&\check{\textbf{u}}\in \textbf{PWC},&&\check{\textbf{u}}-\overline{\textbf{u}}_0\in L^1(\mathbb{R}_-;\mathbb{R}^2), \\
				&p\in \textbf{PWC},&&p-\bar{p}_{b}\in L^1(\mathbb{R}_+;\mathbb{R}),
			\end{aligned}
			\,\,\, F_{\rm P}(0;b,\check{\textbf{u}}^\top,p)\le \hat{\varepsilon}
			\right\},
		\end{equation*}
		where $F_{\rm P}(0;b,\check{\textbf{u}}^\top,p)$ is the Glimm-type functional corresponding to the initial data $\check{\textbf{u}}$
and the pressure distribution $p_b$ for the potential flow equations, and {\rm cl} stands for the closure with respect to $\|\cdot\|_{Y_{\text{P}}}$.
		\end{definition}

Furthermore, we have
\begin{theorem} \label{Thm:WSLWendingxing}
	For $\hat{\varepsilon}>0$ sufficiently small, take $(b(0),\check{\textbf{u}}_{\infty}^\top,p_b)^\top\in\mathbb{D}^{\hat{\varepsilon}}_{\text{P}}$.
Then, with the initial value $\textbf{u}_{\infty}^\top(y)=\check{\textbf{u}}_{\infty}^\top(y-b(0))$ for $y<b(0)$ and the pressure distribution $p_b$,
a subsequence of $\mu$-approximate solutions $(b^{\mu,\Delta x},\textbf{u}^{\mu,\Delta x})$ converges to a unique limit $(b,\textbf{u})$, as $\mu\rightarrow0$. As a result,
\begin{equation*}
(b(0),\check{\textbf{u}}_{\infty}^\top,p_b,x)^\top\mapsto (b(x),\textbf{u}^\top(x,\cdot+b(x)),\iota_xp_b)^\top
:=\mathfrak{S}^{\text{P}}_{x}(b(0),\check{\textbf{u}}_{\infty}^\top,p_b)^\top
\end{equation*}
	is a semigroup that generates the solution to the inverse problem for the potential flow system{\rm :} If
	\begin{align*}
		(b(0),\check{\textbf{u}}_{\infty}^\top,p_b)^\top,\,(b_i(0),\check{\textbf{u}}_{\infty,i}^\top,p_{b,i})^\top\in \mathbb{D}^{\hat{\varepsilon}}_{\text{P}},
	\end{align*}
	and $x_i\ge 0$, $i=1,\,2$, then
		\begin{align*}
	&\mathfrak{S}^{\text{P}}_{0}(b(0),\check{\textbf{u}}_\infty^\top(\cdot),p_b)^\top
	=(b(0),\check{\textbf{u}}_\infty^\top(\cdot),p_{b})^\top,\\
	&\mathfrak{S}^{\text{P}}_{x_1}\mathfrak{S}^{\text{P}}_{x_2}(b(0),\check{\textbf{u}}_\infty^\top(\cdot),p_b)^\top
	  =\mathfrak{S}^{\text{P}}_{x_1+x_2}(b(0),\textbf{u}_\infty^\top(\cdot),p_b)^\top.
	\end{align*}
	Moreover, there exist $K^\sharp>0$ and $K^\flat>0$ such that,
	\begin{equation*}
	\begin{aligned}
	&\|	\mathfrak{S}^{\text{P}}_{x_1}(b_1(0),\textbf{u}_{\infty,1}^\top(\cdot),p_{b,1})^\top
	   -\mathfrak{S}^{\text{P}}_{x_2}(b_2(0),\textbf{u}_{\infty,2}^\top(\cdot),p_{b,2})^\top\|_{Y_{\text{P}}}\\
	 \leq& K^\sharp\,\|(b_1(0),\textbf{u}_{\infty,1}^\top(\cdot),p_{b,1})^\top-(b_2(0),\textbf{u}_{\infty,2}^\top(\cdot),p_{b,2})^\top\|_{Y_{\text{P}}}+K^\flat\,|x_1-x_2|.
			\end{aligned}
		\end{equation*}
		\end{theorem}
\subsection{Comparison of the wave curves and Riemann-type problems}
It follows from (\ref{JS:1})--(\ref{JS:2}) and (\ref{JS:3})--(\ref{JS:4}) that
$\overline{\textbf{u}}_{\infty}=(\overline{u}_{\infty},0)^\top$ and
$\overline{U}_{\infty}=(\overline{u}_{\infty},0,(\mathscr{R}(|\overline{\textbf{u}}_\infty|))^{\gamma},\mathscr{R}(|\overline{\textbf{u}}_\infty|))^\top$.
According to \cite{Lax1957} (see also \cite{Chen2006,Smoller1983}), there exists $\delta_1>0$ such that,
and, in the neighborhood $O_{\delta_1}(\overline{U}_\infty)$ of
$\overline{U}_\infty=(\overline{\textbf{u}}_\infty^\top,(\mathscr{R}(|\overline{\textbf{u}}_\infty|))^\gamma,\mathscr{R}(|\overline{\textbf{u}}_\infty|))^\top$,
the $j$th wave curve of the full Euler equations through $U_l\in O_{\delta_1}(\overline{U}_\infty)$ is parameterized as
\begin{equation*}
\alpha_j\mapsto\Phi_{\text{E},j}(\alpha_j;U_l),\qquad j=1,\,2,\,3,\,4,
\end{equation*}
in the neighborhood $O_{\delta_1}(\overline{\textbf{u}}_\infty)$ of $\overline{\textbf{u}}_\infty$, the $j$th wave curve
of the potential flow equations through $\textbf{u}_l\in O_{\delta_1}(\overline{\textbf{u}}_\infty)$ is parameterized as
\begin{align*}
\alpha_j\mapsto\Phi_{\text{P},j}(\alpha_j;\textbf{u}_l),\qquad j=1,\,2,
\end{align*}
such that
\begin{align*}
 \frac{\dd\Phi_{\text{P},j}(\alpha_j;\textbf{u}_l)}{\dd \alpha_j}=r_j(\textbf{u}_l).
\end{align*}
We denote
\begin{align*}
&\Phi_{\text{E}}(\alpha_1,\alpha_2,\alpha_3,\alpha_4;U_l)=\Phi_{\text{E},4}(\alpha_4;\Phi_{\text{E},3}(\alpha_3;\Phi_{\text{E},2}(\alpha_2;\Phi_{\text{E},1}(\alpha_1;U_l)))),\\
&\Phi_{\text{P}}(\alpha_1,\alpha_2;\textbf{u}_l)=\Phi_{\text{P},2}(\alpha_2;\Phi_{\text{P},1}(\alpha_1;\textbf{u}_l)),
\end{align*}
and define
\begin{align*}
D_{\text{E}}&:=\Big\{(\textbf{u}^\top,p,\rho)^\top\in O_{\delta_1}(\overline{U}_\infty)\, :\, |\textbf{u}|>c,\, p=\rho^\gamma,\, \frac{|\textbf{u}|^2}{2}+\frac{\gamma p}{(\gamma-1)\rho}=B_{\infty}\Big\},\\
D_{\text{P}}&:=\Big\{\textbf{u}\in O_{\delta_1}(\overline{\textbf{u}}_\infty)\, :\, \frac{2(\gamma-1)B_{\infty}}{\gamma+1}<|\textbf{u}|^2<2B_{\infty} \Big\}.
\end{align*}
For $\textbf{u}_l\in D_{\text{P}}$, define
\begin{align*}
&\Psi_{j}(\alpha_j;\textbf{u}_l):=(\Phi_{\text{P},j}(\alpha_j;\textbf{u}_l)^\top,(\mathscr{R}(|\Phi_{\text{P},j}(\alpha_j;\textbf{u}_l)|))^\gamma,\mathscr{R}(|\Phi_{\text{P},j}(\alpha_j;\textbf{u}_l)|))^\top,\\
&\Psi(\alpha_1,\alpha_{2};\textbf{u}_l):=\Psi_{2}(\alpha_2;\Psi_{1}(\alpha_1;\textbf{u}_l)),
\end{align*}
where $\mathscr{R}(r)$ is given in (\ref{eqn:WSLBonuli}).
We now compare between the wave curves of the full Euler equations and the potential flow equations.
For any $U_l\in	D_{\text{E}}$, denote
\begin{equation*}
\Phi_{\text{E},j}(\alpha_j;U_l):=(u_{\text{E},j}(\alpha_j;U_l),v_{\text{E},j}(\alpha_j;U_l),p_{\text{E},j}(\alpha_j;U_l),\rho_{\text{E},j}(\alpha_j;U_l))^\top.
\end{equation*}

First, for $\Phi_{\text{E},j}$, $j=1,\,4$, we have the following property:
\begin{lemma}
\label{lem:WOBonuli}
For $\alpha_j\ge0$ for $j=1,\,4$,
\begin{equation}
	p_{\text{E},j}=(\rho_{\text{E},j})^\gamma,\qquad\frac{(u_{\text{E},j})^2+(v_{\text{E},j})^2}{2}+\frac{\gamma(\rho_{\text{E},j})^{\gamma-1}}{\gamma-1}=B_{\infty}.\label{eqn:GammaBernoulli}
\end{equation}
\end{lemma}
\begin{proof}
Let $\lambda_{\text{E},j}$ be the $j$th eigenvalue of the full Euler equations.
A direct computation leads to
\begin{align*}
\frac{\dd(p_{\text{E},j}-(\rho_{\text{E},j})^\gamma)}{\dd\alpha_j}
&=\mathfrak{k}_j\rho_{\text{E},j}(\lambda_{\text{E},j}u_{\text{E},j}-v_{\text{E},j})
-\gamma(\rho_{\text{E},j})^{\gamma-1}\mathfrak{k}_j\,\frac{(\rho_{\text{E},j})^2(\lambda_{\text{E},j}u_{\text{E},j}-v_{\text{E},j})}{\gamma p_{\text{E},j}}\\
&=\frac{\mathfrak{k}_j\rho_{\text{E},j}(\lambda_{\text{E},j}u_{\text{E},j}-v_{\text{E},j})}{\gamma p_{\text{E},j}}\big(p_{\text{E},j}-(\rho_{\text{E},j})^\gamma\big).
\end{align*}
Notice that the first part of (\ref{eqn:GammaBernoulli}) holds when $\alpha_j=0$ and $p_l-\rho_l^\gamma=0$.
Then, since
	\begin{align*}
		\frac{\dd}{\dd\alpha_j}\Big(\dfrac{(u_{\text{E},j})^2+(v_{\text{E},j})^2}{2}+\frac{\gamma p_{\text{E},j}}{(\gamma-1)\rho_{\text{E},j}}\Big)
		&=-u_{\text{E},j}\mathfrak{k}_j\lambda_{\text{E},j}+\mathfrak{k}_jv_{\text{E},j}+\frac{\gamma \mathfrak{k}_j\rho_{\text{E},j}(\lambda_{\text{E},j}u_{\text{E},j}-v_{\text{E},j})}{(\gamma-1)\rho_{\text{E},j}}\\
		&\quad\,-\frac{\gamma p_{\text{E},j}}{(\gamma-1)(\rho_{\text{E},j})^2}\,\frac{\rho_{\text{E},j}\mathfrak{k}_j(\lambda_{\text{E},j}u_{\text{E},j}-v_{\text{E},j})}{(c_{\text{E},j})^2}\\
		&=0,
	\end{align*}
and $p_{\text{E},j}=(\rho_{\text{E},j})^\gamma$, we obtain
	\begin{align*}
		\dfrac{(u_{\text{E},j})^2+(v_{\text{E},j})^2}{2}+\frac{\gamma (\rho_{\text{E},j})^{\gamma-1}}{\gamma-1}&=\dfrac{(u_{\text{E},j})^2+(v_{\text{E},j})^2}{2}+\frac{\gamma p_{\text{E},j}}{(\gamma-1)\rho_{\text{E},j}}\\
		&=\dfrac{u_l^2+v_l^2}{2}+\frac{\gamma p_l}{(\gamma-1)\rho_l}=B_{\infty}.
	\end{align*}
This completes the proof.
\end{proof}

Next, we have (see also \cite{Zhang2007Z})
\begin{lemma}\label{lem:coincide}
For $U_l=(\textbf{u}_l^\top,p_l,\rho_l)^\top\in D_{\text{E}}$ and $\textbf{u}_l\in D_{\text{P}}$, when $\alpha_j\ge0$,
then
	\begin{equation}
		\frac{\dd}{\dd\alpha_j}\left(u_{\text{E},j}(\alpha_j;U_l),v_{\text{E},j}(\alpha_j;U_l)\right)^\top
  =	\frac{\dd\Phi_{\text{P},\sqrt{j}}(\alpha_j;\textbf{u}_l)}{\dd\alpha_j},\qquad j=1,\,4.\label{eqn:coincide}
	\end{equation}
\end{lemma}
\begin{proof}
By (\ref{eqn:Coeffi}), it is direct to see
\begin{align*}
	&\frac{\dd}{\dd\alpha_j}\left(u_{\text{E},j}(\alpha_j;U_l),v_{\text{E},j}(\alpha_j;U_l)\right)^\top\\
 &=	\mathfrak{k}_j(-\lambda_{\text{E},j},1)^\top
 =\frac{2}{\gamma+1}\,\frac{(u_{\text{E},j}^2-c_{\text{E},j}^2)\lambda_{\text{E},j}-u_{\text{E},j}v_{\text{E},j}}{(1+\lambda_{\text{E},j}^2)(\lambda_{\text{E},j}u_{\text{E},j}-v_{\text{E},j})}(-\lambda_{\text{E},j},1)^\top.
\end{align*}
Form Lemma \ref{lem:WOBonuli}, (\ref{eqn:Eigenval}){\blue--(\ref{eqn:Coeffi})}, and (\ref{eqn:PFEigenvalue})--(\ref{eqn:PFEigenvector}), we obtain
\begin{align*}
	\lambda_{\text{E},j}(\Phi_{\text{E},j}(\alpha_j;U_l))=\lambda_{\text{P},\sqrt{j}}((u_{\text{E},j}(\alpha_j;U_l),v_{\text{E},j}(\alpha_j;U_l))^\top),
\end{align*}
where $\lambda_{\text{P},\sqrt{j}}(\textbf{u})$ is the $\sqrt{j}$th eigenvalue of the potential flow equations. By Lemma \ref{lem:WOBonuli}, we see that $c_{\text{E},j}^2=\gamma(\rho_{\text{E},j})^{\gamma-1}
$ so that
\begin{align*}
	&\frac{2}{\gamma+1}\,\frac{(u_{\text{E},j}^2-c_{\text{E},j}^2)\lambda_{\text{E},j}-u_{\text{E},j}v_{\text{E},j}}{(1+\lambda_{\text{E},j}^2)(\lambda_{\text{E},j}u_{\text{E},j}-v_{\text{E},j})}(-\lambda_{\text{E},j},1)^\top\\
	&=\frac{(-\lambda_{\text{P},\sqrt{j}}((u_{\text{E},j}(\alpha_j;U_l),v_{\text{E},j}(\alpha_j;U_l))^\top),1)^\top}{(-\lambda_{\text{P},\sqrt{j}}((u_{\text{E},j}(\alpha_j;U_l),v_{\text{E},j}(\alpha_j;U_l))^\top),1)\cdot\nabla\lambda_{\text{P},\sqrt{j}}((u_{\text{E},j}(\alpha_j;U_l),v_{\text{E},j}(\alpha_j;U_l))^\top)}.
\end{align*}
Note that
\begin{align*}
	\frac{\dd\Phi_{\text{P},\sqrt{j}}(\alpha_j;\textbf{u}_l)}{\dd\alpha_j}
	=\frac{(-\lambda_{\text{P},\sqrt{j}}(\Phi_{\text{P},\sqrt{j}}(\alpha_j;\textbf{u}_l)),1)^\top}{(-\lambda_{\text{P},\sqrt{j}}(\Phi_{\text{P},\sqrt{j}}(\alpha_j;\textbf{u}_l)),1)\cdot\nabla\lambda_{\text{P},\sqrt{j}}(\Phi_{\text{P},\sqrt{j}}(\alpha_j;\textbf{u}_l))},
\end{align*}
$(u_{\text{E},j}(0;U_l),v_{\text{E},j}(0;U_l))^\top=\textbf{u}_l=\Phi_{\text{P},\sqrt{j}}(0;\textbf{u}_l)$.
Therefore, we obtain (\ref{eqn:coincide}).
\end{proof}

By Lemma \ref{lem:WOBonuli}--\ref{lem:coincide}, noting that the wave curves $\Phi_{\text{E},j}$ and $\Psi_j$ are $C^2$ functions, we have (see also \cite{Zhang2007Z})

\begin{lemma}\label{lem:expcoincide}
For $U_l=(\textbf{u}_l^\top,p_l,\rho_l)^\top\in D_{\text{E}}$
and $\textbf{u}_l\in D_{\text{P}}$, if $\alpha_j\ge0$, then, for $j=1,\,4$,
\begin{align*}
&\Phi_{\text{E},j}(\alpha_j;U_l)=\Psi_{j}(\alpha_j;\textbf{u}_l),\\
&\frac{\partial\Psi_{j}(0;\textbf{u}_l)}{\partial\alpha_j}=\textbf{r}_j(U_l),
\qquad
\frac{\partial^2\Phi_{\text{E},j}(0;U_l)}{\partial\alpha_j^2}
=\frac{\partial^2\Psi_{j}(0;\textbf{u}_l)}{\partial\alpha_j^2}.
\end{align*}	
\end{lemma}

\begin{proposition}
\label{prop:zhongyao1}
Assume that $U_l=(\textbf{u}_l^\top,p_l,\rho_l)^\top\in D_{\text{E}}$ and $\textbf{u}_l\in D_{\text{P}}$.
For $\alpha_j$ sufficiently small, $j=1,\,4$, the equation{\rm :}
	\begin{align}\label{eqn:compare}
		\Phi_{\text{E}}(\beta_1,\beta_2,\beta_3,\beta_4;U_l)=(\Phi_{\text{P},\sqrt{j}}(\alpha_j;\textbf{u}_l),(\mathscr{R}(|\Phi_{\text{P},\sqrt{j}}(\alpha_j;\textbf{u}_l)|)^{\gamma}),\mathscr{R}(|\Phi_{\text{P},\sqrt{j}}(\alpha_j;\textbf{u}_l)|))^\top		
	\end{align}
	has a unique solution $(\beta_1,\beta_2,\beta_3,\beta_4)$ satisfying
 \begin{align*}
&\beta_{j}=\alpha_{j}+O(1)|\alpha_{j}^{-}|^{3},\\
&\beta_{k}=O(1)|\alpha_{j}^-|^{3},\qquad k\neq j,
\end{align*}
where $a^{-}=\min\{a,0\}$ and the bound of $O(1)$ is independent of $\alpha_{j}$ and $\textbf{u}_l$.
\end{proposition}

\begin{proof}
	Note that
	\begin{align*}
		\det\left.\Big( \frac{\partial \Phi_{\text{E}}
		}{\partial(\beta_{1},\beta_{2},\beta_{3},\beta_{4})} \Big)\right|_{\beta_{1}=\beta_{2}=\beta_{3}=\beta_{4}=0}=
		\det(\textbf{r}_{1},\textbf{r}_{2},\textbf{r}_{3},\textbf{r}_{4})\neq0.
	\end{align*}
Then the implicit function theorem implies that there exists a unique solution $(\beta_{1},\beta_{2},\beta_{3},\beta_{4})$, $\beta_{k}=\beta_{k}(\alpha_{j},\textbf{u}_l)\in C^{2}$, of equation (\ref{eqn:compare}) with $\beta_{k}|_{\alpha_{j}=0}=0$.

In order to obtain the expansion of $\beta_{k}$, we first differentiate equation (\ref{eqn:compare}) with respect to $\alpha_{j}$ and then let $\alpha_{j}=0$. By Lemma \ref{lem:expcoincide}, we obtain
	\begin{equation*}
		\sum_{k=1}^{4}\left.\frac{\partial \beta_{k}}{\partial \alpha_{j}}\right|_{\alpha_{j}=0}\textbf{r}_{k}(U_l)=\textbf{r}_j(U_l),\qquad j=1,\,4,
	\end{equation*}
which imply
	\begin{equation*}
		\left.\frac{\partial \beta_{k}}{\partial \alpha_{j}}\right|_{\alpha_{j}=0}=\delta_{jk}.\label{eqn:yijiezhankai}
	\end{equation*}
Next, we take the second-order derivatives in equation (\ref{eqn:compare}) with respect to $\alpha_j$
and then let $\alpha_{j}=0$. Therefore, from Lemma \ref{lem:expcoincide} and (\ref{eqn:yijiezhankai}), we have
	\begin{equation*}
		\sum_{k=1}^{4}\left.\frac{\partial^2 \beta_{k}}{\partial \alpha_{j}^{2}}\right|_{\alpha_{j}=0}\textbf{r}_{k}(U_l)+\frac{\partial^{2}\Phi_{\text{E},j}(0;U_l)}{\partial \beta_{j}^{2}}=\frac{\partial^{2}\Psi_{j}(0;\textbf{u}_l)}{\partial \alpha_{j}^{2}},\qquad j=1,\,4,
	\end{equation*}
	which, together with Lemma \ref{lem:expcoincide}, leads to
	\begin{equation}
		\left.\frac{\partial^2 \beta_{k}}{\partial \alpha_{j}^{2}}\right|_{\alpha_{j}=0}=0. \label{eqn:erjiezhankai}
	\end{equation}
Thus, when $a_{j}<0$, from (\ref{eqn:yijiezhankai}) and (\ref{eqn:erjiezhankai}), we obtain the result.
When $\alpha_j\geq0$, from Lemma \ref{lem:expcoincide} and the uniqueness of $(\beta_{1}, \beta_{2}, \beta_{3}, \beta_{4})$,
we see that $\beta_{j}=\delta_{kj}\alpha_{j}$. This completes the proof.
\end{proof}

\begin{proposition}
	\label{prop:zhongyao2}
	For $\textbf{u}_l,\textbf{u}_r\in D_{\text{P}}$ satisfying
	\begin{align*}
	&\textbf{u}_r=\Phi_{\text{P}}(\alpha_1,\alpha_{4};\textbf{u}_l),\\
 &U_r=\Phi_{\text{E}}(\beta_1,\beta_{2},\beta_{3},\beta_{4};U_l),
\end{align*}
with $U_r=(\textbf{u}_r^\top,(\mathscr{R}(|\textbf{u}_r|))^{\gamma},\mathscr{R}(|\textbf{u}_r|))^\top$
and $U_l=(\textbf{u}_l^\top,(\mathscr{R}(|\textbf{u}_l|))^{\gamma},\mathscr{R}(|\textbf{u}_l|))^\top$.
Then
\begin{align*}
&\beta_{j}=\alpha_{j}+O(1)(|\alpha_1^-|+|\alpha_4^-|)^3,\qquad\, j=1,\,4,\\		
&\beta_{k}=O(1)(|\alpha_1^-|+|\alpha_4^-|)^3,\qquad\, k=2,\,3,
\end{align*}
where $a^{-}=\min\{a,0\}$, and the bound of $O(1)$ is independent of  $\textbf{u}_l$ and $\alpha_{j}$, $j=1,2,3,4$.
\end{proposition}
\begin{proof}
It suffices to solve
\begin{equation*}	
\Phi_{\text{E}}(\beta_{1},\beta_{2},\beta_{3},\beta_{4};U_l)=(\Phi_{\text{P}}(\alpha_1,\alpha_4;\textbf{u}_l)^\top,\,(\mathscr{R}(|\Phi_{\text{P}}(\alpha_1,\alpha_4;\textbf{u}_l)|))^{\gamma},
\,\mathscr{R}(|\Phi_{\text{P}}(\alpha_1,\alpha_4;\textbf{u}_l)|))^\top
\end{equation*}
for $\beta_k$, $k=1,\,2,\,3,\,4$.	
Carrying out a similar procedure as in the proof of Proposition \ref{prop:zhongyao1},
there exist $C^{2}$ functions $\beta_{k}=\beta_{k}(\alpha_{1},\alpha_{4},\textbf{u}_l), k=1,\,2,\,3,\,4$,
such that $(\beta_{1}, \beta_{2}, \beta_{3}, \beta_{4})$ solves the above system.\par
	
As for estimates on $\beta_{k}$, we let
\begin{align*}
\textbf{u}_m=\Phi_{\text{P}}(\alpha_{1},0;\textbf{u}_l),\quad U_m=(\textbf{u}_m^\top,p_m,\rho_m)^\top =(\textbf{u}_m^\top,(\mathscr{R}(|\textbf{u}_m|))^{\gamma},\mathscr{R}(|\textbf{u}_m|)),
\end{align*}
and consider the equations:
\begin{align*}
&U_m=\Phi_{\text{E}}(\beta_{1}^{'},\beta_{2}^{'},\beta_{3}^{'},\beta_{4}^{'};U_l),\\
&U_r=\Phi_{\text{E}}(\beta_{1}^{''},\beta_{2}^{''},\beta_{3}^{''},\beta_{4}^{''};U_m).
\end{align*}
Thanks to Proposition \ref{prop:zhongyao1}, we have
\begin{equation}
\beta_{k}^{'}=\alpha_{1}\delta_{k1}+O(1)|\alpha_1^{-}|^{3},\qquad
\beta_{k}^{''}=\alpha_{4}\delta_{k4}+O(1)|\alpha_4^{-}|^{3},
\label{eqn:limanwentisanjie}
\end{equation}
where $\delta_{ij}$ is the Kronecker symbol.
Hence, by Glimm's interaction estimates (see Lemma \ref{lem:Iww}), (\ref{eqn:limanwentisanjie}) gives the corresponding estimates for $\beta_{k}$, $k=1,\,2,\,3,\,4$.
\end{proof}

\begin{proposition}
	\label{prop:zhongyao3}
	Suppose that $\textbf{u}_l\in D_{\text{P}}$ and that $p_r$ satisfies
	\begin{align*}
&p_r=(\mathscr{R}(|\Phi_{\text{P}}(\alpha_1,0;\textbf{u}_l)|))^\gamma,\\
&p_r=\Phi_{\text{E}}^{(3)}(\beta_1,0,0,0;U_l),
	\end{align*}
with $U_l=(\textbf{u}_l^\top,(\mathscr{R}(|\textbf{u}_l|))^{\gamma},\mathscr{R}(|\textbf{u}_l|))^\top$.
Then
\begin{equation*}
\beta_{1}=\alpha_{1}+O(1)|\alpha_{1}^{-}|^3=\alpha_{1}+O(1)|\omega^-|^3,
\end{equation*}
where $\omega=p_r-p_l$, $a^{-}=\min\{a,0\}$,
and the bound of $O(1)$ is independent of $\alpha_{1}$ and $\textbf{u}_l$.
\end{proposition}

\begin{proof}
When $\omega\geq0$, $\alpha_1$ and $\beta_1$ are rarefaction waves, according to Lemma \ref{lem:WOBonuli}, it suffices to consider the equations:
\begin{align}
&\frac{|\Phi_{\text{P}}(\alpha_1,0;\textbf{u}_l)|^2}{2}+\frac{\gamma (\omega+p_l)^{\frac{\gamma-1}{\gamma}}}{\gamma-1}=B_{\infty},\label{eqn:bianjieWS}\\
&\frac{|\Xi(\beta_{1},0,0,0;U_l)|^2}{2}+\frac{\gamma (\omega+p_l)^{\frac{\gamma-1}{\gamma}}}{\gamma-1}=B_{\infty}, \label{eqn:bianjieOL}
\end{align}
where $\Xi(\beta_{1},0,0,0;U_l)=(u_{\text{E}}(\beta_{1},0,0,0;U_l),v_{\text{E}}(\beta_{1},0,0,0;U_l))^\top$.
We first take the derivative with respect to $\omega$ and then let $\omega=0$ to obtain
\begin{align*}
&\textbf{u}_l\cdot \frac{\partial\Phi_{\text{P}}(0,0;\textbf{u}_l)}{\partial\alpha_1}\,\left.\frac{\mathrm{d} \alpha_{1}}{\mathrm{d} \omega}\right|_{\omega=0}+(p_l)^{-\frac{1}{\gamma}}=0,\\
&\textbf{u}_l\cdot \frac{\partial\Xi_{\text{E}}(0,0,0,0;U_l)}{\partial\beta_1}\,\left.\frac{\mathrm{d} \beta_{1}}{\mathrm{d} \omega}\right|_{\omega=0}+(p_l)^{-\frac{1}{\gamma}}=0.
\end{align*}
Thus, by Lemma \ref{lem:expcoincide}, we have
\begin{equation}	
	\left.\frac{\mathrm{d} \alpha_{1}}{\mathrm{d} \omega}\right|_{\omega=0}=\left.\frac{\mathrm{d} \beta_{1}}{\mathrm{d} \omega}\right|_{\omega=0}\neq0.\label{eqn:WSOLbianjie1}
\end{equation}
Next, we first take the second-order derivative with respect to $\omega$
in equations (\ref{eqn:bianjieWS}) and (\ref{eqn:bianjieOL}), respectively,
and then let $\omega=0$  to obtain
\begin{align*}
	&\left(
	\frac{\partial\Phi_{\text{P}}(0,0;\textbf{u}_l)}{\partial\alpha_1}
	\cdot \frac{\partial\Phi_{\text{P}}(0,0;\textbf{u}_l)}{\partial\alpha_1}
	+\textbf{u}_l \cdot \frac{\partial^{2}\Phi_{\text{P}}(0,0;\textbf{u}_l)}{\partial\alpha_1^{2}}
	\right)
	\,\left.\left(\frac{\mathrm{d} \alpha_{1}}{\mathrm{d} \omega}\right)^2\right|_{\omega=0}\\
	&\,\,\,+\textbf{u}_l\cdot \frac{\partial\Phi_{\text{P}}(0,0;\textbf{u}_l)}{\partial\alpha_1}\,\left.\frac{\mathrm{d}^{2} \alpha_{1}}{\mathrm{d} \omega^{2}}\right|_{\omega=0}
	-\frac{1}{\gamma}(p_l)^{-\frac{1+\gamma}{\gamma}}=0,\\[2mm]
	&\left(
	\frac{\partial\Xi(0,0,0,0;U_l)}{\partial\beta_1}
	\cdot \frac{\partial\Xi(0,0,0,0;U_l)}{\partial\beta_1}
	+U_l\cdot \frac{\partial^{2}\Xi(0,0,0,0;U_l)}{\partial\beta_1^{2}}
	\right)
	\left.\left(\frac{\mathrm{d} \beta_{1}}{\mathrm{d} \omega}\right)^2\right|_{\omega=0}\\
	&\,\,\,+\textbf{u}_l\cdot\frac{\partial\Xi(0,0,0,0;U_l)}{\partial\beta_1}
      \,\left.\frac{\mathrm{d}^{2} \beta_{1}}{\mathrm{d} \omega^{2}}\right|_{\omega=0}
	-\frac{1}{\gamma}(p_l)^{-\frac{1+\gamma}{\gamma}}=0.
\end{align*}
Combining Lemma \ref{lem:expcoincide} with (\ref{eqn:WSOLbianjie1}), we see that
\begin{equation}	
	\left.\frac{\mathrm{d}^{2} \alpha_{1}}{\mathrm{d} \omega^{2}}\right|_{\omega=0}=\left.\frac{\mathrm{d}^{2} \beta_{1}}{\mathrm{d} \omega^{2}}\right|_{\omega=0}.
	\label{eqn:WSOLbianjie2}
\end{equation}
Hence, when $\omega<0$, since $\alpha_{1}$, $\beta_{1}$, and $\omega$ are quantities of the same order,
we conclude the result;
when $\omega\geq0$, from the uniqueness of $\beta_1(\omega)$ and Lemma \ref{lem:expcoincide},
we also obtain the result.
\end{proof}

\subsection{The proof of Main Theorem II}
 To compare the $\mu$-approximate solutions,
 we need to establish the estimate involving different types of wavefronts.
 To this end, it suffices to consider the case that there is only one wavefront.
 Let $b$ and $p_b$ be constant functions, and let $U$ be a piecewise constant vector.
 For any $W(x_0,\cdot)\in\mathbb{D}^{\varepsilon}$ with $W(x_0,y)=(b(x_0),U(x_0,y+b(x_0)),\iota_{x_0}p_b)$, when $x>x_0$ is sufficiently close to $x_0$, $(\mathfrak{R}_{x-x_0}W)(\cdot+(\mathfrak{S}_{x-x_0}W)^{(1)})$ is an entropy solution that connects all the solutions of the Riemann-type problems
solved at which $U$ has a discontinuity, where
 \begin{align*}
	\mathfrak{R}_xW=((\mathfrak{S}_{x-x_0}W)^{(2)},(\mathfrak{S}_{x-x_0}W)^{(3)},(\mathfrak{S}_{x-x_0}W)^{(4)},(\mathfrak{S}_{x-x_0}W)^{(5)})^\top,
\end{align*}
and $(\mathfrak{S}_{x}W)^{(i)}$ stands for the $i$th components of $\mathfrak{S}_{x}W$ for $i=1,\cdots,6$.

\smallskip
\noindent
{\bf Case 1:}
{\it Suppose that $\textbf{u}_l, \textbf{u}_r\in D_{\text{P}}$ with
$\textbf{u}_l$ and $\textbf{u}_r$ sufficiently close to each other}.
Let $U_l=\Psi(0,0;\textbf{u}_l)$ and $U_r=\Psi(0,0;\textbf{u}_r)$.
For any $s\in\mathbb{R}$, $x_{0}\ge0$, and $y_{0}<b$, let
\begin{align*}
U(x,y)=\left\lbrace\begin{aligned}
	U_l \qquad \mbox{for $y-y_0<s(x-x_{0})$},\\
	U_r\qquad \mbox{for $y-y_0>s(x-x_{0})$}.
\end{aligned}\right.
\end{align*}
In what follows, we denote $U(\tau)=U|_{x=\tau}$ and $U(\tau)(y)=U(\tau,y)$. Then we assume that
\begin{align*}
	U(x_0)(y)=\left\{  \begin{aligned}
	U_r\qquad \mbox{for $y>y_0$},\\
	U_l\qquad \mbox{for $y<y_0$},
\end{aligned}\right.
\end{align*}
and $p_b|_{x=x_0}=p_r$.

\smallskip
\begin{proposition}\label{prop:neibu1}
Suppose that there exists $\alpha_j$ such that $\textbf{u}_r=\Phi_{\text{P},\sqrt{j}}(\alpha_j;\textbf{u}_l)$
for each $j=1,\,4$. Let $s$ be the speed of $\alpha_j$. Then, for $h>0$,
\begin{itemize}
	\item[\rm (i)] if $\alpha_j<0$,
 \begin{equation*}
	\|\mathfrak{S}_hW(x_0)-W(x_0+h)\|_{Y}\leq O(1)h|\alpha_j|^{3};
\end{equation*}
	\item[\rm (ii)] if $\alpha_j\ge0$,
	\begin{equation*}
		\|\mathfrak{S}_hW(x_0)-W(x_0+h)\|_{Y}\leq O(1)h|\alpha_j|^{2},
	\end{equation*}
\end{itemize}
where the bound of $O(1)$ is independent of $x_{0}$, $h$, and $\alpha_{j}$.
\end{proposition}

\begin{proof}
We only consider the case when $j=1$, since the proof for $j=4$ is similar.\par

\smallskip
\noindent
{\bf (a)}. $\alpha_{1}<0$. Then $\alpha_{1}$ is a shock with speed $s$. Let $(\beta_{1}, \beta_{2}, \beta_{3}, \beta_{4})$ be the solution to the equation:
\begin{align*}
	U_r=\Phi_{\text{E}}(\beta_{1}, \beta_{2}, \beta_{3}, \beta_{4};U_l).
\end{align*}
$\mathfrak{R}_hW(x_0)$ contains waves $\beta_{k}$, $k=1,\,2,\,3,\,4$. According to Proposition \ref{prop:zhongyao1},
\begin{align*}
	\beta_{k}=\alpha_{j}\delta_{1k}+O(1)|\alpha_1|^{3},\qquad k=1,\,2,\,3,\,4,
\end{align*}
which gives $\beta_1<0$.\par
In addition, since
\begin{align*}
&\sigma=\lambda_{\text{E},1}(U_l)+\frac{1}{2}\beta_{1}+O(1)|\beta_{1}|^{2}
  =\lambda_{\text{E},1}(U_l)+\frac{1}{2}\alpha_{1}+O(1)|\alpha_{1}|^{2},\\
&s=\lambda_{\text{P},1}(\textbf{u}_l)+\frac{1}{2}\alpha_{1}+O(1)|\alpha_1|^{2},
\end{align*}
we obtain
\begin{align*}
	\sigma=s+O(1)|\alpha_1|^{2}.
\end{align*}
Denote $q_1=\min\{s,\sigma\}$, $q_2=\max\{s,\sigma\}$, and
$\lambda_M=\max\{|\lambda_{\text{E},4}(U)|\,:\,U\in D_{\text{E}}\}+1$.
Therefore, $q_1$ and $q_2$ are bounded, and
\begin{align*}
	q_2-q_1=O(1)|\alpha_{1}|^{2}.
\end{align*}
Furthermore, for $\varpi(h,\cdot):=\mathfrak{R}_hW(x_0)$, we can obtain
\begin{itemize}
	\item[(i)] if $y<q_1h+y_0$, then $\varpi(h,y)=U(x_0+h,y)$.
	\item[(ii)] if $q_1h+y_0<y<q_2h+y_0$, then $\varpi(h,y)-U(x_0+h,y)=O(1)|\alpha_1|$.
	\item[(iii)] if $q_2h+y_0<y<\lambda_M h+y_0$, then $U(x_0+h,y)=U_r$, and
	\begin{align*}
		\varpi(h,y)-U(x_0+h,y)=\varpi(h,y)-U_{r}
  =O(1)(|\beta_2|+|\beta_{3}|+|\beta_{4}|)
  =O(1)|\alpha_{1}|^{3}.
	\end{align*}
	\item[(iv)] if $y>\lambda_M h+y_0$, then $\varpi(h,y)=U(x_0+h,y)=U_r$.
\end{itemize}
Now, the boundary corresponding to $\mathfrak{S}_hW(x_0)$ and the boundary corresponding to $W(x_0+h)$ coincide, and the pressure on the boundary is $p_r$. Then
\begin{align*}
	\|\mathfrak{S}_hW(x_0)-W(x_0+h)\|_{Y}&=\int_{q_{2}h+y_{0}}^{\lambda_Mh+y_{0}}|\varpi(h,y)-U(x_0+h,y)|\,\mathrm{d}y\\
	&\quad +\int_{q_{1}h+y_{0}}^{q_{2}h+y_{0}}|\varpi(h,y)-U(x_0+h,y)|\,\mathrm{d}y\\
	&\leq O(1)|\alpha_{1}|^{3}|\lambda_{M}-q_2|h+O(1)|\alpha_{1}||q_2-q_{1}|h\\
	&\leq O(1)|\alpha_{1}|^3h.
\end{align*}

\noindent
{\bf (b)}. $\alpha_{1}\ge0$. Then $\alpha_{1}$ is a rarefaction wavefront with characteristic speed $s$. Let $(\beta_{1}, \beta_{2}, \beta_{3}, \beta_{4})$ be the solution of the equation:
\begin{align*}
	U_r=\Phi_{\text{E}}(\beta_{1}, \beta_{2}, \beta_{3}, \beta_{4};U_l).
\end{align*}
By Proposition \ref{prop:zhongyao1}, we have
\begin{align*}
	\beta_{1}=\alpha_{1},\qquad \beta_{2}=\beta_{3}=0,
\end{align*}
which means that $\mathfrak{R}_hW(x_0)$ contains a fan-shaped wave $\beta_{1}$.
In addition, at $x_{0}+h$, the width of the fan-shaped wave $\beta_{1}$ has an upper bound  $O(1)h|\alpha_{1}|$.
Therefore, we obtain
\begin{align*}
	\|\mathfrak{S}_hW(x_0)-W(x_0+h)\|_{Y}\leq O(1)h|\alpha_j|^{2},
\end{align*}
which gives the result.
\end{proof}

\begin{proposition}
\label{prop:kao}
If $U(x,y)$ contains nonphysical waves only, then
\begin{equation*}
	\|\mathfrak{S}_hW(x_0)-W(x_0+h)\|_{Y}\leq O(1)h|\textbf{u}_r-\textbf{u}_l|,
\end{equation*}
where $O(1)$ has an upper bound independent of $x_{0}$, $h$, and the strength of the wave.
\end{proposition}

\begin{proof}
	From the construction of the $\mu$-approximate solutions, we have
	\begin{align*}
		\big(\mathfrak{R}_hW(x_0)\big)(y)=U(x_0+h,y+(\mathfrak{S}_hW(x_0))^{(1)})=U_r \qquad \mbox{for $y>y_0+sh-(\mathfrak{S}_hW(x_0))^{(1)}$}.
	\end{align*}
Thus, by Proposition \ref{prop:zhongyao2}, carrying out similar arguments
as in the proof of Proposition \ref{prop:neibu1} leads to the result.
\end{proof}

\smallskip
\noindent{\bf Case 2:}
{\it Suppose that $\textbf{u}_l$, $\textbf{u}_r\in D_{\text{P}}$ with $\textbf{u}_l$ and $\textbf{u}_r$ sufficiently close
each other}.
Let $U_l=\Psi(0,0;\textbf{u}_l)$ and $U_r=\Psi(0,0;\textbf{u}_r)$.
For any $s<\frac{v_r}{u_r}$ and $x_{0}\ge0$, set
\begin{align*}
U(x,y)=\left\lbrace\begin{aligned}
	U_l\qquad \mbox{for $y-b<s(x-x_{0})$},\\
	U_r\qquad \mbox{for $y-b>s(x-x_{0})$}.
\end{aligned}  \right.
\end{align*}
Moreover, we assume that
\begin{align*}
U(x_0,y)=U_l\qquad \mbox{for $y<b$},
\end{align*}
and $p_b|_{x=x_0}=p_r$.

\smallskip
\begin{proposition}\label{prop:bianjie}
Suppose that there exists $\alpha_1$ such that $\textbf{u}_r=\Phi_{\text{P},1}(\alpha_1;\textbf{u}_l)$
and that the speed of $\alpha_1$ is $s$. Then, for $h>0$,
\begin{itemize}
\item[\rm (i)] if $\alpha_1<0$, then
		\begin{equation*}
			\|\mathfrak{S}_hW(x_0)-W(x_0+h)\|_{Y}\leq O(1)h|\alpha_1|^{3};
		\end{equation*}
\item[\rm (ii)] if $\alpha_1\ge0$, then
		\begin{equation*}
			\|\mathfrak{S}_hW(x_0)-W(x_0+h)\|_{Y}\leq O(1)h|\alpha_1|^{2},
		\end{equation*}
	\end{itemize}
	where the bound of $O(1)$ is independent of $x_{0}$, $h$, and $\alpha_{j}$.
\end{proposition}

\begin{proof} We divide the proof into two steps.

\smallskip
{\bf 1}. Case: $\alpha_{1}<0$.
Then $\alpha_{1}$ is a shock whose speed is denoted by $s$.
Let $\beta_{1}$ be the solution of
\begin{align*}
p_r=\Phi_{\text{E},1}^{(3)}(\beta_{1};U_l).
\end{align*}
Using Proposition \ref{prop:zhongyao3}, we obtain
	\begin{align*}
		\beta_{1}=\alpha_{1}+O(1)|\alpha_{1}^{-}|^3,
	\end{align*}
	which implies $\beta_1<0$.\par
For the estimates of $\sigma$ and $\beta_{1}$, since
\begin{align*}
&\sigma=\lambda_{\text{E},1}(U_l)+\frac{1}{2}\beta_{1}+O(1)|\beta_{1}|^{2}
=\lambda_{\text{E},1}(U_l)+\frac{1}{2}\alpha_{1}+O(1)|\alpha_{1}|^{2},\\
&s=\lambda_{\text{P},1}(\textbf{u}_l)+\frac{1}{2}\alpha_{1}+O(1)|\alpha_1|^{2},
\end{align*}
we obtain
\begin{align*}
\sigma=s+O(1)|\alpha_1|^{2}.
\end{align*}
Let $q_1=\min\{s,\sigma\}$, $q_2=\max\{s,\sigma\}$,
and $\lambda_M=\max\{|\lambda_{\text{E},4}(U)|\,:\,U\in D_{\text{E}}\}+1$.
Then $q_1$ and $q_2$ are bounded and satisfy
\begin{align*}
q_2-q_1=O(1)|\alpha_{1}|^{2}.
\end{align*}
Furthermore, denote $\varpi(h,\cdot):=\mathfrak{R}_hW(x_0)$.
Using Lemma \ref{lem:expcoincide} and Proposition \ref{prop:zhongyao3} yields
	\begin{itemize}
		\item[(i)] if $y<q_1h+y_0$, then, $\varpi(h,y)=U(x_0+h,y)$;
		\item[(ii)] if $q_1h+y_0<y<q_2h+y_0$, then, $\varpi(h,y)-U(x_0+h,y)=O(1)|\alpha_1|$;
		\item[(iii)] if $q_2h+y_0<y<b+\frac{v_r}{u_r}h$, then, $U(x_0+h,y)=U_r$,
\begin{align*}
&\varpi(h,y)-U(x_0+h,y)=\varpi(h,y)-U_{r}
           =O(1)|\alpha_{1}|^{3},\\
&\bigg|\frac{\Phi_{\text{E},1}^{(2)}(\beta_{1};U_l)}{\Phi_{\text{E},1}^{(1)}(\beta_{1};U_l)}-\frac{v_r}{u_r}\bigg|=O(1)|a_{1}|^{3}.
\end{align*}
\end{itemize}
Set $\dot{b}_1:=\min\{\frac{\Phi_{\text{E},1}^{(2)}(\beta_{1};U_l)}{\Phi_{\text{E},1}^{(1)}(\beta_{1};U_l)},\frac{v_r}{u_r}\}$ and
$\dot{b}_2:=\max\{\frac{\Phi_{\text{E},1}^{(2)}(\beta_{1};U_l)}{\Phi_{\text{E},1}^{(1)}(\beta_{1};U_l)},\frac{v_r}{u_r}\}$.
Then, by Lemma \ref{lem:LipschitzContinuous}, similar to the proof of Proposition \ref{prop:neibu1}, we have
\begin{align*}
		&\|\mathfrak{S}_hW(x_0)-W(x_0+h)\|_{Y}\\
  &\le O(1)h(\dot{b}_2-\dot{b}_1)+O(1)|\alpha_{1}||q_2-q_{1}|h
   +\int_{q_{2}h+b}^{\dot{b}_1h+b}|\varpi(h,y)-U(x_0+h,y)|\,\mathrm{d}y\\
	&\le O(1)|\alpha_{1}|^3h.
	\end{align*}

 \smallskip
{\bf 2}. Case: $\alpha_{1}\ge0$.
Then $\alpha_{1}$ is a rarefaction wavefront whose speed is $s$.
Let $\beta_{1}$ solve the equation:
	\begin{align*}
		p_r=\Phi_{\text{E},1}^{(3)}(\beta_{1};U_l).
	\end{align*}
Using Proposition \ref{prop:zhongyao3}, we have
\begin{align*}
\beta_{1}=\alpha_{1},\qquad
\frac{\Phi_{\text{E},1}^{(2)}(\beta_{1};U_l)}{\Phi_{\text{E},1}^{(1)}(\beta_{1};U_l)}=\frac{v_r}{u_r}.
\end{align*}
That is, $\mathfrak{R}_hW(x_0)$ contains a fan-shaped wave $\beta_{1}$. Furthermore, at $x_{0}+h$, the width of the fan-shaped wave $\beta_{1}$ has an upper bound $O(1)h|\alpha_{1}|$. Thus,
	\begin{align*}
		\|\mathfrak{S}_hW(x_0)-W(x_0+h)\|_{Y}\leq O(1)h|\alpha_j|^{2}.
	\end{align*}
This completes the proof.
\end{proof}

With those preparations, we are about to prove the following estimates for $\mathfrak{S}$.

\begin{proposition}\label{prop:daoshudier}
Suppose that $\textbf{u}^{\mu}$ is a $\mu$-approximate solution to the potential flow equations \eqref{eqn:TwoDPotential}--\eqref{eqn:Bernouli} satisfying
\eqref{eqn:WSLBoundaryConstruction}--\eqref{eqn:WSLPFBoundaryData},
and $b^\mu_{\text{P}}$ is an approximate boundary.
Let $W^{\mu}_{\text{P}}(x)=(b^\mu_{\text{P}}(x),\Psi(0,0;\textbf{u}^{\mu}(x,\cdot))^\top,\iota_xp_b)$.
Then, at $x>0$,
\begin{equation*}
\liminf_{h\to0+}\frac{1}{h}\left\|\mathfrak{S}_{h}W^{\mu}_{\text{P}}(x)-W^{\mu}_{\text{P}}(x+h)\right\|_{Y}\leq O(1)\left(\text{T.V.}^{-}(\textbf{u}^{\mu}(x,\cdot))\right)^3+O(1)\mu,\label{eqn:zhongyaoyao}
\end{equation*}
where $\text{T.V.}^{-}(\textbf{u}^{\mu}(x,\cdot))$ is the total strength of the shock waves, and the bound of $O(1)$ is independent of $x$ and $\mu$.
\end{proposition}

\begin{proof}
Note that $\textbf{u}^{\mu}$ is a piecewise constant vector, and the number of all wavefronts is finite.
From Propositions \ref{prop:zhongyao2}--\ref{prop:zhongyao3}
and Propositions \ref{prop:neibu1}--\ref{prop:bianjie},
when $h$ is small enough, we have
\begin{align*}
&\frac{1}{h}\left\|\mathfrak{S}_{h}W^{\mu}_{\text{P}}(x)-W^{\mu}_{\text{P}}(x+h)\right\|_{Y}\\
&\leq O(1)\sum_{\alpha\text{ is a shock at }x} |\alpha(x)|^{3}
    +O(1)\sum_{\beta\text{ is a rarefaction wave at }x} |\beta(x)|^{2}\\
&\quad\,\,   +O(1)\sum_{\epsilon(x)\text{ is a nonphysical wave at }x} |\epsilon(x)|\\
&\leq O(1)\Big(\big(\text{T.V.}^{-}(\textbf{u}^{\mu}(x,\cdot)) \big)^3+\text{T.V.}^{+}(\textbf{u}^{\mu}(x,\cdot))\mu+\mu\Big),
\end{align*}
where $\text{T.V.}^{+}(\textbf{u}^{\mu}(x,\cdot))$ is the total variation of the rarefaction wavefronts
in $\textbf{u}^{\mu}(x,\cdot)$, and $|\alpha(x)|$,  $|\beta(x)|$, and $|\epsilon(x)|$
are the strengths of waves $\alpha$, $\beta$, $\epsilon$ at $x$, respectively.
This completes the proof.
\end{proof}

\begin{definition}	\label{def:wokao}
For an interval $I$,
	\begin{align*}
	\text{\rm Lip}_\varepsilon(I;Y)=\{W\, :\, W(t)\in\mathbb{D}^{\varepsilon},\ t\in I,\ \sup_{t\neq\tau}\frac{\|W(t)-W(\tau)\|_{Y}}{|t-\tau|}<\infty\}.
	\end{align*}
	\end{definition}
To obtain a uniform estimate for the $\mu$-approximate solution,
we need the following proposition ({\it cf.} \cite[Lemma 6.2]{Colombo2004}  and \cite[Theorem 2.9]{Bressan2000}):

\begin{proposition}\label{prop:daoshudiyi}
For $T^*>0$, assume that $w:[0,T^*]\to \mathbb{D}^{\varepsilon}$ is a map such that $w-(0,\overline{U}_\infty^\top,\bar{p}_b)^\top\in \text{\rm Lip}_\varepsilon([0,T^*];Y)$.
Then there exists $L>0$, independent of $\varepsilon$, such that
	\begin{align*}
		\|\mathfrak{S}_{T}w(0)-w(T)\|_{Y}\leq L\int_{0}^T\liminf_{h\to0+}\frac{\|w(t+h)-\mathfrak{S}_hw(t)\|_{Y}}{h}\,\mathrm{d}t \qquad \mbox{for any $T\in[0,T^*]$}.
	\end{align*}
\end{proposition}

By the construction of approximate solutions, we know that
$$
W^{\mu}_{\text{P}}-(0,\overline{U}_\infty^\top,\bar{p}_b)^\top\in \text{\rm Lip}_\varepsilon([0,\infty);Y),
\qquad\,\, W^{\mu}_{\text{P}}(x)\in\mathbb{D}^{\varepsilon} \quad\mbox{for $x\ge0$}.
$$
Therefore, we obtain

\begin{theorem}\label{thm:cankaolast}
For any $T>0$, $W^{\mu}_{\text{P}}(x)\in\mathbb{D}^{\varepsilon}$ satisfies
\begin{equation}\label{eqn:0}
	\|\mathfrak{S}_{T}W^{\mu}_{\text{P}}(0)-W^{\mu}_{\text{P}}(T)\|_{Y}
	\leq O(1)T\Big(\big(\text{T.V.}^{-}(\textbf{u}^{\mu}(0,\cdot))+\text{T.V.}(p_b^{\mu}(\cdot))+|p_b(0+)-p^{\mu}(0,0-)|  \big)^3+\mu\Big),
\end{equation}
where the bound of $O(1)$ is independent of $T$ and $\mu$.
\end{theorem}

\begin{proof}
From (\ref{eqn:zhongyaoyao}) and Proposition \ref{prop:daoshudier}--\ref{prop:daoshudiyi}, we have
\begin{equation}
\|\mathfrak{S}_{T}W^{\mu}_{\text{P}}(0)-W^{\mu}_{\text{P}}(T)\|_{Y}
\leq O(1)\int_{0}^{T}\Big(\big(\text{T.V.}^{-}(\textbf{u}^{\mu}(x,\cdot)) \big)^3+\mu\Big)\mathrm{d}x.\label{eqn:dao2}		
\end{equation}
Note that
\begin{equation}\label{eqn:dao1}
		\text{T.V.}^{-}(\textbf{u}^{\mu}(x,\cdot))
  \leq O(1)\big(\text{T.V.}(\textbf{u}^{\mu}(0,\cdot))+\text{T.V.}(p_b^{\mu}(\cdot))+|p_b(0+)-p^{\mu}(0,0-)|  +\mu\big),
	\end{equation}
	and the result follows from (\ref{eqn:dao2})--(\ref{eqn:dao1}).
\end{proof}

Finally, we have the following theorem.
\begin{theorem}	\label{thm:rangshijiechongmanai}
Assume that \eqref{JS:1}--\eqref{JS:2} and
\eqref{JS:3}--\eqref{JS:4} hold
and that $U_1=(\textbf{u}_1^\top,p_1,\rho_1)^\top$ is an entropy solution of \eqref{eqn:ConseForm} satisfying \eqref{eqn:Cauchyda}--\eqref{1.9a}.
Let $U_2=(\textbf{u}_2^\top,(\mathscr{R}(|\textbf{u}_2|))^{\gamma},\mathscr{R}(|\textbf{u}_2|))^\top$
be an entropy solution of \eqref{eqn:Bernouli}--\eqref{eqn:TwoDPotential} satisfying \eqref{eqn:WSLBoundaryConstruction}--\eqref{eqn:WSLPFBoundaryData}.
Moreover, let $b_1$ and $b_2$ be corresponding boundaries.
Then there exist $\varepsilon_{c}>0$ and $C>0$ such that,
when $\|(\widetilde{\textbf{u}}_{\infty},\tilde{p}_{b})\|_{ L^{\infty}\cap BV}<\varepsilon_{c}$,
\begin{equation*}
\|(b_1(x),U_1(x,\cdot),\iota_xp_b)-(b_2(x),U_2(x,\cdot),\iota_xp_b)\|_{Y}
 \leq Cx\,\|(\widetilde{\textbf{u}}_{\infty}, \tilde{p}_{b})\|_{ L^{\infty}\cap BV}^3
 \qquad\,\, \mbox{for any $x>0$}.
\end{equation*} 	
\end{theorem}
\begin{proof}
Using $U_{1}(x,\cdot +b_1(x))=\mathfrak{R}_{x}U_\infty$ and the properties of $\mu$-approximate solutions,
and then letting $\mu\to0+$ in (\ref{eqn:0}), we obtain the result
by carrying out a similar argument as in \cite{Bressan1995,Bressan2000,Bressan1999,Zhang2007Z}.
\end{proof}

\appendix
\section{Proofs of Proposition \ref{prop:EoNP} and Theorem \ref{thm:EaS}}
In this appendix, we give the proofs of Proposition \ref{prop:EoNP} and Theorem \ref{thm:EaS}, respectively.

\subsection{Proof of Proposition \ref{prop:EoNP}}\label{appen:A1}
We first show that the magnitude of each non-physical wave satisfies
\begin{align*}
	\epsilon=O(1)\nu,\qquad \epsilon\in\mathcal{NP}.
\end{align*}
Indeed, a new non-physical wave $\epsilon$ comes out at $x=t$,
when a weak physical wave $\alpha$ interact with another physical wave $\beta$ with $|\alpha||\beta|<\nu$,
or a weak 1-wave $\alpha$ collides the strong shock with $|\alpha|<\nu$.
In both cases, the interaction estimates show that $\epsilon=O(1)\nu$.
Consider the quantity
\begin{align*}
L_{\epsilon}(x):=\sum_{\alpha\in\mathcal{A}(\epsilon)}|\alpha|,
\end{align*}
where $\mathcal{A}(\epsilon)$ is the set of wavefronts which approaches $\epsilon$.
Suppose that the interaction occurs at $x=t$.

\smallskip
\noindent
 \textbf{Case 1}: The interaction does not involve $\epsilon$. From the interaction estimates, we conclude
	\begin{equation}
		\Delta \epsilon(t)=0,\quad\,\, \Delta L_{\epsilon}(t)+\mathcal{K}\Delta Q(t)\leq0. \label{eqn:non-physicalEsti1}
	\end{equation}
\textbf{Case 2}: The non-physical wave $\epsilon$ collides another weak wave $\alpha$.
Again the interaction estimates imply
\begin{equation}
\Delta L_{\epsilon}(t)=-|\alpha|,\quad\,\, \Delta Q(t)<0,\quad\,\,
     \Delta\epsilon(t)\leq O(1)|\epsilon(t-)||\alpha|.  \label{eqn:non-physicalEsti2}
\end{equation}

 By (\ref{eqn:non-physicalEsti1})--(\ref{eqn:non-physicalEsti2}), the map:
	\begin{align*}
		x\mapsto\epsilon(x)\,\exp\{C(L_{\epsilon}(x)+\mathcal{K}Q(x))\}
	\end{align*}
	is non-increasing in $x$, where $M_{1}>|O(1)|$ is a constant.
Therefore, for $x>x_0$, we have
	\begin{align}
			\epsilon(x)&\leq\epsilon(x_0)\,\exp\{M_{1}(L_{\epsilon}(x_0)+\mathcal{K}Q(x_0))\}\nonumber\\
			&\leq O(1)\nu\,\exp\{M_{1}(L_{\epsilon}(0)+\mathcal{K}Q(0))\}\nonumber\\
			&\leq O(1)e^{M_{1}\varepsilon}\nu\nonumber\\
			&\leq M_{2}\nu,
       \label{eqn:non-physicalEsti3}
	\end{align}
	where $\varepsilon>0$ is the constant given in Proposition \ref{prop:befIm}.

 Next, we assign each wavefront in $U^{\mu,\Delta x}$ an integer number,
 counting the number of interactions that occurred to give birth to such a front.
 To be specific, we define the generation order $Od(\alpha)$ of a front $\alpha$ inductively as follows:

 \smallskip
 \indent$\bullet$ Fronts generated from $x=0$ and on the boundary (on points $(h\Delta x, b^{\mu,\Delta x}(h\Delta x))$
 for $h\in\mathbb{N}_+$) have generation order $m=1$.\\
	\indent$\bullet$ The boundary and the strong shock are always attached with generation order $m=1$.\\
	\indent$\bullet$ When two incoming fronts (including the boundary and the strong shock) of the families $i_1, i_2\in\{1,\cdots,5\}$
and of generation orders $m_1$ and $m_2$ interact. Then the outgoing fronts have a generation order
	\begin{equation}
		m=\left\{\begin{aligned}
			&m_1, &&i=i_1,\, i_1\neq i_2,\\
			&m_2, &&i=i_2,\, i_1\neq i_2,\\
			&\max\{m_1, m_2\}+1,&&i\neq i_1,\, i\neq i_2,\\
			&\min\{m_1, m_2\},&&i=i_1=i_2,
		\end{aligned}
		\right.
		\label{eqn:GeneraltionOder}
	\end{equation}
	where $i\in\{1,\cdots,5\}$ indicates the $i$-th family of the outing coming wave-fronts.\par
	For $m\geq1$, let
	\begin{align*}
		L_m(t)=\sum_{\substack{\alpha\text{ cross $x=t$,}\\Od(\alpha)\geq m}}b_{\alpha},
	\end{align*}
	and let
	\begin{align*}
		L_{i,m}(t)=\sum_{\substack{\alpha\text{ is of family $i$ }\\ \text{and cross }x=t,\\Od(\alpha)\geq m}}b_{\alpha}\qquad\,\,\mbox{for $i=1,\,2,\,3,\,4$}.
	\end{align*}
	Moreover, we write $L_0(t)=L(t)$. For $m\geq1$, define
	\begin{align*}
		&Q_m(t)=K_s\sum_{\substack{\alpha\in\mathcal{A}_{s},\\Od(\alpha)\geq m }}|b_{\alpha}|
		+\sum_{	\substack{\beta\in\mathcal{A}_{b},\\Od(\beta)\geq m }}|b_{\beta}|+K\sum_{\substack{(\alpha,\beta)\in\mathcal{A},\\\max\{Od(\alpha),\,Od(\beta)\}\geq m }}|b_{\alpha}||b_{\beta}|,
	\end{align*}
and $Q_0(t)=Q(t)$.
In addition, when $m\geq1$, let $I_m$ be the set of $x$-coordinates at which two waves of order $m_\alpha$
and $m_{\beta}$ interact, with $\max\{m_{\alpha}, m_{\beta}\}=m$,
and let $I_0=\{h\Delta x\,:\, h\in\mathbb{N}\}$.
Similar to the proof in Proposition \ref{prop:befIm}, tracking the generation order of the wavefronts
in a subsequent of interactions, we obtain
	\begin{equation}
		\begin{aligned}
			&\Delta L_m(x)=0,\quad &&  x\in I_{0}\cup\cdots\cup I_{m-2},\ m\geq2,\\
			&\Delta L_m(x)+\mathcal{K}\Delta Q_{m-1}(x)\leq 0,\,&& x\in I_{m-1}\cup I_{m}\cup I_{m+1}\cdots,\ m\geq1,\\
			&\Delta Q_{m}(x)+\mathcal{K}C\Delta Q(x)L_m(x-)\leq 0,\,&& x\in I_{0}\cup\cdots\cup I_{m-2}, \ m\geq2,\\
			&\Delta Q_{m}(x)+\mathcal{K}C\Delta Q_{m-1}(x)L(x-)\leq 0,\,&& x\in I_{m-1},\ m\geq1,\\
			&\Delta Q_{m}(x)\leq 0,\,&& x\in I_{m}\cup I_{m+1}\cdots,\ m\geq0.
		\end{aligned}
		\label{eqn:DecreaseGlimFun}
	\end{equation}
	Estimates (\ref{eqn:DecreaseGlimFun}) imply
	\begin{equation}
		\begin{aligned}
			L_m(x)&\leq \mathcal{K}\sum_{0<t\leq x}{\left(\Delta Q_{m-1}(t)\right)}^-,\\
			Q_{m}(x)&\leq\sum_{0<t\leq x}\left(\Delta Q_{m}(t)\right)^+\\
			&\leq \mathcal{K}C\sum_{0<t\leq x}\left(\Delta Q(t)\right)^-\,\sup_{x}L_m(x)+\mathcal{K}C\sum_{0<t\leq x}(\Delta Q_{m-1}(t))^-\,\sup_{x}L(x),
		\end{aligned}
		\label{eqn:EstimateofLQ1}
	\end{equation}
	both of which are valid for every $x>0$ and $m\geq1$. Furthermore, we have
	\begin{align*}
		0\leq Q_{m}(x)=\sum_{0<t\leq x}\big\{\left(\Delta Q_{m}(t)\right)^+-\left(\Delta Q_{m}(t)\right)^-\big\}.
	\end{align*}
	Since
 $F(x):=L(x)+\mathcal{K}Q(x)+|U_{\diamond}(x)-U_{\infty}^{-}|+|U^{\diamond}(x)-U_{\infty}^{+}|$ is non-increasing, we have
	\begin{equation}
		L_m(x)\leq L(x)\leq F(x)\leq F(0),\qquad
		\sum_{0<t< \infty}\left(\Delta Q(t)\right)^-\leq Q(0)\leq F(0).\label{eqn:EstimateofLQ2}
	\end{equation}
	
	Recalling
	\begin{align*}
		\widetilde{Q}_m:=
  \sum_{x>0}\left(\Delta Q_m(x)\right)^+,\qquad
		\widetilde{L}_m:=
  \sup_{x>0}L_m(x),
	\end{align*}
from (\ref{eqn:GeneraltionOder})--(\ref{eqn:EstimateofLQ2}),
we deduce the sequence of the inequalities (valid for $m\geq1$):
	\begin{equation}
		\left\{\begin{aligned}
			\widetilde{L}_m&\leq \mathcal{K}\widetilde{Q}_{m-1},\\
			\widetilde{Q}_m&\leq  C\mathcal{K}F(0)\widetilde{L}_m+ C\mathcal{K}\widetilde{Q}_{m-1}F(0)\leq C(\mathcal{K}^2+\mathcal{K})F(0)\widetilde{Q}_{m-1}.
		\end{aligned}
		\right.\label{eqn:EstimateofLQfinal}
	\end{equation}
For $F(0)$ sufficiently small, we obtain
	\begin{equation}
		\eta:=
  C(\mathcal{K}^2+\mathcal{K})F(0)<1.\label{eqn:SmallInitialGlimFun}
	\end{equation}
In this case, for every $x>0$ and $m\geq1$, by induction,
(\ref{eqn:EstimateofLQfinal})--(\ref{eqn:SmallInitialGlimFun}) yield
	\begin{equation}
		Q_m(x)\leq\widetilde{Q}_m\leq\varepsilon\eta^m,\qquad
		L_m(x)\leq\widetilde{L}_m\leq \mathcal{K}\varepsilon\eta^{m-1}. \label{eqn:FinalBoundedEsti}
	\end{equation}
Meanwhile, the number of wavefronts of the $m$-th generation can be counted as follows:
Since the wavefronts of generation $1$ are generated at $x=0$, as well as from the change of the pressure distribution,
the number of the first-order fronts is less that $N_{\mu}$.
From each interaction between the fronts of first-order,
recalling that each of the rarefaction fronts has size $<\delta$,
the second-order front is generated and its number is less than $O(1)\delta^{-1}$.
Therefore, the number of second-order wavefronts is $O(1)(N_{\mu})^2\,\delta^{-1}$.
Inductively, it is clear that the number of fronts of order $\leq m$
is bounded by some polynomial function of $(N_{\mu}, \delta^{-1})$, say
	\begin{equation}
	 P_{m}(N_{\mu},\delta^{-1}). \label{eqn:FinalConofEsti}
	\end{equation}
The particular form of $P_m$ is not interested here.

Next, we establish the total strength estimates of non-physical waves.
We track of the fronts of generation order $>m$ and $\leq m$, separately.
Using (\ref{eqn:non-physicalEsti3}) and (\ref{eqn:FinalBoundedEsti})--(\ref{eqn:FinalConofEsti}),
we obtain
		\begin{align}
			&\sum_{\epsilon\in\mathcal{NP},Od(\epsilon)>m}\epsilon+\sum_{\epsilon\in\mathcal{NP},Od(\epsilon)\leq m}\epsilon\nonumber\\
			&\,\leq [\text{total strength of all fronts of order}>m]\nonumber\\
			&\,\,\quad +[\text{number of all fronts of order}\leq m]\times [\text{maximum strength of non physical fronts}]\nonumber\\
			&\,\leq \mathcal{K}\varepsilon\eta^{m}+M_2\nu P_{m}(N_{\mu},\delta^{-1}).\label{eqn:FiEstiNonPhys}
		\end{align}
For any $\mu>0$, recalling $\eta<1$, take $m$ large enough such that $\mathcal{K}\varepsilon\eta^{m}<\frac{\mu}{2}$.
 Then we choose $\nu$ small enough so that $M_2\nu P_{m}(N_{\mu},\delta^{-1})<\mu$.
 For all $x>0$, we conclude
	\begin{align*}
		&\sum_{\epsilon\in\mathcal{NP},Od(\epsilon)>m}\epsilon+\sum_{\epsilon\in\mathcal{NP},Od(\epsilon)\leq m}\epsilon<\mu,
	\end{align*}
	This completes the proof. {\blue \hfill{$\square$}}

\subsection{Proof of Theorem \ref{thm:EaS}}\label{appen:A2}
The proof is completed as follows:
Results	(i)--(iv) are deduced easily from Propositions \ref{prop:Important2}--\ref{Prop:Important2},
together with the Arzel\`{a}-Ascoli Theorem.
	
We now prove (v). Firstly, from the previous analysis, we have
		\begin{align}
			&\int_{-\infty}^{0}|U^{\mu,\Delta x}(x_1,y+b^{\mu,\Delta x}(x))-U^{\mu,\Delta x}(x_2,y+b^{\mu,\Delta x}(x))|\,\dd y\nonumber\\[1mm]
			&\leq O(1)|x_1-x_2|\,[\text{maximum speed}]\times [\text{total strength of wave fronts}]\nonumber\\
			&\leq L|x_1-x_2|,\label{eqn:Lipcon}
		\end{align}
with $L$ being a constant independent of $\mu$.
Then there exists a subsequence (still denoted) $\{U^{\mu,\Delta x}\}$
that converges to a limit function $U\in L_{\text{\rm loc}}^1\big((-\infty,b(x));\mathbb{R}^4\big)$,
guaranteed by the Helly Theorem.

To show that $U$ is a weak solution, it suffices to prove that,
for every $\phi\in C_{\rm c}^1(\mathbb{R}^2)$ and $\psi\in C_{\rm c}^1(\Omega)$,
	\begin{equation}
		\begin{aligned}
			&\int_0^{\infty}\int_{-\infty}^{b(x)}
			(\phi_x\rho u+\phi_y\rho v)\,\dd y\dd x
			+\int_{-\infty}^0\phi(0,y){u}_{\infty}{\rho}_{\infty}\,\dd y=0, \\[1mm]
			&\int_0^{\infty}\int_{-\infty}^{b(x)}
			\big(\psi_x(\rho u^2+p)+\psi_y\rho uv\big)\,\dd y\dd x
			+\int_{-\infty}^0\psi(0,y)\big({\rho}_{\infty}{u}_{\infty}^2+{p}_{\infty}\big)\,\dd y=0,\\[1mm]
			&\int_0^{\infty}\int_{-\infty}^{b(x)}
			\big(\psi_x\rho uv+\psi_y(\rho v^2+p)\big)\,\dd y\dd x
			+\int_{-\infty}^0\psi(0,y){\rho}_{\infty}{u}_{\infty}{v}_{\infty}\,\dd y=0,\\[1mm]
			&\int_0^{\infty}\int_{-\infty}^{b(x)}
			\Big(\psi_x\rho u\big(\dfrac{u^2+v^2}{2}+\frac{\gamma p}{(\gamma-1)\rho}\big)
                +\psi_y\rho v\big(\dfrac{u^2+v^2}{2}+\frac{\gamma p}{(\gamma-1)\rho}\big)\Big)\dd y\dd x\\[1mm]
			&\quad\, +\int_{-\infty}^0\psi(0,y)\big(\dfrac{{u}_{\infty}^2}{2}+\frac{\gamma {p}_{\infty}}{(\gamma-1){\rho}_{\infty}}\big)\,\dd y=0.
		\end{aligned}
		\label{eqn:WeakSoluDef}
	\end{equation}

We only give a proof of the first equality in (\ref{eqn:WeakSoluDef}),
since the remaining part can be obtained analogously.
	
Since $\phi$ is compactly supported, it is required to verify
	\begin{equation}
		\begin{aligned}
			\int_0^{X}\int_{-\infty}^{b(x)}
			(\phi_x\rho u+\phi_y\rho v)\,\dd y\dd x
			+\int_{-\infty}^0\phi(0,y)u_{\infty}\rho_{\infty}\,\dd y=0
		\end{aligned}
		\label{eqn:WeakSoluCheck}
	\end{equation}
	for some $X>0$. To calculate the term
	\begin{equation}
		\int_0^{X}\int_{-\infty}^{b^{\mu,\Delta x}(x)}
		\big(\phi_x\rho^{\mu,\Delta x} u^{\mu,\Delta x}+\phi_y\rho^{\mu,\Delta x} v^{\mu,\Delta x}\big)\,\dd y\dd x,\label{eqn:WeakSoluCheck1}
	\end{equation}
we fix $\mu$ and assume that, on any level set $x=t$, $y_\alpha(t)$ is a jump for
$\alpha\in\mathcal{S}\cup\mathcal{R}\cup\mathcal{NP}$.
Let
	\begin{align*}
		\Delta \rho^{\mu,\Delta x} u^{\mu,\Delta x}:=
  \rho^{\mu,\Delta x} u^{\mu,\Delta x}(t,y_{\alpha}+)-\rho^{\mu,\Delta x} u^{\mu,\Delta x}(t,y_{\alpha}-),\\
		\Delta \rho^{\mu,\Delta x} v^{\mu,\Delta x}:=
  \rho^{\mu,\Delta x} v^{\mu,\Delta x}(t,y_{\alpha}+)-\rho^{\mu,\Delta x} v^{\mu,\Delta x}(t,y_{\alpha}-).
	\end{align*}
Observe that the polygonal lines $y=y_{\alpha}(x)$ divide stripe $[0,X]\times\mathbb{R}$
into regions $D_j$ on which $U^{\mu,\Delta x}$ is constant. Define
	\begin{align*}
		\Phi_s:=
  (\phi\rho^{\mu,\Delta x} u^{\mu,\Delta x},\phi\rho^{\mu,\Delta x} v^{\mu,\Delta x}).
	\end{align*}
By the divergent theorem, (\ref{eqn:WeakSoluCheck1}) can be written as
	\begin{equation}
		\sum_j\iint_{D_j}\text{\rm div}\Phi_s(x,y)\,\dd y\dd x=\sum_j\int_{\partial D_j}\Phi_s\cdot\textbf{n}\,\dd s,\label{eqn:DivThm}
	\end{equation}
where $\partial D_j$ is the boundary and $\textbf{n}$ is the unit outer normal.
Since, on the polygonal line $y=y_\alpha(x)$, $\textbf{n}\dd s=\pm(\dot{y}_\alpha,-1)\dd x$,
while $\phi(x,y)=0$ on line $x=X$.
Therefore, (\ref{eqn:DivThm}) is computed by
		\begin{align}
			&\int_0^X\sum_\alpha
			\big(
			\dot{y}_\alpha(x)\,\Delta \rho^{\mu,\Delta x} u^{\mu,\Delta x}(x,y_{\alpha})
			-\Delta \rho^{\mu,\Delta x} v^{\mu,\Delta x}(x,y_{\alpha})
			\big)\,\phi(x,y_\alpha(x))\,\dd y\dd x\nonumber\\
			&+\int_0^X\big(\dot{b}^{\mu,\Delta x}(x)\rho^{\mu,\Delta x} u^{\mu,\Delta x}(x,b^{\mu,\Delta x}(x))
            -\rho^{\mu,\Delta x} v^{\mu,\Delta x}(x,b^{\mu,\Delta x}(x))\big)\,\phi(x,b^{\mu,\Delta x}(x))\,\dd x\nonumber\\
			&-\int_{-\infty}^{0} \rho_{\infty}^{\mu,\Delta x} u_{\infty}^{\mu,\Delta x}(y)\,\phi(0,y)\,\dd y.\label{eqn:WeakSoluCheck2}
		\end{align}
	If the discontinuity $\alpha$ is physical, then
	\begin{align*}
		\dot{y}_\alpha(x)\,\Delta \rho^{\mu,\Delta x} u^{\mu,\Delta x}(x,y_{\alpha})
		-\Delta \rho^{\mu,\Delta x} v^{\mu,\Delta x}(x,y_{\alpha})
		=O(1)\mu\,|\alpha|.
	\end{align*}
	On the other hand, if the wave at $y_\alpha$ is non-physical, then
	\begin{align*}
		\dot{y}_\alpha(x)\,\Delta \rho^{\mu,\Delta x} u^{\mu,\Delta x}(x,y_{\alpha})
		-\Delta \rho^{\mu,\Delta x} v^{\mu,\Delta x}(x,y_{\alpha})
		=O(1)\,|\alpha|.
	\end{align*}
	Moreover, on the approximate boundary, from the construction, we obtain
	\begin{align*}
		\dot{b}^{\mu,\Delta x}(x)\rho^{\mu,\Delta x} u^{\mu,\Delta x}(x,b^{\mu,\Delta x}(x))-\rho^{\mu,\Delta x} v^{\mu,\Delta x}(x,b^{\mu,\Delta x}(x))=0.
	\end{align*}
	Therefore, by Proposition \ref{prop:EoNP}, we have
		\begin{align}
			\limsup_{\mu\rightarrow0}&\Big|\sum_{\alpha\in\mathcal{S}\cup\mathcal{R}\cup\mathcal{NP}}\big(
			\dot{y}_\alpha(x)\,\Delta \rho^{\mu,\Delta x} u^{\mu,\Delta x}(x,y_{\alpha})
			-\Delta \rho^{\mu,\Delta x} v^{\mu,\Delta x}(x,y_{\alpha})
			\big)\,\phi(x,y_\alpha(x))\Big|\nonumber\\
			&\leq
			\max_{x,y}|\phi(x,y)|
			\,\limsup_{\mu\rightarrow0}
			\Big\{\sum_{\alpha\in\mathcal{S}\cup\mathcal{R}}O(1)\,\mu|\alpha|
			+\sum_{\mathcal{NP}}O(1)\,|\alpha| \Big\}\nonumber\\
			&=0.\label{eqn:WeakSoluChecklim}
		\end{align}
Moreover, from (i), (iii), and the $ L_{\text{\rm loc}}^1$ convergence of $\{U^{\mu,\Delta x}\}$,
the dominated convergence theorem implies
\begin{align}
&\limsup_{\substack{\Delta x\rightarrow0,\\\mu\rightarrow0 }}
\int_0^X\big(\dot{b}^{\mu,\Delta x}(x)\rho^{\mu,\Delta x} u^{\mu,\Delta x}(x,b^{\mu,\Delta x}(x))
-\rho^{\mu,\Delta x} v^{\mu,\Delta x}(x,b^{\mu,\Delta x}(x))\big)\,\phi(x,b^{\mu,\Delta x}(x))\dd x\nonumber\\
&=\int_0^X\big(\dot{b}(x)\rho u(x,b(x))-\rho v(x,b(x))\big)\,\phi(x,b(x))\dd x.\label{eqn:WeakSoluChecklim1}
\end{align}
		
Again, the Dominated Convergence Theorem yields
		\begin{align}
			&\limsup_{\substack{\Delta x\rightarrow0,\\\mu\rightarrow0 }}\int_0^X\int_{-\infty}^{b(x)}
  \big(\phi_x\, \rho^{\mu,\Delta x} u^{\mu,\Delta x}+\phi_y\, \rho^{\mu,\Delta x} v^{\mu,\Delta x}\big)\dd y\dd x\nonumber\\
			&=\int_0^\infty\int_{-\infty}^{b(x)}\left(\phi_x\rho u+\phi_y\rho v\right)\dd y\dd x.\label{eqn:WeakSoluChecklim2}
		\end{align}
		
	Finally, noting (\ref{eqn:WeakSoluCheck2})--(\ref{eqn:WeakSoluChecklim2}), we conclude
		\begin{align}
			&\int_0^\infty\int_{-\infty}^{b(x)}\big(\phi_x\rho u+\phi_y\rho v\big)\dd y\dd x+\int_{-\infty}^{0}\phi(0,y)u_{\infty}\rho_{\infty}\dd y\nonumber\\
			&=\limsup_{\substack{\Delta x\rightarrow0,\\\mu\rightarrow0 }}\int_0^X\int_{-\infty}^{b^{\mu,\Delta x}(x)}
\big(\phi_x\, \rho^{\mu,\Delta x} u^{\mu,\Delta x}+\phi_y\, \rho^{\mu,\Delta x} v^{\mu,\Delta x}\big)\dd y\dd x\nonumber\\
			&\quad +\limsup_{\Delta x\rightarrow0}\int_{-\infty}^{0}\phi(0,y)\, u_{\infty}^{\mu,\Delta x}\rho_{\infty}^{\mu,\Delta x}\,\dd y.\label{eqn:WeakSoluChecklimfin}
		\end{align}
	By (\ref{eqn:WeakSoluCheck2})--(\ref{eqn:WeakSoluChecklim}), we obtain (\ref{eqn:WeakSoluCheck}).

\bigskip
\medskip
\noindent{\bf Acknowledgements:}
This work was initiated when Yun Pu studied at the University of Oxford as a recognized DPhil student
through the Joint Training Ph.D. Program between the University of Oxford and Fudan University -- He would like to express his sincere thanks
to both the home and the host universities for providing him with such a great opportunity.
The research of Gui-Qiang G. Chen is partially supported
by the UK Engineering and Physical Sciences Research Council Awards
EP/L015811/1, EP/V008854/1, and EP/V051121/1.
The research of Yun Pu was partially supported by the Joint Training Ph.D. Program of China Scholarship Council, No. 202006100104.
The research of Yongqian Zhang is supported in part by the NSFC Project 12271507.

\bigskip
\noindent
{\bf Conflict of Interest}. The authors declare that they have no conflict of interest. The authors also
declare that this manuscript has not been previously published, and will not be submitted elsewhere
before your decision.

\medskip
\noindent
{\bf Data Availability}: Data sharing is not applicable to this article as no datasets were generated or
analyzed during the current study.

\end{document}